\documentclass{amsart}
\addtocontents{toc}{\setcounter{tocdepth}{1}} 

\usepackage{mathrsfs}
\usepackage{amssymb}
\usepackage[usenames,dvipsnames]{xcolor}
\usepackage{subcaption}
\usepackage{mathtools}
\usepackage{caption}
\usepackage{enumerate}
\usepackage{hyperref}
\usepackage[nameinlink, capitalize, noabbrev]{cleveref}
\usepackage{ stmaryrd }
\usepackage{tikz-cd}
\usetikzlibrary {arrows.meta,automata,positioning}
\usetikzlibrary{positioning,decorations.pathreplacing}
\usepackage{forest,adjustbox}
\usetikzlibrary{intersections}
\hypersetup{ colorlinks=true,linkcolor=teal,    citecolor=magenta, }

\def\set#1{{\def\st{\;:\;}\left\{#1\right\}}}
\def\abs#1{\left\vert{#1}\right\vert}
\def \<#1>{{\left\langle{#1}\right\rangle}}

\def\ZZ{\mathbb Z}
\def\QQ{\mathbb Q}

\def\NN{\mathbb N}
\def\RR{\mathbb R}
\def\Q-{\overline{\mathbb Q}}

\DeclareMathOperator{\Aut}{Aut}

\DeclareMathOperator{\Hom}{Hom}

\DeclareMathOperator{\IMG}{IMG}

\DeclareMathOperator{\St}{St}
\DeclareMathOperator{\RiSt}{RiSt}
\DeclareMathOperator{\rist}{rist}
\DeclareMathOperator{\Sym}{Sym}
\DeclareMathOperator{\id}{id}
\DeclareMathOperator{\st}{st}

\DeclareMathOperator{\lcm}{lcm}

\newtheorem{Theorem}{Theorem}
\newtheorem{Conjecture}[Theorem]{Conjecture}
\newtheorem{Corollary}[Theorem]{Corollary}
\newtheorem{Proposition}{Proposition}[section]
\newtheorem{Lemma}[Proposition]{Lemma}
\newtheorem{corollary}[Proposition]{Corollary}
\newtheorem{theorem}[Proposition]{Theorem}

\newtheorem{Definition}[Proposition]{Definition}
\newtheorem{Example}[Proposition]{Example}
\newtheorem{Remark}[Proposition]{Remark}

\title{Groups of finite type: classification and structural properties}
\author{Santiago Radi}
\address{Department of Mathematics, Texas A\&M University, 77843 College Station, U.S.A.
}
\email{santiradi@tamu.edu}
\thanks{The author is supported by Grigorchuk's Simons Foundation Grant MP-TSM-00002045 and the department of Mathematics of Texas A\&M University.}

\begin{document}

\maketitle

\begin{abstract}
Groups of finite type (also called finitely constrained groups), introduced by Grigorchuk, are known to be the closure of regular branch groups. This article explores many of their properties.

Firstly, we prove that being finitely generated, just-infinite and strongly complete are equivalent in a vast family of groups of finite type. As a consequence, we prove that the closure of the Hanoi towers group on 3 pegs is just-infinite although the group itself is not. 

Secondly, we improve the algorithm given by Bondarenko and Samoilovych in \cite{Bondarenko2014}, to compute all the groups of finite type of a given depth and acting on a given tree. We use this to find the groups of finite type acting on the ternary tree with depth 2 and 3.

Thirdly, we give a sufficient condition for a group generated by a finite automaton of Mealy type to have as closure a group of finite type. This allows us to identify groups of finite type as the closure of explicit groups generated by a finite automaton. 

Lastly, we give an algorithm to prove whether two groups of finite type are isomorphic. With this result, we classify groups of finite type up to isomorphism in the binary tree for depths 2, 3 and 4 and in the ternary tree for depths 2 and 3.
\end{abstract}

\tableofcontents

\section{Introduction}
\label{sec_Introduction}

A profinite group can be defined as a totally disconnected compact group or as the inverse limit of an inverse system of finite groups. Moreover, it was observed in \cite{GrigorchukNekrashevichSuschanski2000} that any countably based profinite group can be embedded as a closed subgroup of the group of automorphisms of a spherically homogeneous rooted tree. Three important properties of profinite groups are topologically finitely generation, just-infiniteness and strongly completeness. A profinite group is \textit{topologically finitely generated} if it contains a dense finitely generated subgroup. It is \textit{just-infinite} if every closed non-trivial normal subgroup is of finite index. First results about just-infinite groups were given in \cite{Hall1964, Mennicke1965} around 1965. In \cite{Grigorchuk2000}, Grigorchuk proved that just-infinite profinite groups split into two categories, being just-infinite branch groups one of the possibilities. In the same article, the author also gave an equivalent condition for a branch group to be just-infinite. In \cite{Wilson2000}, Wilson  extended  his  result  from \cite{Wilson1971} about just-infinite abstract groups to profinite groups, splitting them in those that have finite and infinite structure lattice. In fact, the second class of Wilson's just-infinite groups coincides with the class of just-infinite branch groups.

A profinite group $G$ is \textit{strongly complete} if one of the following equivalent conditions occur:
\begin{itemize}
\item Every subgroup of finite index in $G$ is open
\item $G$ is isomorphic to its profinite completion $\widehat{G}$
\item Every group homomorphism from $G$ to any other profinite group is continuous.
\end{itemize}

In other words, this means that the algebraic structure completely determines the topology of $G$. 

The problem of determining which profinite groups are strongly complete has been studied by several authors in the last fifty years. In \cite{Serre1997}, Serre proved that topologically finitely generated profinite pro-$p$ groups are strongly complete. In \cite{Pletch1981}, Pletch generalized the result of Serre, proving that profinite groups with topologically finitely generated pro-$p$-Sylows are strongly complete. Finally, Nikolov and Segal, using the classification of finite simple groups, proved that topologically finitely generated profinite groups are strongly complete in \cite{NikolovSegal2007}. Constructions of profinite groups that are not strongly complete also exist (see \cite[Example 4.2.13]{RibesZalesskii2000}).

In 2005, Grigorchuk introduced in \cite{Grigorchuk2005} the class of groups of finite type (also known as finitely constrained groups). These groups are profinite groups acting on regular rooted trees whose action locally around every vertex is given by a finite group of allowed actions. This finite group of allowed actions is called the minimal pattern subgroup and it acts on a truncated regular rooted tree with finitely many levels. The number $D$ of levels necessary to define the group of finite type is called its depth. The definition of groups of finite type mimics the idea of shifts of finite type, where the set of infinite words in these subshifts are given by those words whose finite snippets avoid a finite set of forbidden words (see \cite[Chapter 2]{Lind_Marcus_1995}). 

In \cite{Grigorchuk2005}, Grigorchuk proved that groups of finite type are regular branch groups (a subfamily of branch groups introduced in \cite{Grigorchuk2000}) and in \cite{Sunic2006}, Sunic proved the converse, namely, the closure of any regular branch group is a group of finite type. As a consequence, groups of finite type appear to be the closure of many well-studied abstract groups acting on rooted trees such as: the first Grigorchuk group \cite{Grig1980}, Grigorchuk-Gupta-Sidki groups \cite{GGS1983}, the Hanoi towers group \cite{GrigNekraSunic2006}, the iterated monodromy group of the polynomial $z^2+i$ \cite{GrigSavchukSunic2007}, iterated wreath products or the whole group of automorphisms of a regular rooted tree. Recently, it has also been shown that certain iterated Galois groups, used in number theory and arithmetic dynamics, are also groups of finite type (see \cite{Radi2025FPP}). In \cite{FariñaAsategui2024}, the author proved that fractal branch profinite groups are groups of finite type. 

The class of groups of finite type has been studied by many authors in the last two decades. For example, in \cite{Bartholdi2013, FariñaAsategui2024, Penland2017nearly, PenlandSunic2016, Sunic2006}, different authors studied possible values of their Hausdorff dimensions, showing that the Hausdorff dimension is always positive (if the group of finite type is not finite), rational in the case of prime power regular trees and calculated the spectrum, namely, all the possible values for the Hausdorff dimension in groups of finite type. In \cite{Bondarenko2010, Bondarenko2014, Sunic2010}, authors gave necessary and sufficient conditions over the minimal pattern subgroup to know whether a group of finite type is topologically finitely generated. In \cite{Bondarenko2014}, authors computed all groups of finite type of depth $2$, $3$ and $4$ acting on the binary tree, studying how many of them are topologically finitely generated. Moreover, they gave an explicit description of the minimal pattern subgroups of the topologically finitely generated groups of finite type of depth $4$ acting on the binary tree.

If $G$ is a group of finite type of depth $D$, its properties are in general studied by understanding the abelianization of the subgroup $\St_G(D-1)$, that corresponds to the subgroup of elements in $G$ fixing all the vertices at level $D-1$ of the tree. In \cite[Theorem 3]{Bondarenko2014}, Bondarenko and Samoilovych indirectly proved that just-infinite level-transitive groups of finite type are topologically finitely generated. On the other side, we already observed that Nikolov and Segal proved that topologically finitely generated profinite groups are strongly complete. One of the main results of this article closes the cycle between these three concepts for a vast family of groups of finite type (see \cref{figure: diagram equivalences}). 

\begin{figure}[t]
\centering
\begin{tikzpicture}[node distance=2.5cm, auto]
    % Define nodes: A at top center, B at bottom right, C at bottom left.
    \node (A) at (0,4) {Just-infinite};
    \node (B) at (4,0) {\shortstack{Topologically finitely \\ generated}};
    \node (C) at (-4,0) {\shortstack{Strongly \\  complete}};

    % Draw arrows with bends:
    \draw[->, bend left=20] (A) to node[midway, right] {BS} (B);
    \draw[->, bend left=20] (B) to node[midway, above] {NS} (C);
    \draw[->, bend left=20] (C) to node[midway, left] {*} (A);
\end{tikzpicture}
\caption{Cycle of equivalences in level-transitive groups of finite type. The label BS indicates that the implication was proven by Bondarenko and Samoilovych in \cite{Bondarenko2014}. The label NS that the implication was proven by Nikolov and Segal in \cite{NikolovSegal2007} and the label * that the implication uses that $\St_G(D-1)/\overline{\St_G(D-1)'}$ is torsion.}
\label{figure: diagram equivalences}
\end{figure}
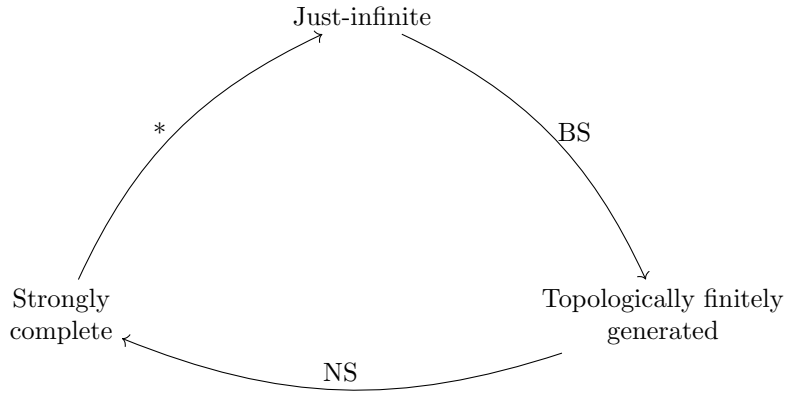

\begin{Theorem}
Let $G$ be a level-transitive group of finite type of depth $D$ such that $\St_G(D-1)/\overline{\St_G(D-1)'}$ is torsion. Then, the following properties are equivalent:

\begin{enumerate}[(i)]
\item $G$ is just-infinite,
\item $G$ is topologically finitely generate,
\item $G$ is strongly complete.
\end{enumerate}
\label{Theorem: finite type equivalences}
\end{Theorem}

The condition ``$\St_G(D-1)/\overline{\St_G(D-1)'}$ is torsion" has been studied by the author for many groups of finite type and it is suspected that it always holds. Unfortunately, a proof cannot be claimed yet.

\begin{Conjecture}
If $G$ is a group of finite type of depth $D$, then $\St_G(D-1)/\overline{\St_G(D-1)'}$ is torsion. In particular, \cref{Theorem: finite type equivalences} holds for all groups of finite type. 
\label{conjecture: property (E)}
\end{Conjecture}

In \cref{proposition: property (E) equivalences}, equivalences for \cref{conjecture: property (E)} are given. For level-transitive iterated wreath products, the quotient $\St_G(D-1)/\overline{\St_G(D-1)'}$ is always torsion and therefore we obtain the following corollary from \cref{Theorem: finite type equivalences}:

\begin{Corollary}
Let $d \geq 2$ and $\mathcal{P}$ be a transitive subgroup of $\Sym(d)$. Denote $W_\mathcal{P}$ the iterated wreath product of $\mathcal{P}$ acting on a $d$-regular tree and $\mathcal{P}'$ the commutator subgroup of $\mathcal{P}$. Then, the following properties are all equivalent:

\begin{enumerate}[(i)]
\item $W_\mathcal{P}$ is just-infinite,
\item $W_\mathcal{P}$ is topologically finitely generate,
\item $W_\mathcal{P}$ is strongly complete.
\item $\mathcal{P} = \mathcal{P}'$. 
\end{enumerate}

In particular, the whole group of automorphisms of the tree does not have any of these properties. 
\label{Corollary: IWP equivalence}
\end{Corollary}

\cref{Corollary: IWP equivalence} generalizes the result of Bondarenko in \cite{Bondarenko2010}, where he proved that for iterated wreath products, condition (ii) and condition (iv) of \cref{Corollary: IWP equivalence} are equivalent. 

The second corollary involves the Hanoi towers group on 3 pegs. This group, introduced in \cite{GrigNekraSunic2006}, it is probably the most famous example of a regular branch group that is not just-infinite (see \cite[Proposition 3.17]{Bartholdi2012} for a proof of this fact). As the group is regular branch and finitely generated, its closure is a topologically finitely generated group of finite type, and by \cref{Theorem: finite type equivalences}, we conclude the following:

\begin{Corollary}
The closure of the Hanoi towers group on 3 pegs is just-infinite.
\label{Corollary: closure Hanoi towers joo}
\end{Corollary}

This result yields to the first example of a regular branch group that is not just-infinite but its closure is. 

\bigskip

We will also be interested in distinguishing when two groups of finite type are isomorphic. Apart from allowing us to do a classification of these groups by isomorphism, the classification has direct applications to abstract groups. Let $G_1, G_2$ be two abstract subgroups of $\Aut(T)$, the group of automorphism of a regular rooted tree $T$. For $i = 1,2$, let $\overline{G_i}$ be the closure of $G_i$ in $\Aut(T)$ and $\widehat{G_i}$ the profinite completion of $G_i$. Let us assume that the groups $G_1$ and $G_2$ have the congruence subgroup property, namely, $\widehat{G_i} \simeq \overline{G_i}$ for $i = 1,2$. This assumption is known to be true in groups such as the first Grigorchuk Group or GGS groups with no constant defining vector. Assume further that the closure of the groups $G_1$ and $G_2$ are groups of finite type, which happens if for example the groups $G_1$ and $G_2$ are regular branch. Then, we have two cases: if $\overline{G_1} \ncong \overline{G_2}$, this means that the profinite completions of $G_1$ and $G_2$ are not isomorphic and therefore, the groups $G_1$ and $G_2$ are not isomorphic, giving a tool to distinguish abstract groups. If, contrarily, we have $\overline{G_1} \simeq \overline{G_2}$ and $G_1 \ncong G_2$, this means that the groups $G_1$ and $G_2$ are not profinite rigid, leading to another important property of these groups. 

To study groups of finite type up to isomorphism, we should start noticing that since they are profinite groups, an isomorphism must be a continuous group isomorphism. To find a condition, denote $G_\mathcal{P}$ and $G_\mathcal{Q}$ the groups of finite type, where $\mathcal{P}$ and $\mathcal{Q}$ are the respective minimal pattern subgroups. We will see in \cref{section: classification of group of finite type up to isomorphism} that we may assume that $G_\mathcal{P}$ and $G_\mathcal{Q}$ have the same depth. The first guess one can imagine is that $G_\mathcal{P}$ is isomorphic to $G_\mathcal{Q}$ if and only if $\mathcal{P}$ is isomorphic to $\mathcal{Q}$. However, this will be disproved in \cref{proposition: P iso Q does not imply GP iso GQ}. Denote by $\mathcal{S}_{\mathcal{P}\mathcal{Q}}$ the set of elements acting on the truncated tree with $D$ levels that conjugate $\mathcal{P}$ onto $\mathcal{Q}$. In \cref{section: classification of group of finite type up to isomorphism}, we will see that we can endow this set with an equivalence relation and we can endow the set of classes with a graph structure. We will denote this directed graph $\Gamma_{\mathcal{P}\mathcal{Q}}$. Then, the main result is the following:

\begin{Theorem}
Let $\mathcal{P}$ and $\mathcal{Q}$ be minimal pattern subgroups of depth $D$ and $\Gamma_{\mathcal{P}\mathcal{Q}}$ their associated directed graph. 

\begin{enumerate}[(i)]
\item If $\Gamma_{\mathcal{P}\mathcal{Q}}$ has a cycle, then $G_\mathcal{P}$ and $G_\mathcal{Q}$ are conjugated in the group of automorphisms of the tree.
\item If $G_\mathcal{P}$ is fractal, then $G_\mathcal{P}$ and $G_\mathcal{Q}$ are conjugated in the group of automorphisms of the tree if and only if $\Gamma_{\mathcal{P}\mathcal{Q}}$ has a cycle.
\end{enumerate}
\label{Theorem: finite type isomorphic conjugated iff gamma cycle}
\end{Theorem}

The definition of fractal and equivalences for a group of finite type to be fractal can be found in \cref{section: Fractality}. 

As a corollary of \cref{Theorem: finite type isomorphic conjugated iff gamma cycle} and using results of rigidity proved in \cite{NekrashevychBartholdi2006} and \cite{Fariña2025Boston}, we classify groups of finite type of a given depth and acting on a given tree:

\begin{Corollary}
We have the following bounds for the amount of isomorphic classes in different trees and with different depths:
\begin{enumerate}[(i)]
\item There are 5 isomorphic classes of groups of finite type in the binary tree with depth $2$.
\item There are between 16 and 23 isomorphic classes of groups of finite type in the binary tree with depth $3$.
\item Among the 32 topologically finitely generated groups of finite type in the binary tree of depth $4$, there are exactly 8 different isomorphic classes.
\item There are between 15 and 40 isomorphic classes in the ternary tree with depth $2$.
\item Among the 216 topologically finitely generated pro-3-groups of finite type in the ternary tree of depth $3$, there are exactly 12 different isomorphic classes.
\end{enumerate}
\label{Corollary: number of isomorphic classes}
\end{Corollary}

As part of \cref{Corollary: number of isomorphic classes}, the algorithm presented by Bondarenko and Samoilovich in \cite{Bondarenko2014} to find all minimal pattern subgroups of a certain depth acting on a certain tree will be improved to reach larger depths in bigger trees (see \cref{section: algorithms to find minimal pattern subgroups}).

\bigskip

By \cite{Sunic2006}, the closure of regular branch groups are groups of finite type, but, can we ensure with milder conditions whether an abstract group has as closure a group of finite type? can we explicitly give the depth and minimal pattern subgroup? The following theorem answer this question in the case of groups given by finite automata:

\begin{Theorem}
Let $\mathcal{A}$ be a finite automaton and $G(\mathcal{A})$, the group generated by $\mathcal{A}$.  Let $\mathcal{P}$ be a minimal pattern p-subgroup of depth $D$ such that its associated group of finite type $G_\mathcal{P}$ is topologically finitely generated. If the action on the first $D$ levels of $G(\mathcal{A})$ coincides with the action of $\mathcal{P}$, then $G_\mathcal{P}$ is the closure of $G(\mathcal{A})$.
\label{Theorem: closure automaton and finite type}
\end{Theorem}

Even though \cref{Theorem: closure automaton and finite type} gives no information whether the automata group is regular branch or not, it is possible to use the group of finite type that coincides with its closure to answer this question. In \cref{section: The maximal branching subgroup}, an algorithm and examples are presented. 

As the Hausdorff dimension of a group is the same as the Hausdorff dimension of its closure, \cref{Theorem: closure automaton and finite type} will be a very good tool to calculate the Hausdorff dimension of automata groups. It will also be a good tool to find automata groups whose closures are given groups of finite type. One corollary is the following:

\begin{Corollary}
\label{Corollary: group closure Grig group}
There exists a $4$-state automaton such that the group generated by this automaton contains the first Grigorchuk group as a subgroup of infinite index and such that its closure equals the closure of the first Grigorchuk group.  
\end{Corollary}

Combining \cref{Theorem: finite type isomorphic conjugated iff gamma cycle} with \cref{Theorem: closure automaton and finite type}, we also obtain 

\begin{Corollary}
The closure of the group $\IMG(z^2+i)$ and the closure of the third Grigorchuk group are isomorphic. 
\label{Corollary: closure IMG 3rd Grig group}
\end{Corollary}

It is known that the third Grigorchuk group has the congruence subgroup property (the proof follows verbatim \cite[Proposition 10]{Grigorchuk2000}), so its profinite completion is isomorphic to its closure, but it was recently proved in \cite{Radi2025IMG} that $\IMG(z^2+i)$ does not have the congruence subgroup property, so its profinite completion is not isomorphic to its closure. In virtue of \cref{Corollary: closure IMG 3rd Grig group}, this implies that their profinite completions are not isomorphic and therefore, the third Grigorchuk group and $\IMG(z^2+i)$ are not isomorphic, even though their closures are.  

\subsection*{Organization}
In \cref{section: preliminaries} the necessary background is introduced. In \cref{section: algorithms to find minimal pattern subgroups}, improvements of the algorithms to calculate all the minimal pattern subgroups are presented. In \cref{section: Hausdorff dimension fractality and torsion}, results about the Hausdorff dimension, as well as necessary and sufficient conditions to prove when groups of finite type are fractal, strongly fractal and super strongly fractal are established. Finally, a result about when the groups are torsion is given, showing with a family of groups of finite type that they can even be torsion-free. In \cref{section: Topological finite generation just infiniteness and strong completeness}, \cref{Theorem: finite type equivalences} and all its corollaries are proven. In \cref{section: Automata groups and groups of finite type}, \cref{Theorem: closure automaton and finite type} and \cref{Corollary: group closure Grig group} are proven. The \cref{section: classification of group of finite type up to isomorphism} is devoted to prove \cref{Theorem: finite type isomorphic conjugated iff gamma cycle}. In \cref{section: analysis of different cases}, \cref{Corollary: number of isomorphic classes} and \cref{Corollary: closure IMG 3rd Grig group} are proven. In \cref{section: The maximal branching subgroup}, properties of the maximal branching subgroup of a group of finite type are given, as well as an algorithm to prove whether an abstract group whose closure is a group of finite type is regular branch. Finally, in \cref{section: List of automata groups whose closure is a group of finite type}, examples of automata groups whose closure match with the topologically finitely generated groups of finite type studied along the article are listed.

\subsection*{Acknowledgements}

The author would like to thank R. Grigorchuk, D. Savchuk and J. Fariña-Asategui for the valuable discussions held and remarks. 

\section{Preliminaries}
\label{section: preliminaries}

\subsection{About general notation} 

Given $S$ a set, we will write $\#S$ to denote its cardinality. In the case that $S$ is a subgroup, we write $\abs{S}$ instead. If $G$ is a group, and $H$ is a subgroup of finite index, we will write this as $H \leq_f G$. If furthermore, $H$ is normal, we will write it as $H \lhd_f G$. If $x,y$ are two elements of $G$, the commutator element is defined as $[x,y] := xyx^{-1}y^{-1}$ and the commutator subgroup of $G$ is denoted by $G'$. The abelianization of $G$ will be denoted by $G^{ab} := G/G'$. If the group $G$ is a topological group, we will write $G^{\overline{ab}} := G/\overline{G'}$. The $n$-th derived series will be denoted $G^{(n)}$ and the lower central series $\gamma_k(G)$. Finally, $C_n$ denotes the cyclic group of order $n$. 

\subsection{Groups acting on rooted trees}
\label{subsection: Groups acting on rooted trees}

A \textit{spherically homogeneous rooted tree} $T$ is an infinite tree with a root $\emptyset$, where the vertices at the same distance from the root all have the same number of descendants. The set of vertices at a distance exactly $n \geq 1$ from the root form the $n$-th level of $T$ and will be denoted $\mathcal{L}_n$. The vertices whose distance is at most $n$ from the root form the $n$-th truncated tree $T^n$. If all the vertices of the tree have the same number of descendants $d$, the tree $T$ will be called \textit{$d$-regular}. In this case, we  can describe the vertices as finite words on $X = \set{1,\dots,d}$. 

The group of automorphisms of $T$, denoted $\Aut(T)$, is the group of bijective functions from $T$ to $T$ that preserve the root and adjacency between vertices. Then, $\Aut(T)$ acts on $\mathcal{L}_n$ for all $n \ge 1$. For any vertex $v$ in $T$, the subtree rooted at $v$ which is again a spherically homogeneous infinite rooted tree, is denoted $T_v$. Notice also that if $v$ and $w$ are vertices on the same level, then $T_v$ and $T_w$ are isomorphic and that if $T$ is $d$-regular then $T_v$ is also $d$-regular. In this article, the action of $\Aut(T)$ on $T$ will be on the left, so if $g,h \in \Aut(T)$ and $v$ is a vertex in $T$, then $(gh)(v) = g(h(v))$.

Given a vertex $v$ in $T$, we write $\st(v)$ for the \textit{stabilizer of the vertex $v$}, namely, the subgroup of the elements $g \in \Aut(T)$ such that $g(v) = v$. Given $n \ge 1$, we write $\St(n) := \bigcap_{v\in \mathcal{L}_n}\st(v)$, and we call it the \textit{stabilizer of level $n$}. The subgroup $\mathrm{St}(n)$ is a normal subgroup of finite index in $\mathrm{Aut}(T)$.

We can make $\Aut(T)$ a topological group by declaring $\{\St(n)\}_{n\ge 1}$ to be a base of neighborhoods of the identity. This topology is called the \textit{congruence topology} and with it, $\Aut(T)$ is homeomorphic to a profinite group.

Let $v$ be a vertex in $T$, $1\le n \le \infty$ and $g \in \Aut(T)$. By preservation of adjacency, we have $g(T_v) = T_{g(v)}$. The \textit{section} of $g$ at $v$, denoted $g|_v$, is the map obtained by restricting the action of $g$ to $T_v$. Truncating the action of $g|_v$ to the first $n$ levels, we obtain $g|_v^n \in \Aut(T_v^n)$ that satisfies that $$g(vw)=g(v)g|_v^n(w)$$ for all $w$ vertex in $T_v^n$. For $n=1$, the permutation $g|_v^1$ is known as the \textit{label} of $g$ at $v$. The collection of all labels $\set{g|_v^1: v \text{ vertex in } T}$ is called the \textit{portrait} of $g$ and uniquely determines $g$. For every $n\ge 1$, we define the map $\pi_n:\Aut(T)\to \Aut(T^n)$ given by $\pi_n(g):=g|_\emptyset^n $. The following three properties are satisfied: 
\begin{align}
(gh)|_v^n = g|_{h(v)}^n h|_v^n, \qquad (g^{-1})|_v^n = (g|_{g^{-1}(v)}^n)^{-1} \qquad \text{ and } \qquad g|_{vw} = (g|_v)|_w.
\label{equation: properties sections}
\end{align}

Sections allow us to have the following isomorphism for any natural number $n\ge 1$: 
\begin{align}
\begin{split}
\psi_n:\Aut(T) &\to \big(\Aut(T_{v_1}) \times \dotsb \times \Aut(T_{v_{N_n}})\big) \rtimes \Aut(T^n) \\ g &\mapsto (g|_{v_1},\dotsc,g|_{v_{N_n}}) \pi_n(g),
\end{split}
\label{equation: Aut and semidirect product}
\end{align}
where $v_1,\dotsc,v_{N_n}$ are all the distinct vertices in $\mathcal{L}_n$ labeled from left to right.

To describe the elements of $\Aut(T)$, we will use the isomorphism $\psi_1$ and write $g=(g_1,\dotsc,g_{\abs{\mathcal{L}_1}})\pi_1(g)$ by a slight abuse of notation. Sometimes, we will also use the isomorphism $\psi_n$ to represent the elements of $\St(n)$ as tuples $g = (g_1, \dots, g_{\abs{\mathcal{L}_n}})_n$, where the subscript $n$ at the end of the tuple is used to indicate that the action of $g$ on the first $n$ levels is trivial. 

Given $v$ a vertex in $T$, define also the map
\begin{equation}
\varphi_v: \Aut(T) \rightarrow \Aut(T_v) \text{, given by } g \mapsto g|_v.
\label{equation: map varphi v}
\end{equation}
Note that by \cref{equation: properties sections}, the restriction of $\varphi_v$ to $\st(v)$ is a group homomorphism.

Let us fix a subgroup $G \le \Aut(T)$. We define vertex stabilizers and level stabilizers as $\st_G(v) := \st(v) \cap G$ and $\St_G(n) := \St(n) \cap G$ respectively, for any $v$ vertex in $T$ and $n \ge 1$. We say that a group $G \leq \Aut(T)$ is \textit{level-transitive} if the action of $G$ on each level $\mathcal{L}_n$ is transitive for all $n \ge 1$. Notice that if $G$ is level-transitive, it must be infinite. 

The following is a well-known result of level-transitivity:

\begin{Lemma}
Let $T$ be a spherically homogeneous tree and $G \leq \Aut(T)$. Then $G$ is level-transitive if and only if there exists a sequence of vertices $\set{v_n}_{n \in \NN}$ with $v_n \in \mathcal{L}_n$ such that $\st_G(v_n)$ acts transitive on the immediate descendants of $v_n$. 
\label{lemma: level transitivity equivalence}
\end{Lemma}

Notice that since $\Aut(T)$ is a profinite group, if $G$ is a closed subgroup of $\Aut(T)$, then $G$ has a unique normalized Haar measure $\mu$.

\subsection{Branch groups}

Let $T$ be a spherically homogeneous rooted tree, a subgroup $G \leq \Aut(T)$, a vertex $v$ in $T$ and $n \in \NN$. We define the \textit{rigid vertex stabilizer} of $v$ (denoted by $\rist_G(v)$) to the subgroup of $G$ consisting of automorphisms which fix every vertex not in $T_v$. We define the \textbf{$n$-th rigid level stabilizer} and denote it $\RiSt_G(n)$ to the subgroup generated by the subgroups $\rist_G(v)$ with $v \in \mathcal{L}_n$. It is well-known that $\RiSt_G(n)$ is a normal subgroup of $G$ and it is isomorphic to $\prod_{v \in \mathcal{L}_n} \rist(v)$.

We say that a group $G \leq \Aut(T)$ is \textit{weakly branch} if $\RiSt_G(n) \neq 1$ for all $n \in \NN$. The following result, due to Abért, is a corollary from \cite[Corollary 1.4]{Abert2003}:

\begin{theorem}
\label{theorem: torsion measure weakly branch}
Let $G$ be a closed weakly branch group and $\mu$ its Haar measure. Then $$\mu(\set{g \in G: \abs{g} < \infty} = 0.$$
\end{theorem}

For branch groups, the definition often used is that the group must be level-transitive and $\RiSt_G(n) \leq_f G$ for all $n \in \NN$. However, for the purpose of this article, we will need a more general definition also given by Grigorchuk in \cite[Definition 1]{Grigorchuk2000}. 

\begin{Definition}
A group $G$ is called \textit{branch} if there exist a spherically homogeneous infinite rooted tree $T$ such that $G \hookrightarrow \Aut(T)$ is level-transitive and there exists a family of pairs $\set{(H_n, L_n)}_{n = 1}^{\infty}$ such that for all $n \in \NN$,

\begin{enumerate}[(i)]
\item $H_n \leq \St(n)$, 
\item $H_n \lhd_f G$, 
\item $H_n \simeq \prod_{\mathcal{L}_n} L_n$, where each copy of $L_n$ acts non-trivially only on one vertex of $\mathcal{L}_n$. 
\end{enumerate}
\label{definition: branch group}
\end{Definition}

The family of pairs is called a \textit{branch structure} for $G$, and it is not unique. In the case where $G$ is a profinite group, we say that $G$ is a \textit{branch group} if there exists an embedding of $G$ into $\Aut(T)$ as a closed subgroup and a branch structure defined as before where $H_n$ is closed for every $n \in \NN$.

Recall that a topological infinite group is \textit{just-infinite} if every closed non-trivial normal subgroup has finite index. The main result that it will be used for profinite branch groups in this article is the following:

\begin{theorem}[{\cite[Theorem 4]{Grigorchuk2000}}]
Let $G$ be a profinite branch group with branch structure $\set{(H_n,L_n)}_{n \in \NN}$. Then $G$ is just-infinite if and only if $L_n^{\overline{ab}}$ is finite for all $n \in \NN$.
\label{theorem: equivalence joo in branch groups}
\end{theorem}

If we now restrict ourselves to work on $d$-regular trees, we can identify $T_v$ and $T$ in a natural way for any $v \in \mathcal{L}_n$, and therefore we can also identify $\Aut(T)$ with $\Aut(T_v)$. We say that a group $G \leq \Aut(T)$ is \textit{self-similar} if $g|_v \in G$ for all $v$ vertex in $T$ and $g \in G$.  

Given a vertex $v$ in $T$, define the map $\delta_v: \Aut(T) \rightarrow \rist(v)$ such that 
\begin{equation}
\delta_v(g) |_v = g.
\label{equation: map deltav}
\end{equation}
The following properties are easy to check:
\begin{enumerate}[(i)]
\item $\delta_v$ is an injective homomorphism of groups.
\item $\delta_v \circ \delta_w = \delta_{vw}$.
\end{enumerate}

If $K$ is a subgroup of $\Aut(T)$ and $n \in \NN$, we define the \textit{geometric product} of $K$ on level $n$ as $$K_n := \set{g \in \St(n): g|_v \in K \text{ for all $v \in \mathcal{L}_n$}}.$$

\begin{Definition}
We say that $G \leq \Aut(T)$ is \textit{regular branch} over a subgroup $K$ if $K_1 \leq_f K \leq_f G$. The group $K$ is called the \textit{branching subgroup}.
\end{Definition}

In the case that $G$ is a self-similar group, the condition $[K:K_1] < \infty$ need not be checked:

\begin{Lemma}
Let $T$ be a $d$-regular tree, a group $G \leq \Aut(T)$ self-similar and $K \leq G$ such that $K_1 \leq K \leq_f G$. Then, $G$ is regular branch over $K$.
\label{lemma: regular branch for self-similar groups}
\end{Lemma}

\begin{proof}
Since $G$ is self-similar, we have $K \cap \St_G(1) \leq G_1$ and therefore
\begin{align*}
[K:K_1] = [K:K \cap \St_G(1)][K \cap \St_G(1): K_1] \leq [K:\St_K(1)][G_1: K_1] \\ \leq \abs{\pi_1(K)} [G:K]^d < \infty.
\end{align*}
\end{proof}

The following lemma will be used in \cref{section: Topological finite generation just infiniteness and strong completeness}:

\begin{Lemma}
Let $T$ be a $d$-regular tree and $K \leq \Aut(T)$. Then $(K_1)' = (K')_1$.
\label{lemma: commutator and geometric product}
\end{Lemma}

\begin{proof}
The group $(K_1)'$ is generated by elements of the form $[a,b]$ with $a,b \in K_1$. Therefore, there exist $a_1, \dots, a_d, b_1, \dots, b_d \in K$ such that $a = (a_1, \dots, a_d)_1$ and $b = (b_1, \dots, b_d)_1$. Then $$[a,b] = ([a_1,b_1], \dots, [a_d,b_d])_1 \in (K')_1.$$  

Conversely, it is not hard to see that $(K')_1$ is generated by elements of the form $\delta_i([a,b])$, where $i$ is a vertex of level one and $\delta_i$ is the map defined in \cref{equation: map deltav}. Then $\delta_i([a,b]) = [\delta_i(a), \delta_i(b)]$ and each $\delta_i(a)$ and $\delta_i(b) \in K_1$.
\end{proof}

The following result will be used in \cref{section: classification of group of finite type up to isomorphism}:

\begin{Lemma}[{\cite[Lemma 1.2]{Bartholdi2012}}]
Let $T$ be a $d$-regular tree and $G \leq \Aut(T)$ be a regular branch group. Then, there exists a unique subgroup $K$ such that $G$ is regular branch over $K$ and $K$ is maximal, namely, if $G$ is also regular branch over another subgroup $H$, then $H \leq K$.
\label{lemma: unique maximal reg branch subgroup}
\end{Lemma}

\subsection{Groups of finite type}

\begin{Definition}
Let $D$ be a positive natural number, $T$ a $d$-regular rooted tree and $\mathcal{P}$ a subgroup of $\Aut(T^D)$. The \textit{group of finite type} of \textit{depth} $D$ and set of \textit{patterns} $\mathcal{P}$ is defined as 
\begin{equation*}
G_\mathcal{P} := \set{g \in \Aut(T): g|_v^D \in \mathcal{P} \text{ for all $v$ vertex in $T$}}.
\label{equation: finite type definition}
\end{equation*}
A profinite group $G$ is said to be of \textit{finite type} if there exists an embedding to a $d$-regular tree, a number $D \in \NN$ and a subgroup $\mathcal{P} \leq \Aut(T^D)$ such that $G$ is isomorphic to $G_\mathcal{P}$.
\end{Definition}

If for example $D = 1$, we are forcing the labels $g|_v^1$ of each element $g$ in $G_\mathcal{P}$ to be in a certain subgroup $\mathcal{P}$ of $\Sym(d)$. In this case, the group will be isomorphic to the inverse limit of wreath products $\mathcal{P} \wr \mathcal{P} \wr \dots \wr \mathcal{P}$. This latter inverse limit group is known as \textit{iterated wreath product} and it will be denoted $W_\mathcal{P}$. Notice that if $\mathcal{P} = 1$, then $W_\mathcal{P} = 1$ and if $\mathcal{P} = \Sym(d)$, then $W_\mathcal{P} = \Aut(T)$. In the case that $d = q$ is a prime power, and $\mathcal{P} = \<(1, \dots, q)>$ where $(1, \dots, q)$ is the cyclic permutation, we will denote the iterated wreath product as $W_q$.

The following lemma will be used in \cref{section: classification of group of finite type up to isomorphism}:

\begin{Lemma}
If $p$ is prime and $G \leq W_p$ is infinite and self-similar, then $G$ is level-transitive.
\label{lemma: self similarity and level-transitivity}
\end{Lemma}

\begin{proof}
Given $n \in \NN$, choose $w_n$ any vertex on level $n$. As $G$ is infinite and $\st_G(w_n)$ has finite index in $G$, then also $\st_G(w_n)$ is infinite, so we can find $g \in \st_G(w_n)$ non-trivial. Let $m$ be the smallest level such that $g|_{u_m}^1 \neq 1$ for some vertex $u_m \in \mathcal{L}_m$ below $w_n$, and write $u_m = u'_{m-n} v_n$ where $u'_{m-n} \in \mathcal{L}_{m-n}$. As $m$ is the smallest, this means that $g \in \st_G(u_m)$. Since $G$ is self-similar, then $g|_{u'_{m-n}} \in \st_G(v_n)$ and acts non-trivially on the immediate descendants of $v_n$. Since $G \leq W_p$, the action of the powers of $g|_{u'_{m-n}}$ below $v_n$ must be transitive. The result follows by \cref{lemma: level transitivity equivalence}.
\end{proof}

The following results about groups of finite type are well-known:

\begin{Proposition}[{\cite[page 3]{Bondarenko2014} and \cite[Proposition 7.5]{Grigorchuk2005}}]
If $G$ is a group of finite type of depth $D$ embedded in a $d$-regular tree $T$, then $G$ is closed in $\Aut(T)$, self-similar and regular branch over $\St_G(D-1)$.
\label{proposition: finite type is closed self-similar and regular branch}
\end{Proposition}

\begin{theorem}[{\cite[Theorem 3]{Sunic2010}}]
Groups of finite type are the closure of regular branch groups. Furthermore if $G$ is a regular branch group branching over a subgroup $K$ containing $\St_G(D-1)$, then $\overline{G}$ is a group of finite type of depth $D$.
\label{theorem: characterization finite type groups}
\end{theorem}

\begin{Lemma}[{\cite[Lemma 10]{Sunic2006}}] 
Let $T$ be a $d$-regular tree and $G \leq \Aut(T)$ be a self-similar regular branch group, branching over a subgroup $K$ containing $\St_G(m)$. Then, for all $n \geq m$, $$\St_G(n) = (\St_G(m))_{n-m},$$ where the right-hand side corresponds to the geometric product of $\St_G(m)$.
\label{lemma: St(n+1) = prod St(n)}
\end{Lemma}

\cref{lemma: St(n+1) = prod St(n)} is the key result to give a branch structure for groups of finite type:

\begin{Lemma}
If $G$ is a group of finite type of depth $D$ embedded in a $d$-regular tree $T$ and acting level-transitively, then $\set{(\St_G(n), \St_G(D-1))}$ is a branch structure for $G$. Therefore $G$ is a branch group in the sense of \cref{definition: branch group}.
\label{lemma: branch structure groups of finite type}
\end{Lemma}

\begin{proof}
By \cref{proposition: finite type is closed self-similar and regular branch}, $G$ is closed in $\Aut(T)$ and therefore $\St_G(n)$ is closed in $G$, normal and with finite index. Finally, \cref{lemma: St(n+1) = prod St(n)} justifies condition (iii).
\end{proof}

As a corollary of \cref{theorem: equivalence joo in branch groups} and \cref{lemma: branch structure groups of finite type}, we obtain:

\begin{corollary}
If $G$ is a group of finite type of depth $D$ embedded in a $d$-regular tree $T$ and acting level-transitively, then $G$ is just-infinite if and only if $\St_G(D-1)^{\overline{ab}}$ is finite.
\label{corollary: finite type just infinite}
\end{corollary}

To finish, we prove that infinite groups of finite type are weakly branch:

\begin{Lemma}
If $G$ is a group of finite type of depth $D$ embedded in a $d$-regular tree $T$, then $\St_G(n+D-1) \leq \RiSt_G(n)$ for all $n \in \NN$. In particular, if $G$ is infinite, then it is weakly branch.
\label{lemma: finite type is weakly regular branch}
\end{Lemma}

\begin{proof}
If $g \in \St_G(n+D-1)$, for each vertex $v \in \mathcal{L}_n$ we have $g|_v \in \St_G(D-1)$. Then $\delta_v(g|_v) \in \rist_G(v)$ and $g = \prod_{v \in \mathcal{L}_n} \delta_v(g|_v) \in \RiSt_G(n)$. As $\St_G(n+D-1) \leq_f G$ and $G$ is not finite, then $\RiSt_G(n) \neq 1$ for all $n \in \NN$.
\end{proof}

\section{Algorithms to find minimal pattern subgroups}
\label{section: algorithms to find minimal pattern subgroups}

\subsection{Minimal pattern subgroups}

Let $T$ be a $d$-regular tree and $D$ a natural number. Then $\mathcal{L}_D$ has $d^D$ vertices that can be labeled from left to right as the numbers from $1$ to $d^D$. This allows us to describe elements in $\Aut(T^D)$ as permutations in $\Sym(d^D)$ that will facilitate the description of explicit examples.

\begin{Example}
In the $2$-regular tree with depth $D = 2$, consider $\mathcal{P} = \<\alpha>$ where $\alpha = (1,2)(3,4)$. Its portrait is shown in the \cref{figure: portrait of not minimal pattern}, where $\sigma$ is the non-trivial permutation of $\Sym(2)$.

\begin{figure}[h!]
    \centering
    \begin{tikzpicture}
    \node {$\id$}
        child {node {$\sigma$}}
        child {node {$\sigma$}};
    \end{tikzpicture}
    \caption{Portrait of the element $\alpha$.}
    \label{figure: portrait of not minimal pattern}
\end{figure}
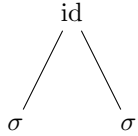

Notice that if an element $g \in G_\mathcal{P}$ had a vertex $v$  in $T$ such that $g|_v^2 = \alpha$, then the vertices underneath $v$ would have to have a pattern that starts with $\sigma$, but $\mathcal{P}$ does not contain any pattern like this. Hence,  $G_\mathcal{P} = 1$. 
\label{example: problem minimal pattern}
\end{Example}

\cref{example: problem minimal pattern} suggests the existence of a minimal pattern subgroup, that is justified in the following proposition:

\begin{Proposition}[{\cite[Section ``Minimal patterns groups"]{Bondarenko2014}}]
Let $T$ be a $d$-regular tree, $\mathcal{Q} \leq \Aut(T^D)$ and $G_{\mathcal{Q}}$ the group of finite type associated to $\mathcal{Q}$. Define the subgroup $\mathcal{P} = \pi_D(G_\mathcal{Q})$. Then $G_\mathcal{Q} = G_\mathcal{P}$ and $\mathcal{P} \leq \mathcal{Q}$.
\label{proposition: existence minimal pattern}
\end{Proposition}

If $G$ is a group of finite type of depth $D$, the subgroup $\mathcal{P}$ given in \cref{proposition: existence minimal pattern} is called the \textit{minimal pattern subgroup}.

In \cite[``Pattern graph"]{Bondarenko2014}, authors provide an algorithm that given a subgroup $H \leq \Aut(T^D)$, it returns $\mathcal{P}(H)$, the associated minimal pattern subgroup. The algorithm implies constructing a labeled directed graph whose vertices are the patterns in $H$, and the edges are labeled with the vertices at level one of the tree. Two patterns $a$ and $b$ are joined by an arrow from $a$ to $b$ with label $x$ if the pattern $a$ can be extended using the pattern $b$ through the vertex $x$. The algorithm deletes those patterns that cannot be extended along some vertex of the tree and when the reduction is no longer possible, the remaining patterns conform $\mathcal{P}(H)$, the minimal pattern subgroup associated to $H$. The problem with this algorithm is that when $H$ is too big, the number of vertices and edges in the graph increases considerably and no computer can take a reasonable time to find $\mathcal{P}(H)$. Nevertheless, the algorithm can be improved by using the group structure of $H$ and $\mathcal{P}(H)$.

\subsection{Improving the algorithm that finds $\mathcal{P}(H)$}

Let $T$ be a $d$-regular tree, $D$ a natural number, $p \in \Aut(T^D)$ a pattern and $x \in \mathcal{L}_1$ a vertex in level one. The notation $p|_x$ will refer to the section of $D-1$ levels of $p$ at the vertex $x \in \mathcal{L}_1$. In particular, $p|_x \in \pi_{D-1}(\Aut(T^{D-1}))$.

\begin{Definition}
Given $H$ a subgroup of $\Aut(T^D)$, we say that $K \leq H$ is \textit{closed under patterns} of $H$ if for all $k \in K$ and $x \in \mathcal{L}_1$ then $k|_x \in \pi_{D-1}(H)$. 
\end{Definition}

\begin{Lemma}
Let $T$ be a $d$-regular tree, $D$ a natural number and $\mathcal{P} \leq \Aut(T^D)$. Then, $\mathcal{P}$ is a minimal pattern subgroup if and only if $\mathcal{P}$ is closed under patterns of itself.
\label{lemma: equivalence minimal pattern}
\end{Lemma}

\begin{proof}
$(\Rightarrow)$ If $\mathcal{P}$ is a minimal pattern subgroup then there exists $G$ a group of finite type of depth $D$ such that $\pi_D(G) = \mathcal{P}$. Then, if $p \in \mathcal{P}$ and $x \in \mathcal{L}_1$, there exists $g \in G$ such that $g|_\emptyset^D = p$ and since $G$ is self-similar, also $g|_x^{D-1} = p|_x \in \pi_{D-1}(\mathcal{P})$.

$(\Leftarrow)$ Conversely, consider $G = G_\mathcal{P}$ and call $\mathcal{Q} = \pi_D(G)$ the minimal pattern subgroup of $G$. By \cref{proposition: existence minimal pattern}, we know that $\mathcal{Q} \leq \mathcal{P}$ and the goal is to prove they are equal. Start with $p \in \mathcal{P}$. Since $p|_x \in \pi_{D-1}(\mathcal{P})$ because $\mathcal{P}$ is closed under patterns of itself, there exists a pattern $q \in \mathcal{P}$ such that $p|_x = \pi_{D-1}(q)$. Gluing $q$ to $p$ and doing this for every vertex $x \in \mathcal{L}_1$, we extend $p$ with depth $D$ to an element defined until level $D+1$. But then, each pattern of depth $D$ ending the element created so far can also be extended by the property of $\mathcal{P}$, allowing us to construct an element $g \in G$ such that $g|_\emptyset^D = p$. Then $p \in \mathcal{Q}$ and $\mathcal{P} = \mathcal{Q}$, implying that $\mathcal{P}$ is the minimal pattern subgroup.
\end{proof}

In particular by \cref{lemma: equivalence minimal pattern}, the subgroup $\mathcal{P}(H)$ is closed under patterns of $H$. The following lemmas have as goal to prove that $H$ has a subgroup that is maximal among the subgroups of $H$ closed under patterns of $H$.

\begin{Lemma}
Let $T$ be a $d$-regular tree, $D$ a natural number, $H$ a subgroup of $\Aut(T^D)$, $K$ a subgroup of $H$ and $S$ a generating set of $K$. Then, $K$ is closed under patterns of $H$ if and only if for all $s \in S$ and $x \in \mathcal{L}_1$ we have $s|_x \in \pi_{D-1}(H)$. 
\label{lemma: closed under H generator set}
\end{Lemma}

\begin{proof}
Direct is straightforward. For the converse, we apply induction on the length of $g$ as a word in $S$. If the length is $1$ then $g = s^{\pm 1}$. The case $g = s$ is the hypothesis. If $g = s^{-1}$, then $(s^{-1})|_x = (s|_{s^{-1}(x)})^{-1} \in \pi_{D-1}(H)$ because $H$ is a subgroup. Similarly, since $(gh)|_x = g|_{h(x)} \, h|_x$ and each term is in $\pi_{D-1}(H)$, then $(gh)|_x \in \pi_{D-1}(H)$.
\end{proof}

\begin{Lemma}
Let $T$ be a $d$-regular tree, $D$ a natural number, $H$ a subgroup of $\Aut(T^D)$ and $K_1, K_2 \leq H$ closed under patterns of $H$. Then $\<K_1,K_2>$ is closed under patterns of $H$.
\label{lemma: join closed under patterns}
\end{Lemma}

\begin{proof}
If $S_i$ is a generating set of $K_i$ for $i = 1,2$, then $S_1 \cup S_2$ is a generating set of $\<K_1, K_2>$. Since $K_1$ and $K_2$ are closed under patterns of $H$, then $\<K_1, K_2>$ is closed under patterns of $H$ by \cref{lemma: closed under H generator set}.
\end{proof}

Let $T$ be a $d$-regular tree, $D$ a natural number and let $\mathcal{L}$ be the lattice of subgroups of $\Aut(T^D)$ ordered by inclusion. If $H \geq K$, we can put a distance between them, defining $d(H,K)$ as the smallest $n$ such that there is a chain of subgroups $H = H_0 \gneq H_1 \gneq \dots \gneq H_n = K$ where $H_{i+i}$ is a maximal subgroup of $H_i$ for all $i = 0,\dots,n-1$. Clearly $d(H,H) = 0$ and it is not hard to see that if $H \geq K \geq J$, then $d(H,J) \leq d(H,K) + d(K,J)$.

\begin{Lemma}
Let $T$ be a $d$-regular tree, $D$ a natural number, $H \leq \Aut(T^D)$ and $K_1$ the closest subgroup of $H$ closed under patterns of $H$. If $K_2$ is another subgroup of $H$ closed under patterns of $H$, then $K_2 \leq K_1$.
\label{lemma: maximal closed under patterns}
\end{Lemma}

\begin{proof}
If $K_2$ is trivial there is nothing to do. Suppose that $K_2$ is not trivial. By closeness of $K_1$ to $H$, then necessarily $K_1$ is also non-trivial. Now, if $K_2 \nleq K_1$, then $K_1 < \<K_1,K_2> \leq H$ and by \cref{lemma: join closed under patterns}, the join is closed under patterns of $H$. However, this contradicts the fact that $K_1$ was the closest to $H$.
\end{proof}

\cref{lemma: maximal closed under patterns} justifies the existence of a subgroup $\mathfrak{P}(H)$ that is maximal between the subgroups of $H$ that are closed under patterns of $H$. Since $\mathcal{P}(H)$ is closed under patterns of $H$, necessarily $\mathcal{P}(H) \leq \mathfrak{P}(H)$.

The next step is to find an efficient algorithm that computes $\mathfrak{P}(H)$. Let $S$ be a generating set of $H$. We can separate $S$ in a disjoint union of sets $S_1$ and $S_2$ such that $s \in S_1$ if and only if $s|_x \in \pi_{D-1}(H)$ for all $x \in \mathcal{L}_1$. Let $K = \<S_1>$. By \cref{lemma: closed under H generator set}, $K$ is closed under patterns of $H$ and by \cref{lemma: maximal closed under patterns} we have that $K \leq \mathfrak{P}(H)$. We can use any generating set $S$ but there is a compromise between choosing $S$ big or small. If it is big, there are more chances to have $S_1$ also big and consequently $K = \<S_1>$ big as well, but on the other side, the bigger $S$ is, the harder is for the software to do calculations with $H = \<S>$. Furthermore, finding useful extra elements $s$ to add in a generating set $S$ is not easy to control.

Once $K$ is obtained, consider $H/K$. Observe that if $t_1 \equiv t_2 \pmod{K}$, then $t_1 \in \mathfrak{P}(H)$ if and only if $t_2 \in \mathfrak{P}(H)$ because $K \leq \mathfrak{P}(H)$. Therefore, a coclass of $H/K$ is included in $\mathfrak{P}(H)$ independently of the representative of $H/K$ we choose. Take $T$ a set of representatives of $H/K$ and let $t \in T \setminus K$. If $t|_x \in \pi_{D-1}(H)$ for all $x \in \mathcal{L}_1$, this means that $t \in \mathfrak{P}(H)$, so we can replace $K$ by $\<K,t>$ (that is necessarily bigger than the old $K$) and recalculate $H/K$. Since $H$ is a finite group, once we find a $K$ such that for some set of representatives $T$ of $H/K$, no $t \in \mathfrak{P}(H)$, that means that this latter $K$ equals $\mathfrak{P}(H)$.

The last ingredient for the algorithm is the following: 

\begin{Lemma}
Let $T$ be a $d$-regular tree, $D$ a natural number and $S$ a generating set of $\mathcal{P} \leq \Aut(T^D)$. Then $\mathcal{P}$ is a minimal pattern subgroup if and only if for all $s \in S$ and $x \in \mathcal{L}_1$, we have $s|_x \in \pi_{D-1}(\mathcal{P})$.
\label{lemma: minimal pattern generating set}
\end{Lemma}

\begin{proof}
By \cref{lemma: equivalence minimal pattern}, the subgroup $\mathcal{P}$ is a minimal pattern subgroup if and only if it is closed under patterns of itself. Then, the result follows by \cref{lemma: closed under H generator set} applied to $K = H = \mathcal{P}$.
\end{proof}

Finally, the algorithm to find $\mathcal{P}(H)$ is as follows:

\begin{enumerate}
\item Take $S$ a minimal generating set of $H$. If \cref{lemma: minimal pattern generating set} applies, then $\mathcal{P}(H) = H$. 
\item If not, we look for $\mathfrak{P}(H)$ as it was explained before. Since step 1 did not apply, necessarily $\mathfrak{P}(H) < H$. On the other side, we know that $\mathcal{P}(H) \leq \mathfrak{P}(H)$, so we may replace $H$ by $\mathfrak{P}(H)$ and start step 1 again. 
\end{enumerate}

\subsection{Finding all minimal pattern subgroups}

We now focus on constructing an algorithm such that given $d$ and $D$ gives us all the minimal pattern subgroups in a $d$-regular tree with depth $D$. Since, the number of subgroups that we would have to check is huge as $d$ and $D$ increase, some theoretic tools are developed to help in the search.

\begin{Lemma}
\label{lemma: minimal pattern St(1)}
Let $T$ be a $d$-regular tree, $D$ a natural number, and $\mathcal{P} \leq \Aut(T^D)$. If $\mathcal{P}$ is a non-trivial minimal pattern subgroup, then $\mathcal{P} \nsubseteq \pi_D(\St(1))$.
\end{Lemma}

\begin{proof}
The subgroup $\mathcal{P}$ is non-trivial so there is an element $g \in G_\mathcal{P}$ such that $g|_v^1 \neq 1$ at some vertex $v$. Then $g|_v^D \in \mathcal{P}$ and it is not in $\pi_D(\St(1))$.
\end{proof}

\cref{lemma: minimal pattern St(1)} explains why the subgroup given in \cref{example: problem minimal pattern} is not a minimal pattern subgroup.

\begin{Lemma}
\label{lemma: join minimal pattern subgroups}
Let $T$ be a $d$-regular tree, $D$ a natural number, and $\mathcal{P}_1, \mathcal{P}_2 \leq \Aut(T^D)$. If $\mathcal{P}_1$ and $\mathcal{P}_2$ are minimal patterns then $\<\mathcal{P}_1, \mathcal{P}_2>$ is a minimal pattern.
\end{Lemma}

\begin{proof}
By \cref{lemma: equivalence minimal pattern}, each $\mathcal{P}_i$ is closed under patterns of itself, so it is closed under patterns of $\<\mathcal{P}_1, \mathcal{P}_2>$. Then by \cref{lemma: join closed under patterns}, we have that $\<\mathcal{P}_1, \mathcal{P}_2>$ is closed under patterns of $\<\mathcal{P}_1, \mathcal{P}_2>$.
\end{proof}

Recall the distance between subgroups defined in the previous subsection.

\begin{Lemma}
Let $T$ be a $d$-regular tree, $D$ a natural number, $H \leq \Aut(T^D)$ and $\mathcal{P}_1 \leq H$ the closest  minimal pattern to $H$. If $\mathcal{P}_2$ is another minimal pattern contained in $H$, then $\mathcal{P}_2 \leq \mathcal{P}_1$.
\label{lemma: maximal minimal patterns}
\end{Lemma}

\begin{proof}
The proof is identical to the proof of \cref{lemma: maximal closed under patterns}.
\end{proof}

\begin{Lemma}
Let $T$ be a $d$-regular tree and $D$ a natural number. If $H \leq \Aut(T^D)$, then $\mathcal{P}(H)$ is the closest minimal pattern subgroup to $H$.
\label{lemma: P(H) closest minimal pattern}
\end{Lemma}

\begin{proof}
Let $\mathcal{P}$ be the closest minimal pattern subgroup to $H$. By \cref{lemma: maximal minimal patterns} we have that $\mathcal{P}(H) \leq \mathcal{P}$ and by monotony, $G_{\mathcal{P}(H)} \leq G_\mathcal{P} \leq G_H$, but the groups in the extremes are the same so $G_{\mathcal{P}(H)} = G_\mathcal{P}$. Then, projecting to the first $D$ levels, $\mathcal{P}(H) = \mathcal{P}$.
\end{proof}

The combination of \cref{lemma: maximal minimal patterns} and \cref{lemma: P(H) closest minimal pattern} is very powerful because it allows us to discard every subgroup of $H$ that is not also a subgroup of $\mathcal{P}(H)$.

We conclude the following algorithm: let $L$ be the waiting set. This set will have the subgroups that need to be analyzed after, and write $P$ for the set of minimal patterns already found. Start setting $L = \set{\Aut(T^D)}$ and $P = \emptyset$. Then, while $L$ is not empty, take in $L$ the subgroup $H$ with the biggest order and remove $H$ from $L$. As we saw before, we have an algorithm to compute $\mathcal{P}(H)$, so we just add $\mathcal{P}(H)$ to $P$ in case the subgroup is not already there. Then, consider $M(\mathcal{P}(H))$ the set of maximal subgroups of $\mathcal{P}(H)$. Let $K \in M(\mathcal{P}(H))$. If $K$ is not already in $L$ nor is $K$ included in $\St(1)$, add $K$ to the set $L$ to analyze it in the future. The program ends when $L$ is empty.

\begin{table}[t]
\centering
\begin{tabular}{|c|c|c|c|}
\hline
\textbf{d} & \textbf{D} & \textbf{\begin{tabular}[c]{@{}c@{}}Number of \\ minimal patterns\end{tabular}} & \textbf{Time} \\ \hline
2 & 2 & 6 & 62 ms \\ \hline
2 & 3 & 60 & 0.653 s \\ \hline
2 & 4 & 4544 & 13m 27s \\ \hline
3 & 2 & 588 & 17.427 s \\ \hline
\end{tabular}
\caption{Number of patterns found and time spent.}
\label{table: time calculating minimal pattern subgroups}
\end{table}

\cref{table: time calculating minimal pattern subgroups} shows the number of minimal pattern subgroups found and how long it took to find the minimal pattern subgroups of depth $D$ in a $d$-regular tree. The first three rows where also calculated in \cite{Bondarenko2014}, and the results obtained are the same.

\section{Hausdorff dimension, fractality and torsion}
\label{section: Hausdorff dimension fractality and torsion}

\subsection{Hausdorff dimension}

The Hausdorff dimension of a group is typically used in geometric group theory to have another way to measure the size of a group with respect to a group that contains it. The study of the Hausdorff dimension on profinite groups was initiated by Abercrombie in \cite{Abercrombie1994} and by Barnea and Shalev in \cite{BarneaShalev1997}, and in the case of groups acting on trees, many interesting results have been proved (see for example \cite{AbertVirag2004, FariñaAsategui2024, Siegenthaler2008}). 

Let $T$ be a spherically homogeneous infinite rooted tree and $H \leq G \leq \Aut(T)$. The \textit{relative Hausdorff dimension} of $H$ in $G$ is defined as the number $$\mathcal{H}_G(H) := \liminf_{n \rightarrow +\infty} \frac{\log(\abs{\pi_n(H)})}{\log(\abs{\pi_n(G)})}.$$

The \textit{Hausdorff dimension} of $H$ is defined as $$\mathcal{H}(H) := \mathcal{H}_{\Aut(T)}(H).$$

Given $\mathcal{P}$ a minimal pattern subgroup of depth $D$ in a $d$-regular tree $T$, define $$\St_\mathcal{P}(D-1) := \pi_D(\St_{G_\mathcal{P}}(D-1)).$$

\begin{Proposition}
\label{proposition: number elements finite type and Hausdorff dimension}
Let $T$ be a $d$-regular tree, $D$ a natural number and $\mathcal{P} \leq \Aut(T^D)$ a minimal pattern subgroup. Then 
\begin{equation}
\abs{\pi_n(G_\mathcal{P})} = \abs{\mathcal{P}} \abs{\St_\mathcal{P}(D-1)}^{d+\dots+d^{n-D}}
\label{equation: finitetype pin(G)}
\end{equation}
for all $n \geq D$. In particular, its Hausdorff dimension is
\begin{equation}
\mathcal{H}(G_\mathcal{P}) = \frac{1}{d^{D-1}} \frac{\log(\abs{\St_\mathcal{P}(D-1)})}{\log(d!)}.
\label{equation: finite type Hausorff dimension}
\end{equation}
\end{Proposition}

\begin{proof}
The formula of $\abs{\pi_n(G_\mathcal{P})}$ is deduced in \cite[Proposition 1]{Bondarenko2014}. To calculate the Hausdorff dimension, recall that $\Aut(T)$ is a group of finite type with depth $1$ and $\mathcal{P} = \Sym(d)$. Then $\St_{\Sym(d)}(1-1) = \Sym(d)$ and therefore
\begin{align*}
\mathcal{H}(G_\mathcal{P}) = \liminf_{n \rightarrow \infty} \frac{\log(\abs{\pi_n(G_\mathcal{P})}}{\log(\abs{\pi_n(\Aut(T))})} = \frac{1}{d^{D-1}} \frac{\log(\abs{\St_\mathcal{P}(D-1)})}{\log(d!)}.
\end{align*}
\end{proof}

We obtain two corollaries from \cref{proposition: number elements finite type and Hausdorff dimension}:

\begin{corollary}
Let $T$ be a $d$-regular tree, $D$ a natural number and $\mathcal{P} \leq \Aut(T^D)$ a minimal pattern subgroup. The following statements are equivalent:

\begin{itemize}
\item $G_\mathcal{P}$ is finite,
\item $\St_\mathcal{P}(D-1)$ is trivial,
\item $G_\mathcal{P}$ is isomorphic to $\mathcal{P}$,
\item $\mathcal{H}(G_\mathcal{P}) = 0$.
\end{itemize}
\label{theorem: finite type finiteness condition}
\end{corollary}

\begin{proof}
A profinite group as $G_\mathcal{P}$ is finite if and only if there exists $n_0 \in \NN$ such that for all $n \geq n_0$ then $\abs{\pi_{n_0}(G_\mathcal{P})} = \abs{\pi_n(G_\mathcal{P})}$. By \cref{equation: finitetype pin(G)}, the sequence $\set{\abs{\pi_n(G_\mathcal{P})}}_{n \in \NN}$ stabilizes if and only if $\St_\mathcal{P}(D-1)$ is trivial and it does stabilize with $n = D$, so $G_\mathcal{P}$ is finite if and only if it is isomorphic to $\mathcal{P}$. By \cref{equation: finite type Hausorff dimension}, this implies that $\mathcal{H}(G_\mathcal{P})$ is zero. Conversely, if $\mathcal{H}(G_\mathcal{P}) = 0$, then $\St_\mathcal{P}(D-1)$ must be trivial.
\end{proof}

\begin{corollary}
Let $T$ be a $d$-regular tree, $D$ a natural number and $\mathcal{P} \leq \Aut(T^D)$ a minimal pattern subgroup. Then $\mathcal{H}(G_\mathcal{P}) = 1$ if and only if $G_\mathcal{P} = \Aut(T)$.
\label{theorem: finite type hausdorff dimension 1}
\end{corollary}

\begin{proof}
The group $\St_\mathcal{P}(D-1)$ is a subgroup of $\pi_D(\St(D-1))$ and this latter is isomorphic to $\Sym(d)^{\mathcal{L}_{D-1}}$. If $\mathcal{H}(G_\mathcal{P}) = 1$, by \cref{equation: finite type Hausorff dimension}, this means that $\St_\mathcal{P}(D-1) = \pi_D(\St(D-1))$. We claim that $\mathcal{P} = \Aut(T^D)$. Suppose by induction that we have proved that $\pi_D(\St(m)) \leq \mathcal{P}$ for $m < D$ and we want to show that $\pi_D(\St(m-1)) \leq \mathcal{P}$. Let $a \in \pi_D(\St(m-1))$ and consider $b \in \pi_D(\St(m))$ such that $b|_x = \pi_{D-1}(a)$ for all $x \in \mathcal{L}_1$. Since $b$ is in the minimal pattern subgroup, it can be extended to an element $g \in G_\mathcal{P}$ such that $\pi_D(g) = b$. Then, $g|_x^D \in \mathcal{P}$.

Notice that $g|_x^D$ equals $a$ on the first $D-1$ levels but not necessarily until level $D$. But since $\pi_D(\St(D-1)) \leq \mathcal{P}$, this can be corrected multiplying by the necessary element. Then $\pi_D(\St(m-1)) \leq \mathcal{P}$. By the inductive process, we obtain that $\pi_D(\St(0)) = \Aut(T^D) \leq \mathcal{P}$.
\end{proof}

\subsection{Fractality}
\label{section: Fractality}

Let $T$ be a $d$-regular tree. Recall the map $\varphi_v$ defined in \cref{equation: map varphi v}. We say that a group $G \leq \Aut(T)$ is \textit{fractal} if $G$ is self-similar, level-transitive and $\varphi_v(\st_G(v)) = G$ for all $v$ vertex in $T$. We say that $G$ is \textit{strongly fractal} if $G$ is self-similar, level-transitive and $\varphi_v(\St_G(1)) = G$ for all $v \in \mathcal{L}_1$. Finally we say that a group $G \leq \Aut(T)$ is \textit{super strongly fractal} if $G$ is self-similar, level-transitive and $\varphi_v(\St_G(n)) = G$ for all $v \in \mathcal{L}_n$ and $n \ge 1$. 

Fractality responds to the level of self-repetition that the group has. Fractal properties of self-similar groups were recently related to the dynamics and ergodic properties of self-similar profinite groups (see \cite{Jorge2025}), to the action of random subgroups in profinite groups (see \cite{FariñaRadi2025RandomSubgroups}) and to density problems in number theory (see \cite{FariñaRadi2025}).

The three notions of fractality are not equivalent, as it was proved in \cite[Section 3 and Proposition 4.3]{UriaAlbizuri2016}, but the following implications are true:
$$\text{Super strongly fractal} \Rightarrow \text{Strongly fractal} \Rightarrow \text{Fractal}.$$

The following proposition summarizes necessary and sufficient conditions for a group of finite type to be fractal, strongly fractal or super strongly fractal. For convenience, we set the notation $\varphi^{D-1}_v$ to refer to $\pi_{D-1} \circ \varphi_v$. 

\begin{Proposition}
Let $T$ be a $d$-regular tree, $D$ a natural number and $\mathcal{P} \leq \Aut(T^D)$ a minimal pattern subgroup. 

\begin{enumerate}
\item If $D = 1$, then $G_\mathcal{P}$ is super strongly fractal if and only if $\mathcal{P}$ is transitive. 

\item If $D \geq 2$, then $G_\mathcal{P}$ is fractal if and only if $G_\mathcal{P}$ is transitive on the first level and  there exists $v_0 \in \mathcal{L}_1$ such that $\varphi_{v_0}^{D-1}(\st_{G_\mathcal{P}}(v_0)) = \pi_{D-1}(\mathcal{P})$.

\item If $D \geq 2$, then $G_\mathcal{P}$ is strongly fractal if and only if $G_\mathcal{P}$ is transitive on the first level and there exists $v_0 \in \mathcal{L}_1$ such that $\varphi_{v_0}^{D-1}(\St_{G_\mathcal{P}}(1)) = \pi_{D-1}(\mathcal{P})$.

\item If $D \geq 2$, then $G_\mathcal{P}$ is super strongly fractal if and only if $G_\mathcal{P}$ is transitive on the first level and there exists $v_0 \in \mathcal{L}_{D-1}$ such that we have the equality $\varphi_{v_0}^{D-1}(\St_{G_\mathcal{P}}(D-1)) = \pi_{D-1}(\mathcal{P})$.
\end{enumerate}
\label{proposition: fractal sf and ssf finite type}
\end{Proposition}

\begin{proof}
For (1), if $D = 1$, the group of finite type is the iterated wreath product $W_\mathcal{P}$. In particular, for any $n \in \NN$ and any vertex $v \in \mathcal{L}_n$, we have $\varphi_v(\St_{W_\mathcal{P}}(n)) = W_\mathcal{P}$. So, the group $W_\mathcal{P}$ is super strongly fractal if and only if it is level-transitive and this is if and only if $\mathcal{P}$ is transitive. 

For (2), the direct, if $G_\mathcal{P}$ is fractal and we take any vertex $v_0 \in \mathcal{L}_1$, then $\varphi_{v_0}(\st_{G_\mathcal{P}}(v_0)) = G_\mathcal{P}$. So, if we restrict to the first $D-1$ levels, we obtain the result. 

For the converse, let $g \in G_\mathcal{P}$ and $p = \pi_{D-1}(g)$. Then, there exists $s_1 \in \st_{G_\mathcal{P}}(v_0)$ such that $(s_1)|_{v_0}^{D-1} = p$. Call $g_1 = (s_1)|_{v_0}$. By self-similarity of $G_\mathcal{P}$ we have that $g_1 \in G_\mathcal{P}$ and therefore $g_1^{-1}g \in \St_{G_\mathcal{P}}(D-1)$. Define $s = \delta_{v_0}(g_1^{-1}g)$ where $\delta_{v_0}$ is the funciton defined in \cref{equation: map deltav}. Then $s \in \St_{G_\mathcal{P}}(D)$, the element $s_1s \in \st_{G_\mathcal{P}}(v_0)$ and $$(s_1s)|_{v_0} = s_1 \mid_{v_0} \cdot  s \mid_{v_0} = g.$$ This proves that $\varphi_{v_0}(\st_{G_\mathcal{P}}(v_0)) = G_\mathcal{P}$. Then, if $v$ is another vertex in $\mathcal{L}_1$, by transitivity on the first level, there exists $h \in G_\mathcal{P}$ such that $h(v_0) = v$. This gives that $$\varphi_v(\st_{G_\mathcal{P}}(v)) = h \mid_{v_0} \varphi_{v_0}(\st_{G_\mathcal{P}}(v_0)) (h \mid_{v_0})^{-1} = G_\mathcal{P}.$$

Finally if $v$ is a vertex on level $n$, we proceed by induction on $n$. The case $n = 1$ is the previous paragraph. Then, let $g \in G_\mathcal{P}$, $v \in \mathcal{L}_n$ and write $v = v_1w_{n-1}$ with $v_1 \in \mathcal{L}_1$. So, there exists $g_{n-1} \in \st_G(w_{n-1})$ such that $(g_{n-1})|_{w_{n-1}} = g$. By the hypothesis, there exists $g_n \in \st_G(v_1)$ such that $(g_n)|_{v_1} = g_{n-1}$ and combining everything, we see that $$(g_n)|_v = ((g_n)|_{v_1})|_{w_{n-1}} = (g_{n-1})|_{w_{n-1}} = g$$ and $g_n(v) = g_n(v_1) (g_n)|_{v_1}(w_{n-1}) = v_1 w_{n-1} = v$, so $g_n \in \st_{G_\mathcal{P}}(v)$ as we wanted. The level-transitivity follows from the fact that $G_\mathcal{P}$ acts transitively on the first level. 

For (3), the proof follows the same ideas as the previous point. The direct follows as in (2). For the converse, repeating the process in (2), we conclude that $\varphi_{v_0}(\St_{G_\mathcal{P}}(1)) = G_\mathcal{P}$ and by transitivity on the first level that $\varphi_v(\St_{G_\mathcal{P}}(1)) = G_\mathcal{P}$ for all $v \in \mathcal{L}_1$. Then, since this condition implies fractal and it is transitive on the first level, it is level-transitive by point (2). 

For (4), the direct follows as in (2). For the converse, let $g \in G_\mathcal{P}$ and let $p = \pi_{D-1}(g)$. Then, there exists $s_1 \in \St_{G_\mathcal{P}}(D-1)$ such that $(s_1)|_{v_0}^{D-1} = p$. Call $g_1 = (s_1)|_{v_0}$. By self-similarity of $G_\mathcal{P}$ we have that $g_1 \in G_\mathcal{P}$ and therefore $g_1^{-1}g \in \St_{G_\mathcal{P}}(D-1)$. Define again $s = \delta_{v_0}(g_1^{-1}g)$. Then $s \in \St_{G_\mathcal{P}}(2D-2)$, the element $s_1s \in \St_{G_\mathcal{P}}(D-1)$ and $$(s_1s)|_{v_0} = s_1 \mid_{v_0} \cdot s \mid_{v_0} = g.$$ This proves that $\varphi_{v_0}(\St_{G_\mathcal{P}}(D-1)) = G_\mathcal{P}$. 

Write $v_0 = w_0 u_0$ with $u_0 \in \mathcal{L}_n$ and $n \leq D-1$. Then, given $g \in G_\mathcal{P}$, by the previous paragraph we know there exists $s \in \St_{G_\mathcal{P}}(D-1)$ such that $$g = s|_{v_0} = s|_{w_0 u_0} =  (s|_{w_0})|_{u_0}. $$

By self-similarity, we know that $s|_{w_0} \in \St_{G_\mathcal{P}}(n)$ and so $G_\mathcal{P} \leq \varphi_{u_0}(\St_{G_\mathcal{P}}(n))$. Again by self-similarity, we know that $G_\mathcal{P} \geq \varphi_{u_0}(\St_{G_\mathcal{P}}(n))$, concluding that for any level $n \leq D-1$, there exists $u_0 \in \mathcal{L}_n$ such that $G_\mathcal{P} = \varphi_{u_0}(\St_{G_\mathcal{P}}(n))$.

In the particular case that $n = 1$, we have that $G_\mathcal{P}$ is fractal by (2), so $G_\mathcal{P}$ is level-transitive. Therefore, given $v \in \mathcal{L}_n$ with $n \leq D-1$, there exists $h \in G_\mathcal{P}$ such that $h(u_0) = v$ with $u_0 \in \mathcal{L}_n$ and $G_\mathcal{P} = \varphi_{u_0}(\St_{G_\mathcal{P}}(n))$. Applying the same trick used in (2), $$\varphi_v(\St_{G_\mathcal{P}}(n)) = h \mid_{u_0} \varphi_{u_0}(\St_{G_\mathcal{P}}(n)) (h \mid_{u_0})^{-1} = G_\mathcal{P}.$$

Finally, if $v \in \mathcal{L}_n$ with $n > D-1$, we can write it as $v = uw$ with $w \in \mathcal{L}_{D-1}$. Then, by \cref{lemma: St(n+1) = prod St(n)}, $$\varphi_v(\St_{G_\mathcal{P}}(n)) = \varphi_w(\St_{G_\mathcal{P}}(D-1)) = G_\mathcal{P}.$$
\end{proof}

\subsection{Torsion}

We end up this section applying \cref{theorem: torsion measure weakly branch} to conclude the following:

\begin{corollary}
Let $T$ be a $d$-regular tree, $D$ a natural number and $\mathcal{P} \leq \Aut(T^D)$ a minimal pattern subgroup. Then $G_\mathcal{P}$ is torsion if and only if $G_\mathcal{P}$ is finite. If it is not finite and $\mu$ is its Haar measure, then $\mu(\set{g \in G_\mathcal{P}: \abs{g} < \infty}) = 0$. 
\end{corollary}

\begin{proof}
If $G_\mathcal{P}$ is finite, it is torsion. If not, it is weakly regular branch by \cref{lemma: finite type is weakly regular branch}.
\end{proof}

Although the measure is zero, the group $G_\mathcal{P}$ may not be torsion-free. For example, in the iterated wreath products $W_\mathcal{P}$, we can embed $\pi_n(W_\mathcal{P}) \hookrightarrow W_\mathcal{P}$ by sending $h$ to $\psi_n^{-1}((1,\dots,1) \, h)$. Which is harder is to prove that we can in fact find groups of finite type that are torsion-free:

\begin{Example}
In the $d$-regular tree $T$ with $D = 2$, consider $\sigma = (1,\dots,d)$ the $d$-cycle  in $\Sym(d)$ and the pattern subgroup $\mathcal{P} = \<\alpha>$, where $\alpha \in \Aut(T^D)$ is the element whose portrait is described in \cref{figure: portrait of torsion-free group}.

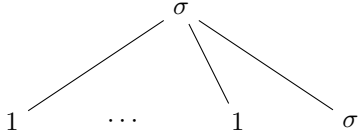
\begin{figure}[h!]
    \centering
    \begin{tikzpicture}
    \node {$\sigma$}
        child {node {$1$}}
        child {node {$\cdots$} edge from parent[draw=none]}
        child {node {$1$}}        
        child {node {$\sigma$}};
    \end{tikzpicture}
    \caption{Portrait of the element $\alpha$.}
    \label{figure: portrait of torsion-free group}
\end{figure}

By \cref{lemma: minimal pattern generating set}, the subgroup $\mathcal{P}$ is a minimal pattern subgroup. The element $\alpha$ has order $d^2$ and its powers can be explicitly described as $$\alpha^{ad+b} = (\sigma^a, \dots, \sigma^a, \underbrace{\sigma^{a+1}, \dots, \sigma^{a+1}}_{\text{last $b$ vertices}}) \sigma^b$$ where $a,b \in \set{0,\dots,d-1}$. 

The claim is that $G_\mathcal{P}$ is torsion-free. Suppose there exists $g \in G_\mathcal{P} \setminus \set{1}$ with finite order. Observe that by \cref{proposition: number elements finite type and Hausdorff dimension}, the primes that divide $\pi_n(G_\mathcal{P})$ for all $n \in \NN$ are the divisors of $d$. Therefore, if $g$ has finite order, taking powers of $g$, we may assume that $\abs{g} = p$ with $p$ a prime divisor of $d$. Since $g \neq 1$, there exists $n \in \NN$ such that $g \in \St_G(n) \setminus \St_G(n+1)$, which in particular implies that there is a vertex $v \in \mathcal{L}_n$, such that $g_1 = g|_v$ acts non-trivially on the first level. On the other hand, $$1 = (g^p)|_v = g_1^p,$$ so $g_1$ also has order $p$, which in particular implies that $\pi_1(g_1)$ and $\pi_2(g_1)$ have order $p$. Since $G_\mathcal{P}$ is of finite type, there must exist $a,b \in \set{0,\dots,d-1}$ with $b \neq 0$ such that $\pi_2(g_1) = \alpha^{ad+b}$ and then $\pi_1(g_1) = \sigma^b$. Therefore, $1 = \pi_1(g_1^p) = \sigma^{pb}$, implying that $pb \equiv 0 \pmod{d}$, or equivalently, there exists $k \in \NN$ such that $pb = kd$. Since $0< b < d$, then $0<k<p$. Similarly, $1 = \pi_2(g_1^p) = \alpha^{apd+pb}$, so $$apd+pb = (ap+k)d \equiv 0 \pmod{d^2}$$ and so $ap+k \equiv 0 \pmod{d}$. Reducing modulo $p$, we obtain $k \equiv 0 \pmod{p}$ which is a contradiction since $0<k<p$. Hence, the group $G_\mathcal{P}$ is torsion-free.
\label{example: torsion free}
\end{Example}

\section{Topological finite generation, just-infiniteness and strong completeness}
\label{section: Topological finite generation just infiniteness and strong completeness}

In this section we will prove \cref{Theorem: finite type equivalences} and its corollaries presented in the Introduction. 

\subsection{Strongly completeness}

Let $T$ be a spherically homogeneous infinite rooted tree and $G$ a subgroup of $\Aut(T)$. Since $\Aut(T)$ is a profinite group with $\set{\St(n): n \in \NN}$ as a base of open neighborhoods of the identity, then the closure of $G$ in $\Aut(T)$ is isomorphic to $$\overline{G} = \varprojlim G/\St_G(n)$$ by \cite[Corollary 1.1.8]{RibesZalesskii2000}. In the case that $\RiSt_G(n) \leq_f G$ for all $n \in \NN$, then $$\widetilde{G} := \varprojlim G/\RiSt_G(n)$$ is also a profinite group and since $\RiSt_G(n) \leq \St_G(n)$ we have a well-defined surjective homomorphism $$\psi_1: \widetilde{G} \rightarrow \overline{G}.$$ The kernel of $\psi_1$ is called the \textit{rigid kernel}.

Another standard construction in group theory is the \textit{profinite completion} of a group, defined as the inverse limit $$\widehat{G} := \varprojlim G/N,$$ where the limit runs over all normal subgroups $N$ of finite index in $G$. In the case that $\RiSt_G(n) \leq_f G$ for every $n \in \NN$, the family of rigid stabilizers is a subset of the family of normal subgroups of finite index of $G$, and therefore we have a well-defined surjective homomorphism $$\psi_2: \widehat{G} \rightarrow \widetilde{G}.$$ The kernel of $\psi_2$ is called the \textit{branch kernel}. Composing the last two maps, we obtain an epimorphism $$\psi_3 = \psi_1 \circ \psi_2: \widehat{G} \rightarrow \overline{G}.$$ The kernel of $\psi_3$ is called the \textit{congruence kernel}. 

An abstract group $G$ is said to have the \textit{congruence subgroup property} if the map $\psi_3$ is an isomorphism. In the case that $G$ is a profinite group, we say that $G$ is \textit{strongly complete}.

The following equivalences are well-known results:

\begin{Lemma}
Let $T$ be a spherically homogeneous infinite rooted tree and a subgroup $G \leq \Aut(T)$. Then
\begin{enumerate}
\item The rigid kernel is trivial if and only if for any $n \in \NN$, there exists $m \in \NN$ such that $\RiSt_G(n) \geq \St_G(m)$.
\item The branch kernel is trivial if and only if for any $N \lhd_f G$, there exists $n \in \NN$ such that $N \geq \RiSt_G(n)$.
\item The congruence kernel is trivial if and only if for any $N \lhd_f G$, there exists $n \in \NN$ such that $N \geq \St_G(n)$.
\end{enumerate}
\label{lemma: equivalences inclusion profinite kernels}
\end{Lemma}

\subsection{Topological finite generation of pro-nilpotent groups}

Let $G$ be an abstract group. Define $\gamma_1(G) = G$ and $\gamma_r(G) = [\gamma_{r-1}(G),G]$.  We say that $G$ is \textit{nilpotent} if there exists $r \in \NN$ such that $\gamma_r(G) = 1$. We are interested in the following properties of nilpotent groups:

\begin{Lemma}
Let $G$ be an abstract group. 
\begin{enumerate}
\item If $G$ is a finite $p$-group then $G$ is nilpotent.
\item If $G$ is nilpotent and $S$ is a set of $G$ such that the projection of $S$ to $G/G'$ is a generating set, then $S$ generates $G$.
\end{enumerate}
\label{lemma: properties nilpotent groups}
\end{Lemma}

\begin{proof}
For (1), see \cite[Theorem 7.2, Chapter 2]{Hungerford1974}. For (2), see \cite[Lemma 5.9, page 350]{MagnusKarrasSolitar1966}.
\end{proof}

Let $I$ be a directed set, $\set{G_i}_{i \in I}$ a family of compatible groups and $G = \varprojlim_I G_i$ its inverse limit with transition maps $\pi_i : G \rightarrow G_i$. We say that $G$ is \textit{pro-nilpotent} if each $G_i$ is nilpotent. We also define the \textit{congruence topology} of $G$ by declaring $\ker(\pi_i)$ a basis of open neighborhoods of the identity for every $i \in I$.

\begin{Definition}
Let $G = \varprojlim_I G_i$ be an inverse limit of groups with surjective transition maps $\pi_i : G \rightarrow G_i$. Let $i \in I$, an element $g \in G$ and $g_i \in G_i$. We say that $g$ \textit{extends} or is an \textit{extension} of $g_i$ if $\pi_i(g) = g_i$. Also, if $S$ is a subset of $G$ and $S_i$ is a subset of $G_i$, we say that $S$ \textit{extends} or is an \textit{extension} of $S_i$ if $\pi_i: S \rightarrow S_i$ is bijective. 
\end{Definition}

\begin{Example}
In the context of automorphism of trees, if $T$ is a spherically homogeneous tree, we have that $\Aut(T) = \varprojlim \Aut(T^n)$, so here if $\alpha \in \Aut(T^n)$, the extensions of $\alpha$ are going to be automorphisms of $\Aut(T)$ such that they equal $\alpha$ when we restrict them to level $n$.
\end{Example}

The following is an easy result from group theory:

\begin{Lemma}
Let $A$ be a group, $\pi: A \rightarrow \pi(A)$ a surjective map, $B \lhd A$ and call $p_1: A \rightarrow A/B$ and $p_2: \pi(A) \rightarrow \pi(A)/\pi(B)$ the natural projections. Then we have a commutative diagram

\begin{equation}
\centering
\begin{tikzcd}
A \arrow{rr}{p_1} \arrow{dd}{\pi} & & A/B \arrow{dd}{\varphi} \\
 & & \\
\pi(A) \arrow{rr}{p_2} & & \pi(A)/\pi(B) \\
\end{tikzcd}
\end{equation}
where $\varphi$ is a well-defined surjective map and $\ker(\varphi) \simeq \ker(\pi)/(\ker(\pi) \cap B)$. Moreover, we have the following equivalences:

\begin{enumerate}
\item $\varphi$ is an isomorphism,
\item $\ker(\pi) \leq B$,
\item $[A:B] = [\pi(A):\pi(B)]$. 
\end{enumerate}
\label{lemma: map defined in abelianizations}
\end{Lemma}

\begin{proof}
First, since $\pi$ is surjective and $B \lhd A$, then $\pi(B) \lhd \pi(A)$, ensuring that $\varphi$ is well-defined. Moreover, since all the maps involved are surjective, $\varphi$ is also surjective.

For the kernel, we have the following equivalences:
\begin{align*}
g B \in \ker(\varphi) \Longleftrightarrow \varphi(gB) = p_2 \, \pi(g) = 1 \Longleftrightarrow  \pi(g) \in \pi(B) \Longleftrightarrow \\ g \in \pi^{-1}(\pi(B)) = \<\ker(\pi),B> = \ker(\pi)B.
\end{align*}
Hence $$\ker(\varphi) = \ker(\pi)B/ B \simeq \ker(\pi)/(\ker(\pi) \cap B)$$ by the second isomorphism theorem.

For the equivalences, the first two follow from the form of the kernel of $\varphi$. Indeed, $$\ker(\varphi) = 1 \Longleftrightarrow \ker(\pi) \cap B = \ker(\pi) \Longleftrightarrow \ker(\pi) \leq B.$$

Finally, since we already know $\varphi$ is a surjective morphism of groups, $$\varphi \text{ is an isomorphism } \Longleftrightarrow \abs{A/B} = \abs{\pi(A)/\pi(B)}.$$
\end{proof}

Given an abstract group $G$, define $g(G) = \min \set{\# S: \<S> = G}$. The number $g(G)$ will be infinite if $G$ is not finitely generated. If $G$ is a topological group, we define $g(G)$ to be the minimal number of generators of a dense subgroup. A topological group $G$ is said \textit{topologically finitely generated} if $g(G) < \infty$.

We are now ready to prove the main result of this subsection:

\begin{theorem}
Let $\set{G_n}_{n \in \NN}$ be a family of finite compatible groups. Let $G = \varprojlim_{\NN} G_n$ be a pro-nilpotent profinite group with surjective transition maps $\pi_n : G \rightarrow G_n$. Denote $S_n = \pi_n(\ker(\pi_{n-1}))$. If there exists $n_0 \in \NN$ such that $S_n \leq G_n'$ for all $n \geq n_0$, then $G$ is topologically finitely generated. Moreover, if $T_{n_0}$ is a generating set of $G_{n_0}$, then any extension $T \subseteq G$ of $T_{n_0}$ is a topological generating set of $G$ in the congruence topology. In particular, $g(G) = g(G_{n_0})$.
\label{theorem: tfg pro-nilpotent}
\end{theorem}

\begin{proof}
Take $T \subseteq G$ any extension of $T_{n_0}$ and define $T_n = \pi_n(T)$. Our first step is to prove that $T_n$ is a generating set of $G_n$ for every $n \geq n_0$. We proceed by induction on $n$. The case $n = n_0$ follows by definition of $T_{n_0}$. Then, consider the diagram 
\begin{equation}
\centering
\begin{tikzcd}
G_{n+1} \arrow{rr}{ab_{n+1}} \arrow{dd}{{}_{n+1} \pi_n} & & G_{n+1}^{ab}  \arrow{dd}{{}_{n+1} \varphi_n} \\
 & & \\
G_n \arrow{rr}{ab_n} & & G_n^{ab}  \\
\end{tikzcd}
\label{diag_Gn_abelianizations}
\end{equation}

where ${}_{n+1}\pi_n$ is the transition map. By \cref{lemma: map defined in abelianizations}, ${}_{n+1}\varphi_n$ is well-defined, surjective and its kernel is $$\ker({}_{n+1}\varphi_n) \simeq S_{n+1}/(S_{n+1} \cap G_{n+1}').$$ Since $n \geq n_0$, by hypothesis $S_{n+1} \leq G_{n+1}'$ and therefore ${}_{n+1}\varphi_n$ is an isomorphism. 

On the other hand, $$ab_{n+1}(T_{n+1}) = {}_{n+1} \varphi_n^{-1} \,\, ab_n \,\, {}_{n+1} \pi_n(T_{n+1}) = {}_{n+1}\varphi_n^{-1} \,\, ab_n(T_n)$$ and as ${}_{n+1}\varphi_n^{-1}$ and $ab_n$ are surjective, then $ab_{n+1}(T_{n+1})$ is a generating set of $G_{n+1}^{ab}$. Finally, using that $G$ is pro-nilpotent, the group $G_{n+1}$ is nilpotent and by \cref{lemma: properties nilpotent groups}, the set $T_{n+1}$ is a generating set of $G_{n+1}$ as we wanted.

To prove that $T$ generates $G$ topologically, let $g \in G$ and $n \in \NN$. Then, there exists a word $w_n$ in $T_n$ such that $\pi_n(g) = w_n(T_n)$. Since $\pi_n$ is a homomorphism such that $\pi_n(T) = T_n$, we have $\pi_n(w_n(T)) = w_n(T_n)$ and consequently we obtain $g w_n(T)^{-1} \in \ker(\pi_n)$. Since $n$ is arbitrary and $\set{\ker(\pi_n)}_{n \in \NN}$ is a basis of neighborhoods of the identity, we obtain a sequence with elements in $T$ converging to $g$. Then $G = \overline{\<T>}$.

Finally, we prove that $g(G) = g(G_{n_0})$. Take $T_{n_0}$ to be a minimal generating set of $G_{n_0}$ and $T$ an extension of $T_{n_0}$. Since the map $\pi_{n_0}$ is surjective, we have the inequality $g(G) \geq g(G_{n_0}) = \# T_{n_0}$ by choice of $T_{n_0}$. On the other hand, $T$ is an extension of $T_{n_0}$ and consequently generates $G$ topologically. Therefore, $\# T = \# T_{n_0} = g(G_{n_0}) \geq g(G)$, proving the equality.
\end{proof}

\cref{theorem: tfg pro-nilpotent} applies for example when $G$ is a closed subgroup of $\Aut(T)$ for $T$ an infinite spherically homogeneous rooted tree, where in this case $G_n = \pi_n(G)$ and $S_n = \pi_n(\St_G(n-1))$. If each $\pi_n(G)$ is a $p$-group for a certain prime $p$, by \cref{lemma: properties nilpotent groups}, we have that $G$ is pro-nilpotent and \cref{theorem: tfg pro-nilpotent} can be applied.

In the case of groups of finite type, by \cref{proposition: number elements finite type and Hausdorff dimension}, the group $G_\mathcal{P}$ is a pro-$p$ group if and only if $\mathcal{P}$ is a $p$-group. 

As a corollary of \cref{theorem: tfg pro-nilpotent}, we obtain the following result, well-known for pro-$p$ groups.

\begin{corollary}
If $G = \varprojlim_{\NN} G_n $ is pro-nilpotent and just-infinite, then $G$ is topologically finitely generated.
\end{corollary}

\begin{proof}
If $G$ is abelian, then $G$ is a product of procyclic groups that can only be just-infinite if $G$ is isomorphic to $\ZZ_p$ for some prime $p$. In particular, $G$ is topologically finitely generated.  

If $G$ is not abelian, then $G' \neq 1$ and so $\overline{G'} \lhd_f G$, namely, there exists $n_0 \in \NN$ such that the map ${}_{n+1} \varphi_n$ in the proof of \cref{theorem: tfg pro-nilpotent} is an isomorphism for all $n \geq n_0$. Applying \cref{lemma: map defined in abelianizations} to the case where $(A,B,\pi) := (G_n, G_n', {}_n \pi_{n-1})$, we obtain $G_n' \geq S_n = \pi_n(\ker(\pi_{n-1}))$ for all $n > n_0$ and therefore $G$ is topologically finitely generated by \cref{theorem: tfg pro-nilpotent}.
\end{proof}

\subsection{Proof of \cref{Theorem: finite type equivalences}}

We start collecting the results already known that will help us to prove \cref{Theorem: finite type equivalences}. 

\begin{Proposition}[{\cite[Proposition 7]{Bondarenko2014}}]
Let $T$ be a $d$-regular tree, $D$ a natural number and a minimal pattern subgroup $\mathcal{P} \leq \Aut(T^D)$. If there exists $n \geq D$ such that $\pi_n(\St_{G_\mathcal{P}}(n-1)) \nleq \pi_n(G_\mathcal{P})'$, then $G_\mathcal{P}$ is not topologically finitely generated.
\label{proposition: finite type not tfg}
\end{Proposition}

\begin{theorem}[{\cite[Theorem 3]{Bondarenko2014}}]
Let $T$ be a $d$-regular tree, $D$ a natural number and a minimal pattern subgroup $\mathcal{P} \leq \Aut(T^D)$. If $G_\mathcal{P}$ acts level-transitively and there exists a natural number $n_0 \geq D$ such that $\pi_{n_0}(\St_{G_\mathcal{P}}(n_0-1)) \leq \pi_{n_0}(\St_{G_\mathcal{P}}(D-1))'$, then $G_\mathcal{P}$ is topologically finitely generated.
\label{theorem: finite type tfg}
\end{theorem}

In \cite{Bondarenko2014}, the converse of \cref{theorem: finite type tfg} is also claimed. However, after revision of the author of this article with Grigorchuk and in a private communication with one of the authors of \cite{Bondarenko2014}, it was concluded that the argument of the converse was wrong.

The following results are related to the strongly completeness:

\begin{theorem}[{\cite[Theorem 1.1]{NikolovSegal2007}}]
Every topologically finitely generated profinite group is strongly complete.
\label{theorem: Segal Nikolov tfg CSP}
\end{theorem}

\cref{theorem: Segal Nikolov tfg CSP} is known to use the classification of finite simple groups (see \cite{Gorenstein2023}). In the case of profinite pro-$p$ groups, there is a much simpler argument due to Serre:

\begin{theorem}[{\cite[Chapter 1, 4.2]{Serre1997}}]
Every topologically finitely generated profinite pro-$p$ group is strongly complete.
\label{theorem: Serre tfg CSP}
\end{theorem}

The following result is about stabilization of indices of weakly regular branch groups:

\begin{Lemma}[{\cite[Lemma 3.1]{Fariña2025}}]
Let $T$ be a $d$-regular tree and $G,K \leq \Aut(T)$ such that $G$ is self-similar and $K_1 \leq K \lhd G$. If there exists $n_0 \in \NN$ such that $[\pi_{n_0}(G): \pi_{n_0}(K)] = [\pi_{n_0+1}(G): \pi_{n_0+1}(K)]$, then $$[\pi_{n_0}(G): \pi_{n_0}(K)] = [\pi_n(G): \pi_n(K)]$$ for all $n \geq n_0$.
\label{lemma: stabilizabition indices regular branch}
\end{Lemma}

The following result is a result in homological algebra:

\begin{theorem}[{\cite[Proposition 2.2.4]{RibesZalesskii2000}}]
In the category of inverse systems of profinite groups over the same directed set, the functor $\varprojlim$ is exact.
\label{theorem: inverse limit functor is exact}
\end{theorem}

From \cref{theorem: inverse limit functor is exact}, we obtain the following corollary that we will use along the article:

\begin{Lemma}
\label{lemma: inverse limit closed abelianization}
Let $\mathcal{P}$ be a minimal pattern subgroup of depth $D$ in a $d$-regular tree $T$ and $G = G_\mathcal{P}$ the associated group of finite type. Then $$\St_G(D-1)^{\overline{ab}} \simeq \varprojlim \pi_n(\St_G(D-1))^{ab}.$$
\end{Lemma}

\begin{proof}
For each $n \in \NN$ we have the short exact sequence $$1 \rightarrow \pi_n(\St_G(D-1)') \rightarrow \pi_n(\St_G(D-1)) \rightarrow \pi_n(\St_G(D-1))^{ab} \rightarrow 1.$$

By \cref{theorem: inverse limit functor is exact}, $$1 \rightarrow \varprojlim \pi_n(\St_G(D-1)') \rightarrow \varprojlim \pi_n(\St_G(D-1)) \rightarrow \varprojlim \pi_n(\St_G(D-1))^{ab} \rightarrow 1,$$ and the first two groups are $\overline{\St_G(D-1)'}$ and $\St_G(D-1)$ respectively, so $$\St_G(D-1)^{\overline{ab}} = \St_G(D-1)/\overline{\St_G(D-1)'} \simeq \varprojlim \pi_n(\St_G(D-1))^{ab}.$$ 
\end{proof}

We are now ready to prove the following theorem, that will have \cref{Theorem: finite type equivalences} as a corollary:

\begin{theorem}
Let $T$ be a $d$-regular tree, $D$ a natural number and $\mathcal{P} \leq \Aut(T^D)$ a minimal pattern subgroup. Suppose that $G = G_\mathcal{P}$ is level-transitive and $\St_G(D-1)^{\overline{ab}}$ is torsion. Then the following statements are equivalent:

\begin{enumerate}
\item $G$ is just-infinite,
\item $G$ is topologically finitely generated,
\item $G$ is strongly complete,
\item There exists $n_0 \geq D$ such that $$[\pi_{n_0 + 1}(\St_G(D-1)): \pi_{n_0 + 1}(\St_G(D-1))'] = [\pi_{n_0}(\St_G(D-1)): \pi_{n_0}(\St_G(D-1))']$$
\item There exists $n_0 \geq D$ such that $\pi_{n_0}(\St_G(n_0-1)) \leq \pi_{n_0}(\St_G(D-1))'$,
\item $\St_G(D-1)^{\overline{ab}}$ is finite.
\end{enumerate}
Moreover, if $G$ is pro-nilpotent, the following item is also equivalent to the previous ones:
\begin{enumerate}
\item[(7)] $\pi_n(\St_G(n_0-1)) \leq \pi_n(G)'$ for all $n \geq D$.
\end{enumerate}
\label{theorem: finite type all equivalences}
\end{theorem}

\begin{proof}
The diagram of implications is shown in \cref{figure: diagram of implications}.

\begin{figure}[h!]
\centering
\begin{tikzpicture}[node distance=2.5cm, auto]
    \node (1) at (2,1) {1};
    \node (4) at (2,0) {4};
    \node (5) at (1,-1) {5};
    \node (2) at (-1,-1) {2};
    \node (3) at (-2,0) {3};
    \node (6) at (-1,1) {6};
    \node (7) at (-3,-1) {7};

    \draw[->, bend left=30, shorten >=2pt, shorten <=2pt] (6) to node[midway, right] {} (1);
    \draw[->, bend left=0, shorten >=5pt, shorten <=2pt] (1) to node[midway, right] {} (6);
    \draw[->, bend right=30, shorten >=2pt, shorten <=2pt] (6) to node[midway, right] {} (4);   
    \draw[->, bend right=0, shorten >=2pt, shorten <=2pt] (4) to node[midway, right] {} (6);   
    \draw[->, bend left=15, shorten >=2pt, shorten <=2pt] (4) to node[midway, right] {} (5);
    \draw[->, bend left=15, shorten >=2pt, shorten <=2pt] (5) to node[midway, right] {} (2);
    \draw[->, bend left=15, shorten >=2pt, shorten <=2pt] (2) to node[midway, right] {} (3);
    \draw[->, bend left=15, shorten >=2pt, shorten <=2pt] (3) to node[midway, left] {*} (6);
    \draw[->, bend left=15, shorten >=2pt, shorten <=2pt] (2) to node[midway, right] {} (7);
    \draw[->, bend left=15, shorten >=5pt, shorten <=3pt] (7) to node[midway, right] {} (2);
\end{tikzpicture}
\caption{Diagram of implications in the proof of \cref{theorem: finite type all equivalences}.}
\label{figure: diagram of implications}
\end{figure}
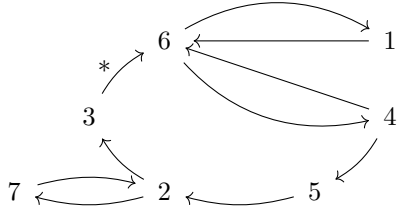

($1 \Leftrightarrow 6$): This is \cref{corollary: finite type just infinite}.

($6 \Rightarrow 4$): If $\St_G(D-1)^{\overline{ab}}$ is finite, by \cref{lemma: inverse limit closed abelianization}, there exists $n_0 \in \NN$ such that for all $n \geq n_0$ then $\pi_n(\St_G(D-1))^{ab} \simeq \pi_{n_0}(\St_G(D-1))^{ab}$. In particular this applies for $n = n_0+1$, obtaining that they have the same order.

($4 \Rightarrow 6$): Since $n_0 \geq D$, we have $$[\pi_{n_0 + 1}(G_\mathcal{P}): \pi_{n_0 + 1}(\St_G(D-1))'] = [\pi_{n_0}(G_\mathcal{P}): \pi_{n_0}(\St_G(D-1))'].$$ By \cref{lemma: commutator and geometric product} and \cref{lemma: St(n+1) = prod St(n)}, $$(\St_G(D-1)')_1 = (\St_G(D-1)_1)' = \St_G(D)' \leq \St_G(D-1)',$$ so applying \cref{lemma: stabilizabition indices regular branch} with $G$ and $\St_G(D-1)'$, we obtain $$[\pi_n(G_\mathcal{P}): \pi_n(\St_G(D-1))'] = [\pi_{n_0}(G_\mathcal{P}): \pi_{n_0}(\St_G(D-1))']$$ for all $n \geq n_0$. Dividing by $$[G_\mathcal{P}: \St_G(D-1)] = [\pi_n(G_\mathcal{P}): \pi_n(\St_G(D-1))] = [\pi_{n_0}(G_\mathcal{P}): \pi_{n_0}(\St_G(D-1))],$$ the result follows.

($4 \Rightarrow 5$): Consider the diagram
\begin{equation*}
\centering
\begin{tikzcd}
\pi_{n_0+1}(\St_G(D-1)) \arrow{rr}{ab_{n_0+1}} \arrow{dd}{\pi} & & \pi_{n_0+1}(\St_G(D-1))^{ab}  \arrow{dd}{\varphi} \\
 & & \\
\pi_{n_0}(\St_G(D-1)) \arrow{rr}{ab_{n_0}} & & \pi_{n_0}(\St_G(D-1))^{ab}
\end{tikzcd}
\end{equation*}

By hypothesis and \cref{lemma: map defined in abelianizations}, we have $$\ker(\pi) = \pi_{n_0+1}(\St_G(n_0)) \leq \pi_{n_0+1}(\St_G(D-1))'.$$

($5 \Rightarrow 2$): This is \cref{theorem: finite type tfg}.

($2 \Rightarrow 3$): This follows from \cref{theorem: Segal Nikolov tfg CSP} (or \cref{theorem: Serre tfg CSP} in the case of pro-$p$ groups). 

($3 \Rightarrow 6$): Suppose that $\St_G(D-1)^{\overline{ab}}$ is not finite. As it is abelian and torsion, then it cannot be topologically finitely generated with respect to the congruence topology. As it is abelian, torsion and profinite, then $$\St_G(D-1)^{\overline{ab}} \simeq \prod_{i = 1}^{\infty} C_{n_i}$$ with $\set{n_i}_{i \in \NN} \subseteq \NN$. Therefore, we can find uncountably many normal subgroups of finite index in $\St_G(D-1)$ containing $\overline{\St_G(D-1)'}$. These subgroups will also have finite index in $G$. 

On the other hand, as $G$ is strongly complete, every normal subgroup of finite index must contain $\St_G(n)$ for some $n \in \NN$ by \cref{lemma: equivalences inclusion profinite kernels}. By the correspondence theorem, we have finitely many normal subgroups containing $\St_G(n)$ since they are in bijection with the normal subgroups of $\pi_n(G)$ that is a finite group. Adding up in all natural numbers $n$, the amount of normal subgroups of finite index in $G$ must be countable. This yields to a contradiction and consequently $\St_G(D-1)^{\overline{ab}}$ must be finite.

($2 \Rightarrow 7$): This is the counter positive of \cref{proposition: finite type not tfg}. 

($7 \Rightarrow 2$): Follows from \cref{theorem: tfg pro-nilpotent} as $G$ is closed by \cref{proposition: finite type is closed self-similar and regular branch} and therefore $G \simeq \varprojlim \pi_n(G)$.
\end{proof}

\begin{Remark}
Notice that the only implication in \cref{theorem: finite type all equivalences} that uses that $\St_G(D-1)^{\overline{ab}}$ is torsion is $(3) \Rightarrow (6)$ as it is indicated in \cref{figure: diagram of implications} with the label *. This in particular shows that if item (4) holds, then the group of finite type is just-infinite, topologically finitely generated and strongly complete.
\label{Remark: item 4 does everything}
\end{Remark}

\subsection{Analysis of when $\St_G(D-1)^{\overline{ab}}$ is torsion}
\label{subsection: Analysis of the property (E)}

Recall that by \cref{conjecture: property (E)}, we expect that $\St_{G_\mathcal{P}}(D-1)^{\overline{ab}}$ is torsion for every group of finite type $G_\mathcal{P}$ of depth $D$. In this subsection, we will analyze equivalences to this conjecture. 

If $G$ is any group, the \textit{exponent} of $G$, denoted by $\exp(G)$ is the smallest natural number $e$ such that $g^e = 1$ for all $g \in G$.

Let $T$ be a $d$-regular tree, $D$ a natural number and $\mathcal{P} \leq \Aut(T^D)$ a minimal pattern subgroup. Recall that $\St_\mathcal{P}(D-1)$ was defined in \cref{section: Hausdorff dimension fractality and torsion} as $\pi_D(\St_{G_\mathcal{P}}(D-1))$.

\begin{Proposition}
\label{proposition: property (E) equivalences}
Let $T$ be a $d$-regular tree, $D$ a natural number and $\mathcal{P} \leq \Aut(T^D)$ a minimal pattern subgroup. Then, the following statements are equivalent:

\begin{enumerate}
\item $\St_{G_\mathcal{P}}(D-1)^{\overline{ab}}$ is torsion,
\item $\St_{G_\mathcal{P}}(D-1)^{\overline{ab}}$ has finite exponent,
\item There exists $S$ a subset of $\St_\mathcal{P}(D-1)$ such that $S$ and their conjugates by elements of $\mathcal{P}$ generate $\St_\mathcal{P}(D-1)$, and for every element $s \in S$ there exists an extension $g_s$ of $s$ to an element in $\St_{G_\mathcal{P}}(D-1)$ such that $g_s$ has finite order in $\St_{G_\mathcal{P}}(D-1)^{\overline{ab}}$.
\end{enumerate}
\end{Proposition}

\begin{proof}
($2 \Rightarrow 1)$ Trivial. 

($1 \Rightarrow 3)$ Take $S = \St_\mathcal{P}(D-1)$ and in particular, any extension of an element in $\St_\mathcal{P}(D-1)$ will have finite order in $\St_{G_\mathcal{P}}(D-1)^{\overline{ab}}$. 

($3 \Rightarrow 2)$ We start proving that if $\alpha \in \St_\mathcal{P}(D-1)$, then there exists $g_\alpha$ extension of $\alpha$ to an element in $\St_{G_\mathcal{P}}(D-1)$ such that $g_\alpha$ has finite order in $\St_{G_\mathcal{P}}(D-1)^{\overline{ab}}$. Firstly, conjugates of elements in $S$ satisfy it. Indeed, if $\gamma \in \mathcal{P}$, there exists $g \in G_\mathcal{P}$ extending $\gamma$ because $\mathcal{P}$ is a minimal pattern subgroup. Then $g_{\gamma s \gamma^{-1}} := g g_s g^{-1}$ is an extension of $\gamma s \gamma^{-1}$. The extension has finite order in $\St_{G_\mathcal{P}}(D-1)^{\overline{ab}}$ because $\overline{\St_{G_\mathcal{P}}(D-1)'} \lhd G_\mathcal{P}$ and it is in the same conjugacy class as $g_s$. Then, if $\alpha$ is any element of $\St_\mathcal{P}(D-1)$, we can write $\alpha = s_1 \dots s_r$, where each $s_i$ is either an element of $S$ or a conjugate of an element in $S$. Define $g_\alpha = g_{s_1} \dots g_{s_r}$. If $$o := \lcm\set{\abs{g_s \, \overline{\St_{G_\mathcal{P}}(D-1)'}}: s \in S},$$ then $$g_\alpha^o \equiv g_{s_1}^o \dots g_{s_r}^o \equiv 1 \pmod{\overline{\St_{G_\mathcal{P}}(D-1)'}}.$$ This proves the first claim that every element in $\St_\mathcal{P}(D-1)$ has an extension to an element in $\St_{G_\mathcal{P}}(D-1)$ with finite order in $\St_{G_\mathcal{P}}(D-1)^{\overline{ab}}$.

Now, set $e = \lcm\set{\exp(\St_\mathcal{P}(D-1)^{ab}), o}$. By \cref{lemma: inverse limit closed abelianization}, we have $$\St_{G_\mathcal{P}}(D-1)^{\overline{ab}} \simeq \varprojlim \pi_n(\St_{G_\mathcal{P}}(D-1))^{ab}.$$ Therefore, $$\exp(\St_{G_\mathcal{P}}(D-1)^{\overline{ab}}) = \lim_{n \rightarrow +\infty} \exp(\pi_n(\St_{G_\mathcal{P}}(D-1))^{ab}),$$ where we also have $$\exp(\pi_n(\St_{G_\mathcal{P}}(D-1))^{ab}) \,\, | \,\, \exp(\pi_{n+1}(\St_{G_\mathcal{P}}(D-1))^{ab})$$ for all $n \in \NN$.

So, in order to prove the result, we will apply induction on $n$. 

If $n = D$, we have $$\pi_n(\St_{G_\mathcal{P}}(D-1))^{ab} = \St_\mathcal{P}(D-1)^{ab}$$ and clearly $$\exp(\St_\mathcal{P}(D-1)^{ab}) \mid e.$$

Assume the result holds until level $n$ and let $g \in \pi_{n+1}(\St_{G_\mathcal{P}}(D-1))$. Then, the element $g$ is the extension of some $\alpha \in \St_\mathcal{P}(D-1)$. Therefore, there exists $s \in \pi_{n+1}(\St_{G_\mathcal{P}}(D))$ such that $gs = \pi_{n+1}(g_\alpha)$ and
\begin{align}
1 = \pi_{n+1}(g_\alpha^e) = (gs)^e \equiv g^e s^e \pmod{\pi_{n+1}(\St_{G_\mathcal{P}}(D-1)')}.
\label{equation: (gs)^e}
\end{align}

Since $G_\mathcal{P}$ is a group of finite type of depth $D$, by \cref{lemma: St(n+1) = prod St(n)} we have the equality $\St_{G_\mathcal{P}}(D) =  (\St_{G_\mathcal{P}}(D-1))_1$. Then, by \cref{lemma: commutator and geometric product}, $$\St_{G_\mathcal{P}}(D)' = (\St_{G_\mathcal{P}}(D-1)_1)' = (\St_G(D-1)')_1,$$ and taking closures,
\begin{align}
\overline{\St_{G_\mathcal{P}}(D-1)'} \geq \overline{\St_{G_\mathcal{P}}(D)'} = \left(\overline{\St_{G_\mathcal{P}}(D-1)'} \right)_1.
\label{equation: commutator StG(D-1)}
\end{align}

Write $s = (s_1,\dots,s_d)$ with $s_i \in \pi_n(\St_{G_\mathcal{P}}(D-1))$ for all $i = 1,\dots,d$. By inductive hypothesis, there exist $s_1',\dots,s_d' \in \pi_n(\St_{G_\mathcal{P}}(D-1)')$ such that $s_i^e = s_i'$ and then by \cref{equation: commutator StG(D-1)}, we obtain that $s' = (s_1',\dots,s_d') \in \pi_{n+1}(\St_{G_\mathcal{P}}(D-1)')$ and that $$s^e = (s_1^e,\dots,s_d^e) = s'.$$

Replacing in \cref{equation: (gs)^e}, we finally conclude that $$g^e \equiv 1 \pmod{\pi_{n+1}(\St_{G_\mathcal{P}}(D-1)')}.$$ as we wanted.
\end{proof}

\subsection{Corollaries and examples}

In this subsection, we will prove \cref{Corollary: IWP equivalence} and \cref{Corollary: closure Hanoi towers joo}. Recall that $N\wr_A H$ denotes the \textit{wreath product} of the groups $H$ and $N$, indexed by the set $A$ and with respect to an action $H \curvearrowright A$. 

\begin{Lemma}[{\cite[Proposition 4, page 214]{delaHarpe2000}}]
Let $H$ and $N$ be two groups, $A$ a finite set and $\varphi$ a left action of $H$ over $A$, faithful and transitive. Then, we have $(N \wr_A H)^{ab} \simeq N^{ab} \times H^{ab}$.
\label{lemma: abelianiation wreath products}
\end{Lemma}

We are now ready to prove \cref{Corollary: IWP equivalence}:

\begin{corollary}
Let $d \geq 2$ and $\mathcal{P}$ be a transitive subgroup of $\Sym(d)$. Then, the following statements are equivalent: 

\begin{enumerate}[(i)]
\item $W_\mathcal{P}$ is just-infinite,
\item $W_\mathcal{P}$ is topologically finitely generate,
\item $W_\mathcal{P}$ is strongly complete,
\item $\mathcal{P} = \mathcal{P}'$. 
\end{enumerate}

In particular, the whole group of automorphisms of the tree does not have any of these properties. 
\label{corollary: IWP equivalence}
\end{corollary}

\begin{proof}
Since $W_\mathcal{P}$ is an iterated wreath product, then $D = 1$ and $\St_{\mathcal{P}}(D-1) = \mathcal{P}$. So, given $\sigma \in \mathcal{P}$, the element $g_\sigma = (1,\dots,1)\sigma \in W_\mathcal{P}$, extends $\sigma$ and has finite order. Therefore, by \cref{proposition: property (E) equivalences}, the quotient $\St_{W_\mathcal{P}}(D-1)^{\overline{ab}} = W_\mathcal{P}^{\overline{ab}}$ is torsion.

Now, applying \cref{lemma: abelianiation wreath products}, we have $\pi_n(W_\mathcal{P})^{ab} \simeq (\mathcal{P}^{ab})^n$, so $$[\pi_n(\St_{W_\mathcal{P}}(D-1)): \pi_n(\St_{W_\mathcal{P}}(D-1))'] = \abs{\pi_n(W_\mathcal{P})^{ab}} = \abs{\mathcal{P}^{ab}}^n,$$ that stabilizes if and only if $\mathcal{P}^{ab} = 1$.

Since $\Aut(T) = W_{\Sym(d)}$ and $\Sym(d) \neq \Sym(d)'$, the whole group of automorphisms of the tree is neither topologically finitely generated nor just-infinite nor strongly complete.
\end{proof}

The Hanoi towers group $\mathcal{H}$, firstly defined in \cite{GrigNekraSunic2006}, is a group acting on the ternary tree and generated by the elements $a$, $b$ and $c$ where $a = (a,1,1)(2,3)$, $b = (1,b,1)(1,3)$ and $c = (1,1,c)(1,2)$. The group is regular branch over its commutator subgroup (see \cite[Lemma 2.2]{Skipper2018}), so its closure is a group of finite type by \cref{theorem: characterization finite type groups}. In the next proposition, we calculate the depth:

\begin{Proposition}
The closure of the Hanoi towers group is a group of finite type of depth $D = 2$.
\label{proposition: depth Hanoi towers closure}
\end{Proposition}

\begin{proof}
Using GAP \cite{GAP}, we find that $$[\pi_1(\mathcal{H}):\pi_1(\mathcal{H})'] = [\pi_2(\mathcal{H}):\pi_2(\mathcal{H})'] = 2.$$

Using \cref{lemma: stabilizabition indices regular branch} with $(G,K) := (\mathcal{H}, \mathcal{H}')$, we conclude that $[\overline{\mathcal{H}}: \overline{\mathcal{H}'}] = 2$. Using \cref{lemma: map defined in abelianizations} with $(A,B, \pi) := (\overline{\mathcal{H}}, \overline{\mathcal{H}'}, \pi_1)$, we conclude that $\ker(\pi_1) = \St_{\overline{\mathcal{H}}}(1) \leq \overline{\mathcal{H}'}$ and by \cref{theorem: characterization finite type groups}, then $\overline{\mathcal{H}}$ is a group of finite type of depth $2$.   
\end{proof}

By \cite[Theorem 3.32]{Skipper2018} and \cite{Bartholdi2012}, the group $\mathcal{H}$ is neither just-infinite nor does it have the congruence subgroup property. However:

\begin{corollary}
The group $\overline{\mathcal{H}}$ is just-infinite and strongly complete.
\end{corollary}

\begin{proof}
By \cref{proposition: depth Hanoi towers closure}, we need to study $[\pi_n(\St_\mathcal{H}(1)):\pi_n(\St_\mathcal{H}(1))']$ for $n \geq 2$. Using GAP, $$[\pi_2(\St_\mathcal{H}(1)):\pi_2(\St_\mathcal{H}(1))'] = [\pi_3(\St_\mathcal{H}(1)):\pi_3(\St_\mathcal{H}(1))'] = 4,$$ so by \cref{Remark: item 4 does everything}, the group $\overline{\mathcal{H}}$ is just-infinite and strongly complete.
\end{proof}

Consider the group $\IMG(z^2+i)$ acting on a $2$-regular tree $T$ and generated by the elements $a = (1,1)(1,2)$, $b = (a,c)$ and $c = (b,1)$. In \cite{GrigSavchukSunic2007} is proven that $\IMG(z^2+i)$ is regular branch, in \cite{Bondarenko2014} is pointed out that its closure is a group of finite type with depth $D = 4$ and in \cite{Radi2025IMG} is proven that $\IMG(z^2+i)$ is just-infinite and does not have the congruence subgroup property. As the group is just-infinite, then its closure is just-infinite (see \cite[Corollary, page 150]{Grigorchuk2000}) and by \cref{Remark: item 4 does everything}, the closure is strongly complete. Therefore, we conclude:

\begin{corollary}
The group $\IMG(z^2+i)$ does not have the congruence subgroup property but its closure $\overline{\IMG(z^2+i)}$ in $\Aut(T)$ is strongly complete.
\end{corollary}

\section{Automata groups and groups of finite type}
\label{section: Automata groups and groups of finite type}

In this section we will prove \cref{Theorem: closure automaton and finite type} and \cref{Corollary: group closure Grig group}.

An \textit{automaton} is a tuple $\mathcal{A} = (Q, X, \tau)$ where $Q$ is a set that represents the states, $X$ is another set that represents the alphabet and $\tau: Q \times X \rightarrow X \times Q$ is the transition function. An automaton is said to be finite if the number of states in $Q$ is finite. If each state in the diagram has $d$ outgoing arrows and when we gather the labels of these $d$ arrows we have a permutation of $\Sym(d)$, then the automata is called \textit{invertible}. If $\mathcal{A}$ is an invertible automaton, given a state $q \in Q$, we can construct an automorphism $g_q \in \Aut(T)$ by identifying $T$ with the free monoid $X^*$ of finite words in $X$ and iterating the transition function with the words in $X^*$ as the inputs. The automata group generated by the automorphisms $\set{g_q: q \in Q}$ will be denoted $G(\mathcal{A})$. Known examples of automata groups are the first Grigorchuk group \cite{Grig1980} and iterated monodromy groups \cite[Chapter 5]{Self_similar_groups}. Note that automata groups are self-similar by construction.

By \cref{theorem: characterization finite type groups}, we know that if an automata group is regular branch, then its closure is a group of finite type. However, the following corollary of \cref{theorem: tfg pro-nilpotent} shows that with a milder condition, the result is still true:

\begin{theorem}
Let $T$ be a $d$-regular tree, $D$ a natural number and $\mathcal{P} \leq \Aut(T^D)$ a minimal pattern $p$-subgroup. Let $\mathcal{A}$ be a finite invertible automaton acting on an alphabet of $d$ letters. If $G_\mathcal{P}$ is topologically finitely generated and $\pi_D(G(\mathcal{A})) = \mathcal{P}$, then $\overline{G(\mathcal{A})} = G_\mathcal{P}$.
\label{theorem: closure automaton and finite type}
\end{theorem}

\begin{proof}
Suppose $\mathcal{A}$ has $r$ states and let $a_1,\dots,a_r$ be the elements in $G(\mathcal{A})$ generated by these states. Since $\mathcal{A}$ is an automaton, given $i \in \set{1,\dots,r}$, there exist $i_1,\dots,i_d \in \set{1,\dots,r}$ and $\sigma_i \in \Sym(d)$ such that $a_i = (a_{i_1},\dots,a_{i_{d}}) \sigma_i$. By induction, this implies that given $i \in \set{1,\dots,r}$ and $v$ a vertex in $T$, there exists $j = j(i,v) \in \set{1,\dots,r}$ such that $(a_i)|_v = a_j$. Therefore, $(a_i)|_v^D = \pi_D(a_j) \in \mathcal{P}$ as $\pi_D(G(\mathcal{A})) = \mathcal{P}$. This proves that $G(\mathcal{A}) \subseteq G_\mathcal{P}$.

Since $G_\mathcal{P}$ is topologically finitely generated and $\mathcal{P}$ is a $p$-group, then $G_\mathcal{P}$ is a pro-$p$ group and by \cref{theorem: finite type all equivalences}, we have the inclusion $\pi_n(\St_{G_\mathcal{P}}(n-1)) \leq \pi_n(G_\mathcal{P})'$ for all $n \geq D$. So, by \cref{theorem: tfg pro-nilpotent}, any extension of a generator of $\pi_D(G_\mathcal{P}) = \mathcal{P}$ will generate $G_\mathcal{P}$ topologically. Since $\pi_D(G(\mathcal{A})) = \mathcal{P}$, then $\set{\pi_D(a_1),\dots,\pi_D(a_r)}$ generates $\mathcal{P}$, and since $a_1,\dots,a_r \in G_\mathcal{P}$ and they are naturally extensions of $\pi_D(a_1),\dots,\pi_D(a_r)$, we obtain that $\set{a_1,\dots,a_r}$ generates $G_\mathcal{P}$ topologically. Therefore $\overline{G(\mathcal{A})} = G_\mathcal{P}$. 
\end{proof}

From the side of groups of finite type, \cref{theorem: closure automaton and finite type} is useful to present them as the closure of automata groups (as we will do in \cref{section: List of automata groups whose closure is a group of finite type}). From the side of automata groups, we have the following corollary:

\begin{corollary}
Let $T$ be a $d$-regular tree, $D$ a natural number and $\mathcal{P} \leq \Aut(T^D)$ a minimal pattern subgroup. Let $\mathcal{A}$ be a finite invertible automaton acting on a alphabet of $d$ letters. If $\overline{G(\mathcal{A})} = G_\mathcal{P}$ then 
$$\mathcal{H}(G(\mathcal{A})) = \frac{1}{d^{D-1}} \frac{\log(\abs{\St_\mathcal{P}(D-1)})}{\log(d!)}.$$
\label{corollary: branch and Hausdorff dim of G(A)}
\end{corollary}

\begin{proof}
It is enough to observe that $\mathcal{H}(G) = \mathcal{H}(\overline{G})$, and that the Hausdorff dimension of groups of finite type was already computed in \cref{proposition: number elements finite type and Hausdorff dimension}.
\end{proof}

As another consequence of \cref{theorem: closure automaton and finite type}, we will prove \cref{Corollary: group closure Grig group}. Consider the automaton given in the \cref{figure: automaton grigorchuk brother}. The associated automata group is the group $G = \<a,b,c,d>$ generated by $a = (a,c)\sigma$, $b = (d,d)$, $c = (d,b)$ and $d = (a,c)$.

\begin{figure}[h!]
\begin{tikzpicture}[shorten >=1pt,node distance=3cm,on grid,>={Stealth[round]},
    every state/.style={draw=blue!50,very thick,fill=blue!20}, bend angle = 15]

  \node[state] (a)  {$a$};
  \node[state] (b) [right=of a] {$b$};
  \node[state] (c) [below =of a] {$c$};
  \node[state] (d) [below =of b] {$d$};

  \path[->] 
  (a) edge [loop above] node  {1/2} ()       
  (a)  edge node [left] {2/1} (c)
  (b)  edge node [right] {x/x} (d)
  (c)  edge [bend right] node [below] {1/1} (d)
  (c)  edge [bend right] node [left] {2/2} (b)
  (d)  edge [bend right] node [above] {1/1} (a)
  (d)  edge [bend right] node [above] {2/2} (c);
\end{tikzpicture}
\caption{Automaton of the group whose closure is the same as the closure of the first Grigorchuk group. The label $x/x$ means that the input and the output are the same and in both cases we pass to the state $d$.}
\label{figure: automaton grigorchuk brother}
\end{figure}
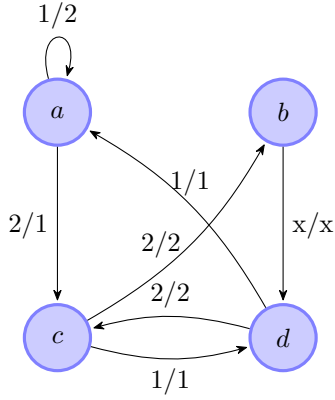

\begin{Proposition}
The group $G$ contains the first Grigorchuk group and it has the same closure. Moreover, the group $G$ is not torsion and the first Grigorchuk group has infinite index in $G$.
\label{Proposition: automaton grigorchuk brother}
\end{Proposition}

\begin{proof}
Using \cref{equation: properties sections} we have 
\begin{align*}
a d^{-1} = (1,1)\sigma, & \qquad cd^{-1} = ((ad^{-1})^{-1}, bc^{-1}), \\ 
bc^{-1} = (1, db^{-1}), & \qquad db^{-1} = (ad^{-1}, cd^{-1}).
\end{align*}
So, calling $$a' = ad^{-1}, \quad b' = db^{-1}, \quad c' = cd^{-1} \quad \text{and} \quad d' = bc^{-1},$$ we obtain $$a' = (1,1)\sigma, \quad b' = (a',c'), \quad c' = (a',d'), \quad \text{and} \quad d' = (1, b').$$

Therefore, the first Grigorchuk group is contained in $G$.

It is well-known that the first Grigorchuk group is regular branch of depth $4$ (see for example \cite[Proposition 3]{Grigorchuk2000}). Let $G_\mathcal{P}$ be its closure. Using GAP, we find that $\pi_4(G) = \pi_4(G_\mathcal{P})$ and by \cref{theorem: closure automaton and finite type}, they have the same closure.

Using GAP \cite{GAP}, we find that the element $g = bc^{-2}b^{-1}d^2b^{-1}cda^{-1}$ has infinite order, since $g^{8} \in \St_G(6)$ and $$g^8|_{1^6} = g,$$ where $1^6$ represents the leftmost vertex at level $6$ of the tree. This implies that $G$ is not torsion. Finally, suppose that the first Grigrochuk group has finite index in $G$. Then, there exists $n \in \NN$ such that $g^n$ is in the first Grigorchuk group. Since the first Grigorchuk group is torsion, this implies that $g$ is torsion, leading to a contradiction.
\end{proof}

\section{Classification of groups of finite type up to isomorphism} \label{section: classification of group of finite type up to isomorphism}

In this section we will give tools to prove whether two groups of finite type acting on the same tree are isomorphic or not. In particular, we will give the proof of \cref{Theorem: finite type isomorphic conjugated iff gamma cycle}. Notice that since groups of finite type are profinite groups, an isomorphism in this context means that the groups are isomorphic as abstract groups and homeomorphic as topological groups. 

Let $G_1$ and $G_2$ be two groups of finite type acting on the same $d$-regular tree $T$. Then, there exists $\mathcal{P} \leq \Aut(T^D)$ and $\mathcal{Q} \leq \Aut(T^E)$ such that $G_1 \simeq G_\mathcal{P}$ and $G_2 \simeq G_\mathcal{Q}$. Although at the beginning they may have different depth, we may assume that the depth is the same by using the following lemma: 

\begin{Lemma}
Let $T$ be a $d$-regular tree and $G$ a group of finite type group of depth $D$ acting on $T$. Then $G$ is a group of finite type group of depth $E$ for all $E \geq D$.
\label{lemma: finite type different depths}
\end{Lemma}

\begin{proof}
Consider $\pi: \Aut(T^E) \rightarrow \Aut(T^D)$ the natural projection. If $G$ is a group of finite type of depth $D$, then there exists $\mathcal{P} \leq \Aut(T^D)$ such that $G \simeq G_\mathcal{P}$. Take $\mathcal{Q} = \pi^{-1}(\mathcal{P})$. We want to prove that $G_\mathcal{P} = G_\mathcal{Q}$. 

If $g \in G_\mathcal{Q}$ and $v$ is a vertex in $T$, then $g|_v^E \in \mathcal{Q}$. Therefore, $g|_v^D = \pi(g|_v^E) \in \pi(\mathcal{Q}) = \mathcal{P}$. Thus, $G_\mathcal{Q} \leq G_\mathcal{P}$. If $g \in G_\mathcal{P}$, then for every vertex $v$ in $T$, we have $g|_v^D \in \mathcal{P}$. Therefore $g|_v^E \in \pi^{-1}(\mathcal{P}) = \mathcal{Q}$.
\end{proof}

Our first guess (and wish) is to have the following result:

\bigskip

``Let $T$ be a $d$-regular tree, $D$ a natural number and $\mathcal{P}, \mathcal{Q} \leq \Aut(T^D)$ minimal pattern subgroups. Then, the groups of finite type $G_\mathcal{P}$ and $G_\mathcal{Q}$ are isomorphic if and only if $\mathcal{P}$ is isomorphic to $\mathcal{Q}$." 

\bigskip

However, this will be disproved in \cref{proposition: P iso Q does not imply GP iso GQ}. 

\subsection{Conjugations of groups of finite type}

Let us start with the following result:

\begin{Proposition}
Let $T$ be a $d$-regular tree, $D$ a natural number and $\mathcal{P}, \mathcal{Q} \leq \Aut(T^D)$ minimal pattern subgroups. Suppose there exists $g \in \Aut(T)$ such that 
\begin{enumerate}
\item For all $v$ vertex in $T$, $(g|_v^D) \, \mathcal{P} \, (g|_v^D)^{-1} = \mathcal{Q}$,
\item For all $n \in \NN$ and $v,w \in \mathcal{L}_n$,  $(g|_v^D) \, (g|_w^D)^{-1} \in \mathcal{Q}$.
\end{enumerate}
Then $G_\mathcal{P}$ and $G_\mathcal{Q}$ are isomorphic via conjugation by $g$.
\label{proposition: isomorphism between finite type groups}
\end{Proposition}

\begin{proof}
Define $\phi: \Aut(T) \rightarrow \Aut(T)$ as $\phi(h) = ghg^{-1}$. Since $\phi$ is an isomorphism of topological groups, we only need to prove that with the hypothesis given, we have that $\phi(G_\mathcal{P}) \leq G_\mathcal{Q}$ and that $\phi^{-1}(G_\mathcal{Q}) \leq G_\mathcal{P}$. Let $v$ be a vertex in $T$ and $h \in G_\mathcal{P}$, then
\begin{align*}
\phi(h)|_v^D = g|_{hg^{-1}(v)}^D \, h|_{g^{-1}(v)}^D \, (g|_{g^{-1}(v)}^D)^{-1} = \\ \left( g|_{hg^{-1}(v)}^D \, (g|_{g^{-1}(v)}^D)^{-1} \right) \left( g|_{g^{-1}(v)}^D \, h|_{g^{-1}(v)}^D \, (g|_{g^{-1}(v)}^D)^{-1} \right).
\end{align*}

By the second property, the first term in parenthesis is in $\mathcal{Q}$ and by the first property, the second term in parenthesis is in $\mathcal{Q}$. \\

The inverse will be $\phi^{-1}(h) = g^{-1}hg$, that has the form of $\phi$ but changing $g$ by $g^{-1}$. By the previous part of the proof, if we prove that $g^{-1}$ satisfies the conditions
\begin{enumerate}
\item[1'.] For all $v$ vertex in $T$, $(g^{-1}|_{v}^D) \mathcal{Q} (g^{-1}|_{v}^D)^{-1} = \mathcal{P}$,
\item[2'.] for all $n \in \NN$ and $v,w \in \mathcal{L}_n$, $(g^{-1}|_v^D) (g^{-1}|_w^D)^{-1} \in \mathcal{P}$,
\end{enumerate}
then $\phi^{-1}$ maps $G_\mathcal{Q}$ to $G_\mathcal{P}$.

For 1', we apply 1 with $v$ being $g^{-1}(v)$. Then, $$(g|_{g^{-1}(v)}) \mathcal{P} (g|_{g^{-1}(v)}^D)^{-1} = \mathcal{Q}.$$

Solving $\mathcal{P}$, $$\mathcal{P} = (g|_{g^{-1}(v)}^D)^{-1} \, \mathcal{Q} \, (g|_{g^{-1}(v)}^D) = (g^{-1}|_{v}^D) \, \mathcal{Q} \, (g^{-1}|_{v}^D)^{-1}.$$

For 2', 
\begin{align*}
(g^{-1}|_{v}^D) (g^{-1}|_{w}^D)^{-1} = (g^{-1}|_{v}^D) (g|_{g^{-1}(w)}^D) = \\ (g^{-1}|_{v}^D) (g|_{g^{-1}(w)}^D) (g|_{g^{-1}(v)}^D)^{-1} (g^{-1}|_{v}^D)^{-1} 
\end{align*}

Applying 2 with $v$ being $g^{-1}(w)$ and $w$ being $g^{-1}(v)$, we conclude that the central term $(g|_{g^{-1}(w)}^D) (g|_{g^{-1}(v)}^D)^{-1} \in \mathcal{Q}$. Finally applying 1', we get that $$(g^{-1}|_{v}^D) (g^{-1}|_{w}^D)^{-1} \in \mathcal{P}.$$
\end{proof}

\begin{Definition}
Let $T$ be a $d$-regular tree, $D$ a natural number and $\mathcal{P}, \mathcal{Q} \leq \Aut(T^D)$ minimal pattern subgroups. Define $\mathcal{S}_{\mathcal{P}\mathcal{Q}} = \set{r \in \Aut(T^D): r\mathcal{P}r^{-1} = \mathcal{Q}}$. We will call $\mathcal{S}_{\mathcal{P}\mathcal{Q}}$ the \textit{set of transformations} from $\mathcal{P}$ to $\mathcal{Q}$.
\label{definition: SPQ}
\end{Definition}

We equip $\mathcal{S}_{\mathcal{P}\mathcal{Q}}$ with an equivalence relation $\sim$ and with a relation $\rightarrow$. We say that $r \sim r'$ if and only if $r r'^{-1} \in \mathcal{Q}$ and that $r \xrightarrow[]{x} r'$ for $x \in \mathcal{L}_1$ if $r$ can be extended using the pattern $r'$, namely, $r|_x = r'|_\emptyset^{D-1}$.

Each structure given to $\mathcal{S}_{\mathcal{P}\mathcal{Q}}$ represents a different characteristic requested to $g$ in the \cref{proposition: isomorphism between finite type groups}. The definition of $\mathcal{S}_{\mathcal{P}\mathcal{Q}}$ represents the condition 1, the equivalence relation the condition 2 and the latter relation the possibility to extend the patterns until we can define a global element $g$. 

\begin{Proposition}
Let $T$ be a $d$-regular tree, $D$ a natural number, $\mathcal{P}, \mathcal{Q} \leq \Aut(T^D)$ minimal pattern subgroups and suppose there exists $r_0 \in \mathcal{S}_{\mathcal{P}\mathcal{Q}}$. Then, 

\begin{enumerate}
\item $\mathcal{S}_{\mathcal{P}\mathcal{Q}} = N(\mathcal{Q}) \,\, r_0,$ where $N(\mathcal{Q})$ is the normalizer of $\mathcal{Q}$ in $\Aut(T^D)$. 

\item The map $n \mapsto nr_0$ induces the bijection  
\begin{align*}
\mathcal{Q} \backslash N(\mathcal{Q}) \longrightarrow \mathcal{S}_{\mathcal{P}\mathcal{Q}}/\sim
\end{align*}
given by
\begin{align*}
[n] \longrightarrow [nr_0],
\end{align*}
where for the class in the domain we refer to the equivalence class modulo $\mathcal{Q}$ and for the second class to the equivalence relation defined in $\mathcal{S}_{\mathcal{P}\mathcal{Q}}$.
\end{enumerate}
\label{proposition: relation SPQ normalizer}
\end{Proposition}

\begin{proof}
For (1), if $r \in \mathcal{S}_{\mathcal{P}\mathcal{Q}}$, then $r \mathcal{P} r^{-1} = \mathcal{Q}$. As $r_0 \in \mathcal{S}_{\mathcal{P}\mathcal{Q}}$, then we deduce that $(rr_0^{-1}) \mathcal{Q} (r r_0^{-1})^{-1} = \mathcal{Q}$ and consequently $\mathcal{S}_{\mathcal{P}\mathcal{Q}} \, r_0^{-1} \in N(\mathcal{Q})$. Analogously, if $n \in N(\mathcal{Q})$, then $n \mathcal{Q} n^{-1} = (nr_0) \mathcal{P} (nr_0)^{-1} = \mathcal{Q}$, so $nr_0 \in \mathcal{S}_{\mathcal{P}\mathcal{Q}}$. (2) is automatic.
\end{proof}

Define $\Gamma_{\mathcal{P}\mathcal{Q}}$ the directed graph whose vertices are the classes $\mathcal{S}_{\mathcal{P}\mathcal{Q}}/\sim$ and a class $C$ is connected to a class $C'$ and we write $C \rightarrow C'$ if and only if given $r \in C$, there exist $r_1,\dots,r_d \in C'$ such that $r \xrightarrow[]{x} r_x$ for all $x \in \mathcal{L}_1$. 

\begin{Lemma}
Let $T$ be a $d$-regular tree, $D$ a natural number and $\mathcal{P}, \mathcal{Q} \leq \Aut(T^D)$ minimal pattern subgroups such that $\mathcal{S}_{\mathcal{P}\mathcal{Q}} \neq \emptyset$. Let $C$ and $C'$ be two classes in $\mathcal{S}_{\mathcal{P}\mathcal{Q}}/\sim$. Then, $C \rightarrow C'$ if and only if there exists $r \in C$ such that there are $r_1,\dots,r_d \in C'$ with $r \xrightarrow[]{x} r_x$ for all $x \in \mathcal{L}_1$. In other words, we can replace the ``for all elements in the class $C$" by only one.
\label{lemma: equivalence finite type graph isomorphism}
\end{Lemma}

\begin{proof}
Direct is straightforward. For the converse, let $s \in C$. Then, there exists $q \in \mathcal{Q}$ such that $s = rq$. Since $\mathcal{Q}$ is a minimal pattern subgroup, there exist $q_1,\dots,q_d \in \mathcal{Q}$ such that $q \xrightarrow[]{x} q_x$. Consider $s_x = r_{q(x)}q_x$. Then, $s \xrightarrow[]{x} s_x$. Since $q_x \in \mathcal{Q}$, then $s_x \sim r_{q(x)}$ and consequently $s_x \in C'$ for all $x \in \mathcal{L}_1$.  By definition, $C \rightarrow C'$.
\end{proof}

\begin{theorem}
Let $T$ be a $d$-regular tree, $D$ a natural number, $\mathcal{P}, \mathcal{Q} \leq \Aut(T^D)$ minimal pattern subgroups and  $\Gamma_{\mathcal{P}\mathcal{Q}}$ the associated directed graph. Then, the following statements are equivalent:
\begin{enumerate}
\item There exists $g \in \Aut(T)$ such that  
\begin{enumerate}[(a)]
    \item for all $v$ in $T$, then $g|_v^D \mathcal{P} (g|_v^D)^{-1} = \mathcal{Q}$,
    \item For all $n \in \NN$ and $v,w \in \mathcal{L}_n$, then $g|_v^D (g|_w^D)^{-1} \in \mathcal{Q}$,
\end{enumerate}
\item $\Gamma_{\mathcal{P}\mathcal{Q}}$ has a cycle.
\end{enumerate}
\label{theorem: Gamma cycle equivalence}
\end{theorem}

\begin{proof}
For (1) $\Rightarrow$ (2), by (a), the pattern $g|_v^D \in \mathcal{S}_{\mathcal{P}\mathcal{Q}}$, and by (b), patterns in the same level are in the same class of $\Gamma_{\mathcal{P}\mathcal{Q}}$, so we can label $C_n$ the class corresponding to $g|_v^D$ for any vertex $v \in \mathcal{L}_n$. By the existence of $g$, clearly $g|_v^D \xrightarrow[]{x} g|_{vx}^D$ for all $v$ vertex in $T$ and $x \in \mathcal{L}_1$, so by \cref{lemma: equivalence finite type graph isomorphism}, we have $C_n \rightarrow C_{n+1}$ for all $n \in \NN$. This creates a path in $\Gamma_{\mathcal{P}\mathcal{Q}}$. Now, the set $\mathcal{S}_{\mathcal{P}\mathcal{Q}}$ is finite, so $\Gamma_{\mathcal{P}\mathcal{Q}}$ has finitely many vertices and consequently the latter path has to self-intersect, giving a cycle. 

For (2) $\Rightarrow$ (1), let $C_0 \rightarrow C_1 \rightarrow \dots \rightarrow C_m \rightarrow C_0$ be a cycle in $\Gamma_{\mathcal{P}\mathcal{Q}}$ and take any $r \in C_0$. By definition of the graph, there exist $r_1,\dots,r_d \in C_1$ such that $r \xrightarrow[]{x} r_x$ for all $x \in \mathcal{L}_1$. Now $C_1$ is connected to $C_2$, so each $r_x$ can be extended via $r_{xy}$ with each $r_{xy} \in C_2$. Notice that every pattern at the same level is in the same class, so that gives condition (b). On the other hand, each pattern is in $\mathcal{S}_{\mathcal{P}\mathcal{Q}}$, so they satisfy condition (a). Gluing these patterns, we generate the element $g$. 
\end{proof}

Combining \cref{proposition: isomorphism between finite type groups} and \cref{theorem: Gamma cycle equivalence}, we conclude the following:

\begin{corollary}
Let $T$ be a $d$-regular tree, $D$ a natural number, $\mathcal{P}, \mathcal{Q} \leq \Aut(T^D)$ minimal pattern subgroups and  $\Gamma_{\mathcal{P}\mathcal{Q}}$ the associated directed graph. If $\Gamma_{\mathcal{P}\mathcal{Q}}$ has a cycle then $G_\mathcal{P}$ and $G_\mathcal{Q}$ are isomorphic via conjugation by an element in $\Aut(T)$.
\label{corollary: Gamma cycle they are conjugated}
\end{corollary}

The converse is not far to be true:

\begin{theorem}
Let $T$ be a $d$-regular tree, $D$ a natural number, $\mathcal{P}, \mathcal{Q} \leq \Aut(T^D)$ minimal pattern subgroups and  $\Gamma_{\mathcal{P}\mathcal{Q}}$ the associated directed graph. If $G_\mathcal{P}$ is fractal, then $G_\mathcal{P}$ is isomorphic to $G_\mathcal{Q}$ via conjugation by an element $g \in \Aut(T)$ if and only if $\Gamma_{\mathcal{P} \mathcal{Q}}$ has a cycle.
\label{theorem: finite type conjugated in Aut(T) iff gamma cycle}
\end{theorem}

\begin{proof}
For the direct, by \cref{theorem: Gamma cycle equivalence}, it is equivalent to prove that $g$ satisfies: 
\begin{enumerate}[(a)]
    \item for all $v$ vertex in $T$, $(g|_v^D) \mathcal{P} (g|_v^D)^{-1} = \mathcal{Q}$,
    \item For all $n \in \NN$ and $v,w \in \mathcal{L}_n$, then $(g|_v^D) (g|_w^D)^{-1} \in \mathcal{Q}$.
\end{enumerate}

For the point (a), restricting the isomorphism to the first $D$ levels, we have that $\pi_D(g) \, \mathcal{P} \, \pi_D(g)^{-1} = \mathcal{Q}$, so any conjugate of $\mathcal{P}$ must have the same cardinality as $\mathcal{Q}$. Now, take $v$ a vertex in $T$, $\alpha \in \mathcal{P}$, and choose $h \in \st_{G_\mathcal{P}}(v)$ such that $h|_{v}^D = \alpha$. Such an element $h$ exists because $G_\mathcal{P}$ is fractal. Then, 
\begin{align*}
(ghg^{-1})|_{g(v)}^D = (g|_v^D) (h|_v^D) (g|_v^D)^{-1} = (g|_v^D) \, \alpha \, (g|_v^D)^{-1} \in \mathcal{Q}, 
\end{align*}
as $g$ conjugates $G_\mathcal{P}$ onto $G_\mathcal{Q}$. This implies that $(g|_v^D) \mathcal{P} (g|_v^D)^{-1} \leq \mathcal{Q}$ and since they have the same cardinality, we obtain $(g|_v^D) \mathcal{P} (g|_v^D)^{-1} = \mathcal{Q}$ for any vertex $v$ in $T$. 

For (b), by level-transitivity, there exists $k \in G_\mathcal{P}$ such that $k(w) = v$ and by fractalness, there exists $l \in \st_{G_\mathcal{P}}(w)$ such that $l|_w^D = (k|_w^D)^{-1}$. Take $h = kl$. Then $h(w) = v$ and $$h|_w^D = (k|_w^D)(l|_w^D) = 1.$$ Hence, 
\begin{align*}
(ghg^{-1})|_{g(w)}^D = (g|_{h(w)}^D) (h|_w^D) (g|_w^D)^{-1} = (g|_v^D) (g|_w^D)^{-1} \in \mathcal{Q}. 
\end{align*}

The converse is \cref{corollary: Gamma cycle they are conjugated}. 
\end{proof}

\subsection{Results of rigidity for groups of finite type}
\label{subsection: Results of rigidity for groups of finite type}

We have characterized what a conjugation between two groups of finite type must satisfy. The problem is that we do not know if all the isomorphisms between groups of finite type must be conjugations. In this subsection, we will adapt known results of rigidity existing in the literature to groups of finite type.

Given $T$ a spherically homogeneous tree, we say that a group $G \leq \Aut(T)$ is \textit{$T$-rigid} if the unique kinds of isomorphisms of $G$ in $\Aut(T)$ are conjugations by elements in $\Aut(T)$.

\begin{theorem}[{\cite[Proposition 2.4.49]{Nekrashevysh2022}}]
Let $T$ be a spherically homogeneous tree and $G \leq \Aut(T)$ such that $\rist_G(v)$ acts level-transitively on the vertices below $T_v$ for all $v$ vertex in $T$. Then $G$ is $T$-rigid.
\end{theorem}

If $G_\mathcal{P}$ is a group of finite type, let us denote $K_\mathcal{P}$ the maximal regular branching subgroup of $G_\mathcal{P}$. The existence of such a subgroup is justified in \cref{lemma: unique maximal reg branch subgroup} and more results about it will be obtained in \cref{section: The maximal branching subgroup}. 

\begin{Proposition}
Let $T$ be a $d$-regular tree, $D$ a natural number, and $\mathcal{P} \leq \Aut(T^D)$ a minimal pattern subgroup. If $K_\mathcal{P}$ acts transitively on the first level of $T$, then $G_\mathcal{P}$ is $T$-rigid. 
\label{theorem: finite type KP and rigidity}
\end{Proposition}

\begin{proof}
By definition of the geometric product, $(K_\mathcal{P})_n \leq \RiSt_{G_\mathcal{P}}(n)$ for all $n \in \NN$. Therefore, if $v \in \mathcal{L}_n$, then $\varphi_v(\rist_{G_\mathcal{P}}(v)) \geq K_\mathcal{P}$ and as $K_\mathcal{P}$ acts transitively on the first level of $T$, then $\rist_{G_\mathcal{P}}(v)$ acts transitively on the first level of $T_v$. Iterating this process, we obtain level-transitivity of $\rist_{G_\mathcal{P}}(v)$ below $T_v$.
\end{proof}

In the case that $d = p$ is a prime number and $G_\mathcal{P} \leq W_p$, then the hypothesis of \cref{theorem: finite type KP and rigidity} force $G_\mathcal{P}$ to be $W_p$. Indeed, if $K_\mathcal{P}$, the maximal regular branching subgroup of $G_\mathcal{P}$, is transitive, then it must act like $C_p$ on the first level, but $(K_\mathcal{P})_n \leq K_\mathcal{P}$ for all $n \in \NN$ which implies that $K_\mathcal{P} = W_p$. Therefore, we need other results of rigidity to deal with the case where $G_\mathcal{P} \leq W_p$. 

\begin{theorem}[{\cite[Proposition 5.3]{Fariña2025Boston}}]
Let $T$ be a $d$-regular tree and $G \leq \Aut(T)$ a weakly branch group satisfying the following conditions:

(*) For each vertex $v$ in $T$, the stabilizer $\st_G(v)$ acts as a transitive cyclic
group of prime order on the immediate descendants of $v$.

(N) For every triple of vertices $v_1$, $v_2$, $v_3$ in $T$ such that $v_{i+1}$ is a child of $v_i$ for $1 \leq i \leq 2$ and $[\st_G(v_2) : \bigcap_{v \in V} \st_G(v)] > d$ where $V$ are all the vertices below $v_1$ in the level of $v_3$. 

Then $G$ is $T$-rigid.
\label{theorem: Jorge rigidity}
\end{theorem}

The condition (*) was firstly used in \cite{GrigorchukWilson2003} together with another condition to ensure rigidity of groups acting on trees. \cref{theorem: Jorge rigidity} is a recent generalization of the result in \cite{GrigorchukWilson2003} that ensures rigidity for a larger family of groups. Our goal now is to understand the conditions (*) and (N) when we restrict to groups of finite type.

Let us denote $w_n$ the rightmost vertex on level $n$ of a $d$-regular tree $T$. Given $\mathcal{P}$ a minimal pattern subgroup in $\Aut(T^D)$, define $\Gamma_\mathcal{P}$ the directed graph whose vertices are the elements in the set $\st_\mathcal{P}(w_D) := \set{t \in \mathcal{P}: t(w_D) = v}$, and $t_1$ has an edge going to $t_2$ if and only if $t_1$ can be extended using $t_2$ along the vertex $w_1$. 

Clearly, the number of vertices of $\Gamma_\mathcal{P}$ is $\abs{\st_\mathcal{P}(w_D)}$. If $\mathcal{P}$ acts transitively on level $D$, by the orbit-stabilizer theorem, we have $\abs{\st_\mathcal{P}(w_D)} = \abs{\mathcal{P}}/d^D$. We are going to be interested when $G_\mathcal{P}$ is an infinite group included in $W_p$. In that case, we have the following lemma:

\begin{Lemma}
Let $T$ be a $p$-regular tree with $p$ prime, $D$ a natural number, and $\mathcal{P} \leq \Aut(T^D)$ a minimal pattern subgroup such that $G_\mathcal{P} \leq W_p$ and $G_\mathcal{P}$ is infinite. Then every vertex in $\Gamma_\mathcal{P}$ has $\abs{\St_\mathcal{P}(D-1) \cap \St_\mathcal{P}(w_D)}$ outgoing edges and there exists a positive number $m$ such that every vertex has either $m$ ingoing edges or zero. 
\end{Lemma}

\begin{proof}
The first step is to prove that any pattern $t$ has at least one outgoing edge. Let $t \in \st_\mathcal{P}(w_D)$. As $\mathcal{P}$ is a minimal pattern subgroup, there exists $q \in \mathcal{P}$ such that $q$ extends $t|_{w_1}$. If $q|_{w_{D-1}}^1 = 1$, then $q \in \st_\mathcal{P}(w_D)$ and we are done. If not, then $q|_{w_{D-1}}^1 = \sigma^k$, where $\sigma$ is the permutation $(1, \dots, p)$ and $k \in C_p$ is different to zero. 

As $G_\mathcal{P}$ is infinite, there exists a non-trivial pattern $s$ in $\St_\mathcal{P}(D-1)$ by \cref{theorem: finite type finiteness condition}. Therefore, there exists a vertex $v \in \mathcal{L}_{D-1}$ such that $s|_v^1 = \sigma^\ell$ for some $\ell \in C_p$ different to zero. Taking a power of $s$ if necessary, we may assume that $s|_v^1 = \sigma^k$. 

By \cref{lemma: self similarity and level-transitivity}, we have that $G_\mathcal{P}$ is level-transitive, so there exists $r \in \mathcal{P}$ such that $r(v) = w_{D-1}$ and so $rs^{-1}r^{-1} \in \St_\mathcal{P}(D-1)$ and $(rs^{-1}r^{-1})|_{w_{D-1}}^1 = \sigma^{-k}$. Multiplying this element by $q$, we obtain an extension of $t$ such that $(qrs^{-1}r^{-1})|_{w_{D-1}}^1 = 1$ as we wanted. 

Therefore, any pattern $t$ has an extension $q$ in $\st_\mathcal{P}(w_D)$. As all the extensions of $t$ are of the form $qs$ with $s \in \St_\mathcal{P}(D-1)$, these extensions will be in $\st_\mathcal{P}(w_D)$ if and only if $s \in \st_\mathcal{P}(w_D)$, giving the first part of the lemma.

For the second part, let $m$ be the number of ingoing edges for the identity pattern. We know $m \geq 1$ as the identity pattern goes into itself. Now, if $t$ is a pattern in $\st_\mathcal{P}(w_D)$ that has an ingoing arrow, this means that we have $q \in \st_\mathcal{P}(w_D)$ such that $q|_{w_1} = \pi_{D-1}(t)$. If $q'$ is another pattern such that $q'|_{w_1} = \pi_{D-1}(t)$, as $q$ and $q'$ fix the vertex $w_1$, then $q q'^{-1}$ is a pattern having an edge to the identity pattern. On the other hand, if $r$ is a pattern having an ingoing edge to the identity pattern, then $rq$ satisfies $(rq)|_{w_1} = \pi_{D-1}(t)$, giving a bijection between the ingoing edges of $t$ and the ingoing edges of the identity pattern. 
\end{proof}

\begin{Proposition}
Let $T$ be a $d$-regular tree, $D$ a natural number, and $\mathcal{P} \leq \Aut(T^D)$ a minimal pattern subgroup. Then $G_\mathcal{P}$ satisfies conditions (*) and (N) if and only if there exists $p$ prime such that $G_\mathcal{P} \leq W_p$, $G_\mathcal{P}$ is infinite and $\Gamma_\mathcal{P}$ has a cycle and a path from any vertex of the cycle to a pattern $t$ such that the action of $t$ below $w_{D-2}$ is non-trivial.
\label{Proposition: equivalence (N)(*)}
\end{Proposition}

\begin{proof}
By \cref{lemma: finite type different depths}, we may assume that $D \geq 2$. 

$(\Rightarrow)$ Applying (*) with $v = \emptyset$, we obtain $\pi_1(\st_{G_\mathcal{P}}(\emptyset)) = \pi_1(G_\mathcal{P})$ must act transitively as a cyclic group of prime order, which implies that $d = p$ for some prime $p$. If $v$ is any vertex and $g \in G_\mathcal{P}$, then $g|_v^1 \in \pi_1(G_\mathcal{P})$ by self-similarity and therefore $G_\mathcal{P} \leq W_p$. As $\st_{G_\mathcal{P}}(v)$ acts transitively on the immediate descendants of $v$ for any $v$, then $G_\mathcal{P}$ is level-transitive and consequently is infinite.

We now use condition (N) to deduce the last part. First, observe that we can express condition (N) in a different way. Since $v_3$ is a descendant of $v_2$ and $v_3 \in V$, then $$\left[\st_G(v_2):\bigcap_{v \in V} \st_G(V)\right] = \left[\st_G(v_2):\st_G(v_3)\right] \left[\st_G(v_3):\bigcap_{v \in V} \st_G(V) \right].$$ By (*), $[\st_G(v_2):\st_G(v_3)] = p$ and $d = p$, so $$\left[\st_G(v_3):\bigcap_{v \in V} \st_G(v) \right] = \frac{\left[\st_G(v_2):\bigcap_{v \in V} \st_G(v)\right]}{\left[\st_G(v_2):\st_G(v_3)\right]} > \frac{p}{p} = 1,$$ namely, there exists an element in $\st_G(v_3)$ moving vertices of $V$.

Choose $v_i = w_{n+i}$ for $i = 1,2,3$ with $n > \abs{\st_\mathcal{P}(w_D)}$. Then by condition (N), there exists $g \in G_\mathcal{P}$ such that $g \in \st_{G_\mathcal{P}}(w_{n+3})$ but it acts non-trivially on the set of vertices $V = \mathcal{L}_{n+3} \cap T_{w_{n+1}}$. As $g \in \st_{G_\mathcal{P}}(w_{n+3})$, the patterns  $g|_{w_m}^D \in \st_\mathcal{P}(w_D)$ for $m = 0, \dots, n+3-D$ and they make a path in the graph $\Gamma_{\mathcal{P}}$. Since we chose $n$ greater than the amount of vertices in the graph $\Gamma_{\mathcal{P}}$, then this path self-intersects and ends up in a pattern whose action below $w_{D-2}$ is non-trivial.

$(\Leftarrow)$ By \cref{lemma: self similarity and level-transitivity} we have that $G_\mathcal{P}$ is level-transitive, and since it is in $W_p$, we immediately have (*) by \cref{lemma: level transitivity equivalence}. To prove (N), let $v_1,v_2,v_3$ vertices in the hypothesis of the condition (N) with $v_1 \in \mathcal{L}_{n+1}$. Since $G_\mathcal{P}$ is level-transitive, there exists $g \in G_\mathcal{P}$ such that $g(v_3) = w_{n+3}$ and then by adjacency, $g(v_i) = w_{n+i}$ for $i = 1,2,3$. Since conjugating by $g$ is an inner isomorphism of $G_\mathcal{P}$, we have $$\left[\st_G(v_3):\bigcap_{v \in V} \st_G(v) \right]  > 1 \Longleftrightarrow \left[\st_G(w_{n+3}):\bigcap_{v \in W} \st_G(v) \right] > 1,$$ where $W = g(V)$ are the vertices below $w_{n+1}$ on the level of $w_{n+3}$. 

Therefore it is enough to prove the result for $w_{n+1}, w_{n+2}, w_{n+3}$ and $n \geq -1$. Let us first assume that $n+3 \geq D$, namely, the vertex $w_{n+3}$ is at least at level $D$ or further from the root. Then, we can choose a starting point of the cycle of $\Gamma_\mathcal{P}$ such that the pattern $t$ is at the step $n+3-D$ of the path. Completing the action of the element arbitrarily in the other vertices of $T$, we construct an element $g \in \st_\mathcal{P}(w_{n+3})$ that moves vertices in $W$ as we need. If $n+3 < D$, we take any element  $g \in G_\mathcal{P}$ extending $t$ and by self-similarity, $g|_{w_{D-n-3}} \in G_\mathcal{P}$ and satisfies the sought condition.   
\end{proof}

In the case that the group in \cref{theorem: Jorge rigidity} is fractal and included in $W_p$, the author in \cite{Fariña2025Boston} gives a nicer result:

\begin{theorem}[{\cite[Proposition 5.3]{Fariña2025Boston}}]
Let $T$ be a $p$-regular tree with $p$ prime and $G \leq W_p$ a fractal weakly branch group. Then $G$ is $T$-rigid if and only if $\pi_2(G) \ncong C_p \times C_p$. 
\label{theorem: Jorge rigidity fractal}
\end{theorem}

\subsection{Other strategies}

If the results in \cref{subsection: Results of rigidity for groups of finite type} do not apply, the following tool can be useful when the groups of finite type are just-infinite. 

\begin{Definition}
Denote $\mathfrak{G}$ the category of groups. We say that a variant functor $\mathcal{F}: \mathfrak{G} \rightarrow \mathfrak{G}$ is a \textbf{characteristic functor} if it satisfies the following properties: 

\begin{enumerate}
\item $\mathcal{F}(G)$ is a characteristic subgroup of $G$ for all $G \in \mathfrak{G}$,
\item for all $f \in \Hom(G_1,G_2)$, the function $\mathcal{F}(f): \mathcal{F}(G_1) \rightarrow \mathcal{F}(G_2)$ is defined as $\mathcal{F}(f)(x) = f(x)$,
\item if $f$ is surjective, then $\mathcal{F}(f)$ is also surjective. 
\end{enumerate}
\end{Definition}

If in (2) we use the inclusion map, we obtain that if $H \leq G$, then $\mathcal{F}(H) \leq \mathcal{F}(G)$. Also, (2) and (3) imply that $\mathcal{F}(f(G_1)) = f(\mathcal{F}(G_1))$. Examples of these functors are $G \mapsto G^{(k)}$ and $G \mapsto \gamma_k(G)$ for all $k \in \NN$.

Hence, if $G_\mathcal{P}$ and $G_\mathcal{Q}$ were isomorphic as topological groups, for any characteristic functor $\mathcal{F}$, then $G_\mathcal{P}/\overline{\mathcal{F}(G_\mathcal{P})}$ and $G_\mathcal{Q}/\overline{\mathcal{F}(G_\mathcal{Q})}$ must be isomorphic.

Using the same argument as in \cref{lemma: inverse limit closed abelianization}, we obtain that $$G_\mathcal{P}/\overline{\mathcal{F}(G_\mathcal{P})} \simeq \varprojlim \pi_n(G_\mathcal{P})/\mathcal{F}(\pi_n(G_\mathcal{P})).$$

Since we assume $G_\mathcal{P}$ is just-infinite, the previous quotient is finite, so there exists $n_0 \in \NN$ such that $$\varprojlim \pi_n(G_\mathcal{P})/\mathcal{F}(\pi_n(G_\mathcal{P})) \simeq \pi_{n_0}(G_\mathcal{P})/\pi_{n_0}(\mathcal{F}(G_\mathcal{P})) = \pi_{n_0}(G_\mathcal{P})/\mathcal{F}(\pi_{n_0}(G_\mathcal{P})).$$

What we need to find is an explicit way to give $n_0$. The problem that we might have is that such a $n_0$ is too big to be able to prove that for different groups of finite type, the quotients $\pi_{n_0}(G_\mathcal{P})/\mathcal{F}(\pi_{n_0}(G_\mathcal{P}))$ and $\pi_{n_1}(G_\mathcal{Q})/\mathcal{F}(\pi_{n_1}(G_\mathcal{Q}))$ are not isomorphic. To maximize our chances or in other words, to make $n_0$ and $n_1$ as small as possible, we will need to use the maximal branching subgroup of each group of finite type.

\begin{Lemma}
\label{lemma: finite type proj derived series}
Let $T$ be a $d$-regular tree, $D$ a natural number, $\mathcal{F}$ a characteristic functor and $\mathcal{P} \leq \Aut(T^D)$ a minimal pattern subgroup. Denote by $K_\mathcal{P}$ the maximal branching subgroup of $G_\mathcal{P}$. If $K_\mathcal{P}$ is normal and there exists $n_0 \geq D$ such that $$[\pi_{n_0}(K_\mathcal{P}): \mathcal{F}(\pi_{n_0}(K_\mathcal{P}))] = [\pi_{n_0+1}(K_\mathcal{P}): \mathcal{F}(\pi_{n_0+1}(K_\mathcal{P}))],$$ then $$\varprojlim \pi_n(G_\mathcal{P})/\mathcal{F}(\pi_n(G_\mathcal{P})) \simeq \pi_{n_0}(G_\mathcal{P})/\mathcal{F}(\pi_{n_0}(G_\mathcal{P})).$$
\end{Lemma}

\begin{proof}
Since $K_\mathcal{P}$ is the maximal branching subgroup of $G$ and $G$ is regular branch over $\St_G(D-1)$, then $\St_G(D-1) \leq K_\mathcal{P}$. On the other hand, since $(K_\mathcal{P})_1 \leq K_\mathcal{P}$, then $\mathcal{F}((K_\mathcal{P})_1) \leq \mathcal{F}(K_\mathcal{P})$ and $\mathcal{F}(K_\mathcal{P}) \lhd G$, as $K_\mathcal{P} \lhd G$ and $\mathcal{F}(K_\mathcal{P})$ is a characteristic subgroup of $K_\mathcal{P}$. Therefore, by hypothesis and \cref{lemma: stabilizabition indices regular branch}, we obtain that $$[\pi_n(K_\mathcal{P}): \mathcal{F}(\pi_n(K_\mathcal{P}))] = [\pi_{n_0}(K_\mathcal{P}): \mathcal{F}(\pi_{n_0}(K_\mathcal{P}))]$$ for all $n \geq n_0$. As $$[K_\mathcal{P}: \overline{\mathcal{F}(K_\mathcal{P})}] = \lim_{n \rightarrow +\infty} [\pi_n(K_\mathcal{P}): \mathcal{F}(\pi_n(K_\mathcal{P}))],$$ by \cref{lemma: map defined in abelianizations} applied to $(A,B,\pi) := (K_\mathcal{P}, \overline{\mathcal{F}(K_\mathcal{P})}, \pi_{n_0})$, we obtain the inclusion $\St_G(n_0) \leq \overline{\mathcal{F}(K_\mathcal{P})}$. Since $\overline{\mathcal{F}(K_\mathcal{P})} \leq \overline{\mathcal{F}(G_\mathcal{P})}$, applying \cref{lemma: map defined in abelianizations} again with $(A,B,\pi) := (G_\mathcal{P}, \overline{\mathcal{F}(G_\mathcal{P})}, \pi_{n_0})$, we conclude that $$\varprojlim \pi_n(G_\mathcal{P})/\mathcal{F}(\pi_n(G_\mathcal{P})) \simeq \pi_{n_0}(G_\mathcal{P})/\mathcal{F}(\pi_{n_0}(G_\mathcal{P})).$$
\end{proof}

\section{Analysis of different cases $(d,D)$}
\label{section: analysis of different cases}

\begin{figure}[h!]
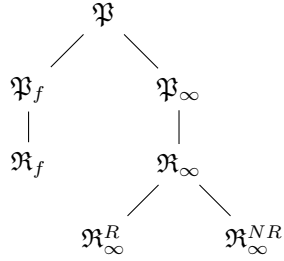

\centering
\tikz {
  \begin{scope}[name prefix = top-]
    \node (P) at (0,0) {$\mathfrak{P}$};
    \node (Pf) at (-1,-1) {$\mathfrak{P}_f$};
    \node (Poo) at (1,-1) {$\mathfrak{P}_\infty$};
    \node (Rf) at (-1,-2) {$\mathfrak{R}_f$};
    \node (Roo) at (1,-2) {$\mathfrak{R}_\infty$};
    \node (RooR) at (0,-3) {$\mathfrak{R}_\infty^R$};
    \node (RooNR) at (2,-3) {$\mathfrak{R}_\infty^{NR}$};
    \draw (P) -- (Pf);
    \draw (P) -- (Poo);
    \draw (Pf) -- (Rf);
    \draw (Poo) -- (Roo);
    \draw (Roo) -- (RooR);
    \draw (Roo) -- (RooNR);
  \end{scope}
}
\caption{Diagram of sets of minimal patterns classifying the associated groups of finite type up to isomorphism.}
\label{figure: isomorphism work diagram}
\end{figure}

In this section we will analyze the cases of binary tree for depths $2$, $3$ and $4$ and ternary for depths $2$ and $3$. We will calculate minimal patterns, study the groups of finite type in terms of fractalness and how many equivalence classes we can say there are based on the theorems proved in \cref{section: classification of group of finite type up to isomorphism}.

To analyze the classes of isomorphisms of a certain set $\mathfrak{P}$ of minimal patterns, we will do it in the following order: 

\begin{enumerate}[(1)]

\item Split $\mathfrak{P}$ in two disjoint subsets: $\mathfrak{P}_f$ the set of minimal pattern subgroups whose associated groups of finite type are finite and $\mathfrak{P}_\infty$ the subset of minimal pattern subgroups whose associated groups of finite type are infinite.

\item If $\mathcal{P}$ is in $\mathfrak{P}_f$, this means that $G_\mathcal{P} \simeq \mathcal{P}$, so we can easily solve how many classes of groups of finite type we have in $\mathfrak{P}_f$. Denote $\mathfrak{R}_f$ for a set of representatives of the classes in $\mathfrak{P}_f$.

\item For the subset $\mathfrak{P}_\infty$, we can use \cref{corollary: Gamma cycle they are conjugated} to split in classes by conjugation. Denote by $\mathfrak{R}_\infty$ a set of representatives of these classes.

\end{enumerate}

By the results we have, we conclude that the maximum number of classes we can have is $\#{\mathfrak{R}_f} + \#{\mathfrak{R}_\infty}$. The next results in \cref{section: classification of group of finite type up to isomorphism} allow us to know if different classes are not equivalent, giving a lower bound for the final number of classes we can have.

\begin{enumerate}[(4)]

\item Using \cref{theorem: finite type KP and rigidity}, \cref{theorem: Jorge rigidity} and \cref{theorem: Jorge rigidity fractal}, we analyze whether the groups are $T$-rigid or not, and this gives a lower bound for the number of isomorphic classes we can ensure. We split $\mathcal{R}_\infty$ in $\mathcal{R}_\infty^R$ and $\mathcal{R}_\infty^{NR}$, where the first is the set of patterns whose associated group of finite type is $T$-rigid and the second set is the complement.
\end{enumerate}

The \cref{figure: isomorphism work diagram} shows how the inclusions of the different subsets defined previously are.

\subsection{The case $(D,d) = (2,2)$} 

To start, we consider the case of groups of finite type in the binary tree with depth $D = 2$. This case was treated in \cite{Bondarenko2014} and \cite{Sunic2010} but they did not calculated the isomorphic classes nor analyzed fractalness. There are six minimal pattern subgroups. Three of them generate a finite group of finite type: the trivial group and two groups of finite type isomorphic to $C_2$. Therefore $\# \mathfrak{R}_f = 2$.

The other three do not generate a topologically finitely generated group of finite type by \cref{proposition: finite type not tfg}. In terms of fractalness, the three groups are super strongly fractal. Since the minimal patterns in $\mathfrak{P}_\infty$ are not pairwise isomorphic, then $\# \mathfrak{R}_\infty = 3$. As they are all fractal by \cref{theorem: Jorge rigidity fractal}, we have 2 infinite groups $T$-rigid and one that is not. Therefore, there are (in total) five isomorphic classes. 

\subsection{The case $(D,d) = (2,3)$} 

Using GAP, we have 60 minimal patterns, where 37 of them generate an infinite group of finite type and the others 23 give finite groups. The 37 infinite groups of finite type are not topologically finitely generated. We have that 27 of them satisfy \cref{proposition: finite type not tfg} with $n = 3$ and 10 of them with $n = 4$. All these results were already obtained in \cite{Bondarenko2014}. 

In terms of fractalness, the 23 finite groups are of course not fractal, so we just focus on the infinite groups of finite type. Among these 37 infinite groups of finite type, 13 of them are super strongly fractal and 12 are strongly fractal but not super strongly fractal. Since we are in the binary tree, fractal and strongly fractal are equivalent, so the remaining 12 are not fractal.

In order to classify them up to isomorphism, the 23 finite groups split in four different classes: $\mathfrak{R}_f = \set{\set{\id}, C_2, C_2 \times C_2,D_4}$. Among the remaining 37 groups in $\mathfrak{P}_\infty$, we have at most 19 classes. 15 of them are fractal. By \cref{theorem: Jorge rigidity fractal}, we find that 11 of them are $T$-rigid and the other 4 are not. For the 4 patterns that are not fractal, none of them satisfies \cref{Proposition: equivalence (N)(*)}, so we cannot ensure if they are $T$-rigid or not. Therefore we conclude that we have at most 23 classes but no less than 16.

\subsection{The case $(d,D) = (2,4)$}
\label{subsection: the case (dD) = (24)}

As it was shown in \cite{Bondarenko2014}, there are 4544 minimal pattern subgroups and only 32 of them give an infinite and topologically finitely generated group of finite type by applying \cref{theorem: finite type tfg} with $n_0 = 6$ (the others are not topologically finitely generated by \cref{proposition: finite type not tfg}). 

The 32 topologically finitely generated groups are super strongly fractal, so we can apply \cref{theorem: finite type conjugated in Aut(T) iff gamma cycle} and \cref{theorem: Jorge rigidity fractal}, obtaining that there are 8 different classes of isomorphisms.

In \cite{Bondarenko2014}, elements $a_i,b_j,c_k \in \Aut(T^4)$ are given such that the 32 minimal patterns can be expressed as $\mathcal{P}_{ijk} = \<a_i,b_j,c_k>$ where $i,k = 1,2,3,4$ and $j = 1,2$. 

Using \cref{theorem: closure automaton and finite type} and comparing some famous groups with the 32 topologically finitely generated groups of finite type we obtain:

\begin{itemize}
\item The closure of the first Grigorchuk group is $G_{\mathcal{P}_{123}}$ (this was already observed in \cite{Bondarenko2014}).
\item The closure of the twin of the first Grigorchuk group is $G_{\mathcal{P}_{123}}$.
\item The closure of $\IMG(z^2+i)$ is $G_{\mathcal{P}_{111}}$ (already observed in \cite{Bondarenko2014}).
\item The closure of the third Grigorchuk group is $G_{\mathcal{P}_{121}}$.
\item The closure of the Grigorchuk Erschler group is $G_{\mathcal{P}_{121}}$.
\end{itemize}

A description of their automata can be found in \cref{figure: automata groups closure depth 4}.

In \cref{section: List of automata groups whose closure is a group of finite type}, we will compute automata groups whose closure is $G_{\mathcal{P}_{ijk}}$ for each  $i,k = 1,2,3,4$ and $j = 1,2$. As a consequence of this list and the classification of the group of finite type up to isomorphism, we deduce the following corollary:

\begin{corollary}
The closure of $\IMG(z^2+i)$ and the closure of the third Grigorchuk group (and the Grigorchuk Erschler group) are isomorphic. 
\label{theorem: closures IMG 3rd Grig isomorphic}
\end{corollary}

\begin{proof}
We have $\overline{\IMG(z^2+i)} = G_{\mathcal{P}_{111}}$ and if $G_3$ denotes the third Grigorchuk group, then $\overline{G_3} = G_{\mathcal{P}_{121}}$. The graph $\Gamma_{\mathcal{P}_{111}\mathcal{P}_{121}}$ is shown in the \cref{figure: Gamma IMG 3rd Grig group} and since it has a cycle (in fact it has three cycles!), by \cref{theorem: Gamma cycle equivalence}, the groups of finite type are isomorphic. 

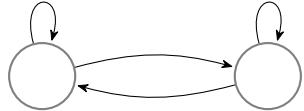
\begin{figure}[h!]
\begin{center}
\begin{tikzpicture}[shorten >=1pt,node distance=3cm,on grid,>={Stealth[round]},
    every state/.style={draw=black!50,thick,fill=white!5}, bend angle = 15]

  \node[state] (1)  {};
  \node[state] (2) [right=of 1] {};
  
  \path[->] 
  (1) edge [loop above] node  {} ()       
  (2) edge [loop above] node  {} ()       
  (1)  edge [bend left] node [above] {} (2)
  (2)  edge [bend left] node [below] {} (1);
\end{tikzpicture}
\end{center}
\caption{Graph $\Gamma_{\mathcal{P}_{111}\mathcal{P}_{121}}$}
\label{figure: Gamma IMG 3rd Grig group}
\end{figure}
\end{proof}

In the case of the groups in \cref{theorem: closures IMG 3rd Grig isomorphic}, the element $r_0$ in \cref{proposition: relation SPQ normalizer} can be taken to have all the labels $\sigma$, being $\sigma$ the non-trivial permutation in $\Sym(2)$. Then, one of the vertices (say the leftmost one) of the graph in \cref{figure: Gamma IMG 3rd Grig group} corresponds to the class of $r_0$ and since $r_0$ can be extended using itself at each level, we have that the element $g = (g,g)\sigma$ conjugates $\overline{\IMG(z^2+i)}$ and $\overline{G_3}$. As the graph indicates, there are other elements $g$ that also conjugate these two groups.

It is important to mention that \cref{Corollary: closure IMG 3rd Grig group} does not prove that the respective profinite completions of $\IMG(z^2+i)$ and $G_3$ are isomorphic. This is because it is known that $G_3$ has the congruence subgroup property, namely, the profinite completion of $G_3$ is isomorphic to $\overline{G_3}$ (the proof follows verbatim \cite[Proposition 10]{Grigorchuk2000}), but $\IMG(z^2+i)$ does not have the congruence subgroup property, as it was recently proved in \cite{Radi2025IMG}.

\begin{figure}
  \centering
\begin{minipage}{0.45\textwidth}
  \centering
\begin{tikzpicture}[shorten >=1pt, node distance=2.5cm, on grid, >={Stealth[round]},
  every state/.style={draw=blue!50, very thick, fill=blue!20}, bend angle=15]

  \node[state] (1) at ( -2,  2) {$1$};
  \node[state] (a) at ( 2, 2) {$a$};
  \node[state] (b) at (2, -2) {$b$};
  \node[state] (c) at (0,0) {$c$};
  \node[state] (d) at (-2,  -2) {$d$};

  \path[->]
    (1) edge [loop above] node {x/x} ()
    (a) edge node [above] {x/x+1} (1)
    (b) edge node [right] {1/1}  (a)
    (b) edge node [above] {2/2} (c)
    (c) edge node [above] {1/1} (a)
    (c) edge node [above] {2/2} (d)
    (d) edge node [left] {1/1} (1)
    (d) edge node [below] {2/2} (b);
\end{tikzpicture}
  \small First Grigorchuk group
\end{minipage}
\hfill
\begin{minipage}{0.45\textwidth}
  \centering
\begin{tikzpicture}[shorten >=1pt, node distance=2.5cm, on grid, >={Stealth[round]},
  every state/.style={draw=blue!50, very thick, fill=blue!20}, bend angle=15]

  \node[state] (1) at ( -2,  2) {$1$};
  \node[state] (a) at ( 2, 2) {$a$};
  \node[state] (b) at (2, -2) {$b$};
  \node[state] (c) at (0,0) {$c$};
  \node[state] (d) at (-2,  -2) {$d$};

  \path[->]
    (1) edge [loop above] node {x/x} ()
    (a) edge node [above] {x/x+1} (1)
    (b) edge node [above] {1/1}  (c)
    (b) edge node [right] {2/2} (a)
    (c) edge node [above] {1/1} (a)
    (c) edge node [above] {2/2} (d)
    (d) edge node [left] {1/1} (1)
    (d) edge node [below] {2/2} (b);
\end{tikzpicture}
  \small Twin of the first Grigorchuk group
\end{minipage}

\vspace{1em} % Space between rows

% Second row
\begin{minipage}{0.45\textwidth}
  \centering
\begin{tikzpicture}[shorten >=1pt,node distance=3cm,on grid,>={Stealth[round]},
    every state/.style={draw=blue!50,very thick,fill=blue!20}, bend angle = 15]

  \node[state] (1)  {$1$};
  \node[state] (a) [right=of 1] {$a$};
  \node[state] (b) [below =of a] {$b$};
  \node[state] (c) [below =of 1] {$c$};

  \path[->] 
  (1) edge [loop above] node  {x/x} ()       
  (a)  edge node [above] {x/x+1} (1)
  (b)  edge node [right] {1/1} (a)
  (b)  edge [bend left] node [below] {2/2} (c)
  (c)  edge [bend left] node [above] {1/1} (b)
  (c)  edge node [left] {2/2} (1);
\end{tikzpicture}
  \small $\IMG(z^2+i)$
\end{minipage}
\hfill
\begin{minipage}{0.45\textwidth}
  \centering
\begin{tikzpicture}[shorten >=1pt,node distance=3cm,on grid,>={Stealth[round]},
    every state/.style={draw=blue!50,very thick,fill=blue!20}, bend angle = 15]

  \node[state] (1)  {$1$};
  \node[state] (a) [right=of 1] {$a$};
  \node[state] (b) [below =of a] {$b$};
  \node[state] (c) [below =of 1] {$c$};

  \path[->] 
  (1) edge [loop above] node  {x/x} ()       
  (a)  edge node [above] {x/x+1} (1)
  (b)  edge node [right] {1/1} (a)
  (b)  edge [bend left] node [below] {2/2} (c)
  (c)  edge node [left] {1/1} (1)
  (c)  edge [bend left] node [above] {2/2} (b);
\end{tikzpicture}
  \small \\
  \small Third Grigorchuk group
\end{minipage}

% Third row
\begin{minipage}{0.45\textwidth}
  \centering
\begin{tikzpicture}[shorten >=1pt, node distance=2.5cm, on grid, >={Stealth[round]},
  every state/.style={draw=blue!50, very thick, fill=blue!20}, bend angle=15]

  \node[state] (1) at ( -2,  2) {$1$};
  \node[state] (a) at ( 2, 2) {$a$};
  \node[state] (b) at (2, -2) {$b$};
  \node[state] (c) at (0,0) {$c$};
  \node[state] (d) at (-2,  -2) {$d$};

  \path[->]
    (1) edge [loop above] node {x/x} ()
    (a) edge node [above] {x/x+1} (1)
    (b) edge node [right] {1/1}  (a)
    (b) edge [loop right] node {2/2} ()
    (c) edge node [above] {1/1} (a)
    (c) edge [bend right] node [above] {2/2} (d)
    (d) edge node [left] {1/1} (1)
    (d) edge [bend right] node [below] {2/2} (c);
\end{tikzpicture}
  \small \\
  \small Grigorchuk Erschler group
\end{minipage}
\caption{Famous automata groups whose closure is a group of finite type of depth $4$. The label $x/x$ means that the output is the same as the input and it passes to the same state. The label $x/x+1$ indicates that the output is one more than the input considering $X = \set{1,2}$ as numbers modulo $2$.}
\label{figure: automata groups closure depth 4}
\end{figure}
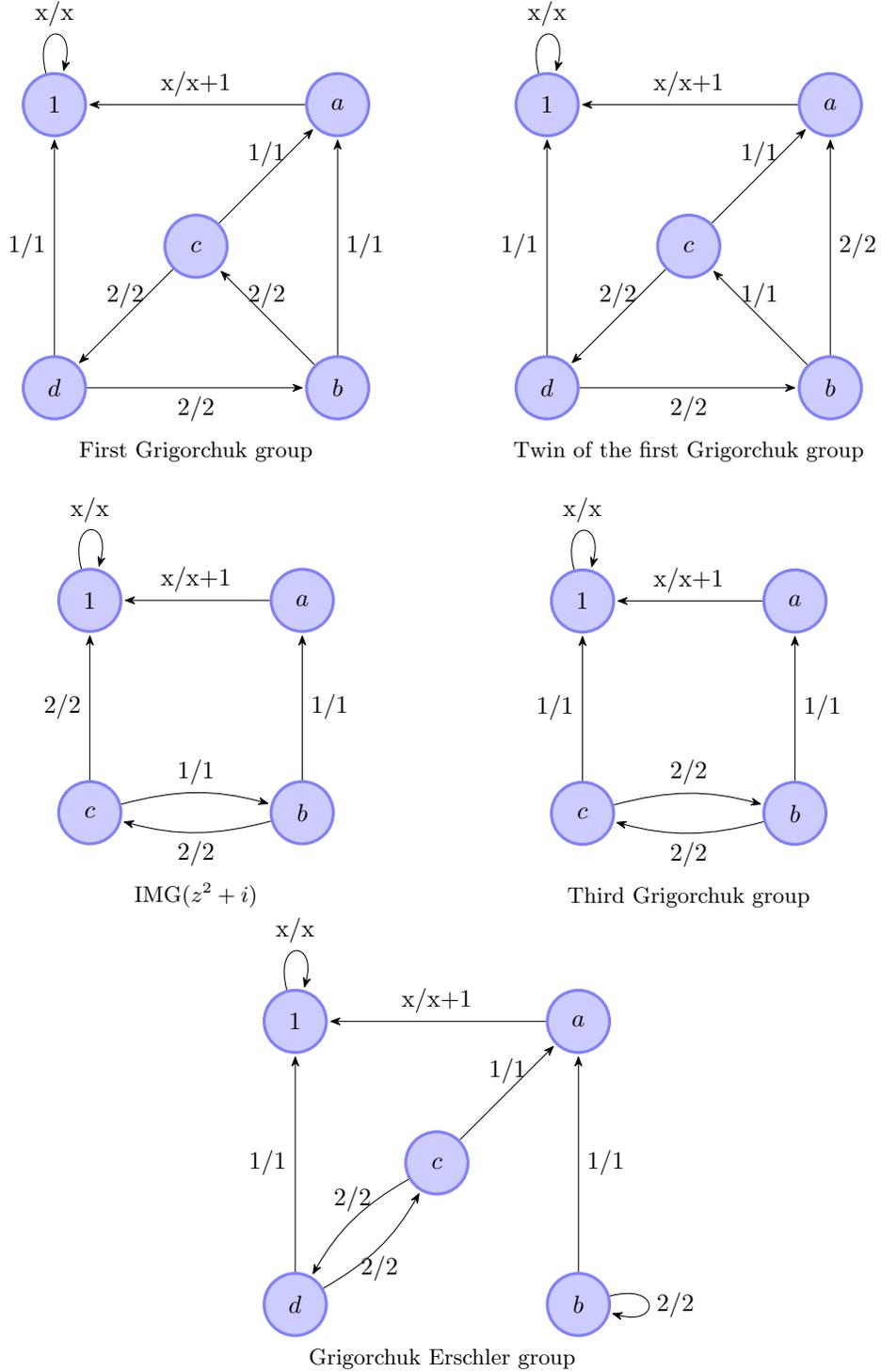

\subsection{The case $(d,D) = (3,2)$}
\label{subsection: the case (dD) = (32)}

Using GAP, we find 588 minimal patterns, where 422 of them give infinite groups of finite type and 166 of them give finite groups of finite type. In the 422 infinite groups of finite type, we have that 380 are not topologically finitely generated by \cref{proposition: finite type not tfg} (322 with $n = 2$ and 58 with $n = 3$) and 2 are topologically finitely generated by \cref{theorem: finite type tfg} with $n = 3$. The remaining 40 groups do not satisfy neither \cref{proposition: finite type not tfg} nor \cref{theorem: finite type tfg} with $n \leq 5$.

In terms of fractalness, the 166 finite groups are of course not fractal (and not level-transitive), so we just focus on the infinite groups of finite type. Among these 422 infinite groups of finite type, 68 of them are super strongly fractal, 76 are fractal but not strongly fractal and 278 are not fractal. 

In order to classify them up to isomorphism, the 166 finite groups of finite type split in four different classes: $\mathcal{R}_f = \set{\set{\id}, C_2, C_3, \Sym(3)}$. Among the remaining 422 groups in $\mathfrak{P}_\infty$, we have at most 36 classes, so the maximum amount of classes is $\# \mathfrak{R}_f + \# \mathfrak{R}_\infty = 40$. Now, 5 of them are in $W_3$ but only 4 are fractal. Among these 4, three of them are $T$-rigid and one is not by \cref{theorem: Jorge rigidity fractal}. The one that is not fractal does not satisfy \cref{theorem: Jorge rigidity}. This allows us to ensure at least 4 non-equivalent classes of infinite groups of finite type that are in $W_3$. If the group is not included in $W_3$, we can use \cref{theorem: finite type KP and rigidity}, obtaining 6 more groups that are $T$-rigid. Therefore, we have at least 15 isomorphic classes.

One of the automata groups whose closure appears in this list is the Hanoi towers group as it was shown in \cref{proposition: depth Hanoi towers closure}. Using \cref{proposition: fractal sf and ssf finite type} we conclude:

\begin{corollary}
The closure of the Hanoi towers group is super strongly fractal.
\end{corollary}

\subsection{The case $(d,D) = (3,3)$}
\label{subsection: the case (dD) = (33)}

To solve this case is necessary to analyze subgroups of the group $\Aut(T^3)$, that has order $6^{13}$. However, this is too big for the program GAP, so the focus was put on finding the infinite groups of finite type in $W_3$.

Inside $\pi_3(W_3)$, there are 44671 minimal patterns whose associated groups of finite type are infinite. This program had to be run for a week. Among those, only 216 are topologically finitely generated. The other groups are discarded using \cref{proposition: finite type not tfg}: 30109 with $n = 3$, 11421 with $n = 4$ and 2925 with $n = 5$, and the remaining 216 are proved to be topologically finitely generated by using $n = 5$ in \cref{theorem: finite type tfg}.

In terms of fractalness, the 216 groups of finite type are super strongly fractal. Due to this, we can apply \cref{theorem: finite type conjugated in Aut(T) iff gamma cycle} and \cref{theorem: Jorge rigidity fractal}, obtaining 12 different isomorphic classes. 

In order to give one minimal pattern of each isomorphic class explicitly, recall that the minimal patterns are in $\Aut(T^3) \hookrightarrow \Sym(27)$, so the generators of the minimal pattern subgroups can be given as permutations of the 27 vertices in the third level of the tree, enumerated from left to right. Define $a_3 := \pi_3((1,1,1)\sigma)$, where $\sigma = (1,2,3) \in \Sym(3)$ is the standard 3-cycle. The list is the following:

\begin{flalign*}
& \mathcal{P}_{1,1} = \<(7,9,8)(13,15,14)(16,17,18)(19,27,23,20,25,24,21,26,22), a_3>, \\
& \mathcal{P}_{1,2} = \<(7,8,9)(10,16,15,12,18,14,11,17,13), a_3>, \\
& \mathcal{P}_{2,1} = \<(7,8,9)(10,14,17)(11,15,18)(12,13,16)(25,27,26), a_3>, \\
& \mathcal{P}_{3,1} = \<(1,3,2)(7,9,8)(13,15,14)(16,17,18)(19,27,23)(20,25,24)(21,26,22), a_3 >, \\
& \mathcal{P}_{3,2} = \<(7,9,8)(10,16,14)(11,17,15)(12,18,13), a_3 >,\\
& \mathcal{P}_{4,1} = \<(1,2,3)(7,9,8)(13,14,15)(19,23,27,20,24,25,21,22,26), a_3 >, \\
& \mathcal{P}_{4,2} = \<(13,14,15)(16,17,18)(19,24,25,21,23,27,20,22,26), a_3 >, \\
& \mathcal{P}_{5,1} = \<(1,3,2)(13,15,14)(16,18,17)(19,27,24,20,25,22,21,26,23), a_3 >, \\
& \mathcal{P}_{5,2} = \<(7,8,9)(13,15,14)(19,25,23,21,27,22,20,26,24), a_3>, \\
& \mathcal{P}_{6,1} = \<(10,16,13)(11,17,14)(12,18,15)(19,22,26)(20,23,27)(21,24,25), a_3 >, \\
& \mathcal{P}_{6,2} = \<(10,13,18)(11,14,16)(12,15,17)(19,25,23)(20,26,24)(21,27,22), a_3 >, \\
& \mathcal{P}_{6,3} = \<(10,13,16)(11,14,17)(12,15,18)(19,25,24)(20,26,22)(21,27,23), a_3 >.
\end{flalign*}

Analyzing isomorphic classes of the groups $\mathcal{P}_{i,j}$ we conclude the following result:

\begin{Proposition}
The minimal pattern subgroups $\mathcal{P}_{i,j}$ and $\mathcal{P}_{i',j'}$ are isomorphic if and only if $i = i'$. Moreover, all the isomorphisms are conjugations by an element in $\Aut(T^3)$. As a consequence, if $T$ is a $d$-regular tree, $D$ a natural number and $\mathcal{P}, \mathcal{Q}$ are two minimal patterns isomorphic by a conjugation by an element in $\Aut(T^D)$, it is not necessarily true that their associated groups of finite type $G_\mathcal{P}$ and $G_\mathcal{Q}$ are isomorphic.
\label{proposition: P iso Q does not imply GP iso GQ}
\end{Proposition}

\begin{proof}
The first part follows by using GAP. In order to show that the minimal patterns are isomorphic by a conjugation in $\pi_3(\Aut(T^3))$, we can compute the directed graphs $\Gamma_{\mathcal{P}_{i,j}, \mathcal{P}_{i,j'}}$. If the graph has at least one vertex, this means that the minimal pattern subgroups are conjugated in $\pi_3(\Aut(T^3))$ (see \cref{definition: SPQ}).

We find three different non-empty directed graphs. The option 1 (\cref{figure: option 1 graph 33}) corresponds to $\Gamma_{\mathcal{P}_{i,1}, \mathcal{P}_{i,2}}$ with $i \in \set{1,3,4,5}$. The option 2 (\cref{figure: option 2 graph 33}) corresponds to $\Gamma_{\mathcal{P}_{6,1}, \mathcal{P}_{6,3}}$ and the option 3 (\cref{figure: option 3 graph 33}) corresponds to $\Gamma_{\mathcal{P}_{6,1}, \mathcal{P}_{6,2}}$ and $\Gamma_{\mathcal{P}_{6,2}, \mathcal{P}_{6,3}}$. As it can be seen, none of the directed graphs have a cycle so the associated groups of finite type are not isomorphic by \cref{theorem: finite type conjugated in Aut(T) iff gamma cycle} and \cref{theorem: Jorge rigidity fractal}. Since the graphs have vertices, the minimal patterns subgroups are conjugated in $\pi_3(\Aut(T^3))$.

\begin{figure}
  \centering
\begin{minipage}{0.55\textwidth}
  \centering
\begin{tikzpicture}[shorten >=1pt,node distance=2cm,on grid,>={Stealth[round]}, every state/.style={draw=black!20,thick,minimum size=15pt,fill=white!5}, bend angle = 15]

  \node[state] (1)  {};
  \node[state] (2) [right=of 1] {};
  \node[state] (3) [right=of 2] {};
  \node[state] (4) [below=of 1] {};
  \node[state] (5) [below=of 2] {};
  \node[state] (6) [below=of 3] {};

  \path[->] 
  (1) edge node [above] {} (4)
  (1) edge node [above] {} (5)
  (1) edge node [above] {} (6)
  (2) edge node [above] {} (4)
  (2) edge node [above] {} (5)
  (2) edge node [above] {} (6)
  (3) edge node [above] {} (4)
  (3) edge node [above] {} (5)
  (3) edge node [below] {} (6);
\end{tikzpicture}
  \caption{Option 1}
  \label{figure: option 1 graph 33}
\end{minipage}

\vspace{2em}

\begin{minipage}{0.48\textwidth}
  \centering
\begin{tikzpicture}[shorten >=1pt,node distance=2cm,on grid,>={Stealth[round]}, every state/.style={draw=black!20,thick,minimum size=10pt,fill=white!5}, bend angle = 15]

  \node[state] (1)  {};
  \node[state] (2) [right=of 1] {};
  \node[state] (3) [right=of 2] {};
  \node[state] (4) [below=of 1] {};
  \node[state] (5) [below=of 2] {};
  \node[state] (6) [below=of 3] {};
  \node[state] (7) [below=of 4] {};
  \node[state] (8) [below=of 5] {};
  \node[state] (9) [below=of 6] {};
  \node[state] (10) [below=of 7] {};
  \node[state] (11) [below=of 8] {};
  \node[state] (12) [below=of 9] {};
  
  \path[->] 
  (1) edge node [above] {} (4)
  (1) edge node [above] {} (5)
  (1) edge node [above] {} (6)
  (2) edge node [above] {} (4)
  (2) edge node [above] {} (5)
  (2) edge node [above] {} (6)
  (3) edge node [above] {} (4)
  (3) edge node [above] {} (5)
  (3) edge node [below] {} (6)
  (4) edge node [above] {} (7)
  (4) edge node [above] {} (8)
  (4) edge node [above] {} (9)
  (5) edge node [above] {} (7)
  (5) edge node [above] {} (8)
  (5) edge node [above] {} (9)
  (6) edge node [above] {} (7)
  (6) edge node [above] {} (8)
  (6) edge node [below] {} (9);
\end{tikzpicture}
  \caption{Option 2}
  \label{figure: option 2 graph 33}  
\end{minipage}
\hfill
\begin{minipage}{0.48\textwidth}
  \centering
\begin{tikzpicture}[shorten >=1pt,node distance=2cm,on grid,>={Stealth[round]}, every state/.style={draw=black!20,thick,minimum size=10pt,fill=white!5}, bend angle = 15]

  \node[state] (1)  {};
  \node[state] (2) [right=of 1] {};
  \node[state] (3) [right=of 2] {};
  \node[state] (4) [below=of 1] {};
  \node[state] (5) [below=of 2] {};
  \node[state] (6) [below=of 3] {};
  \node[state] (7) [below=of 4] {};
  \node[state] (8) [below=of 5] {};
  \node[state] (9) [below=of 6] {};
  \node[state] (10) [below=of 7] {};
  \node[state] (11) [below=of 8] {};
  \node[state] (12) [below=of 9] {};
\end{tikzpicture}
  \caption{Option 3}
  \label{figure: option 3 graph 33}  
\end{minipage}
\end{figure}
\end{proof}

Finally, we are interested in finding known automata groups whose closures correspond to any of the 216 groups of finite type studied in this subsection.

Given a vector $\vec{\alpha} \in C_3^2$, the Grigorchuk-Gupta-Sidki (GGS) group with vector $\vec{\alpha} = (\alpha_1, \alpha_2)$ is an automata group acting on the ternary tree defined as $G_{\vec{\alpha}} := \<a,b>$, where $a = (1,1,1)\sigma$, $\sigma$ is the 3-cycle $(1,2,3) \in \Sym(3)$ and $b = (a^{\alpha_1}, a^{\alpha_2},b)$. GGS groups can be defined for any $d$-regular tree with $d$ a prime power and they have been broadly studied by many authors (see for example \cite{DiDomenico2023, FernandezAlcober2017, GGS1983,  Petschick2019}). 

Notice that $b^{-1} = (a^{-\alpha_1},a^{-\alpha_2},b^{-1})$. Consequently $G_{\vec{\alpha}} = G_{-\vec{\alpha}}$. On the other hand, the group $G_{(0,1)} \simeq G_{(1,0)}$ as it is obtained from the other by interchanging the first two letters in the alphabet of the automaton defining the GGS group. Finally, the group $G_{(0,0)}$ is finite. Hence, this gives us three vectors to analyze: $(1,0), (1,1)$ and $(1,2)$. By \cite[Theorem 3.7]{DiDomenico2023}, the group $G_{(1,1)}$ is not branch.

Using \cref{theorem: closure automaton and finite type}, we obtain that $G_{\mathcal{P}_{3,2}} = \overline{G_{(1,0)}}$ and that $G_{\mathcal{P}_{6,2}} = \overline{G_{(1,2)}}$. In particular, their closures are not isomorphic.

\section{The maximal branching subgroup} \label{section: The maximal branching subgroup}

By \cref{proposition: finite type is closed self-similar and regular branch}, we know that groups of finite type are regular branch and by \cref{lemma: unique maximal reg branch subgroup}, we know there exists a maximal branching subgroup. In this section, we will give an explicit form for the maximal branching subgroup of groups of finite type (\cref{theorem: properties KP}) and we will use this to give a tool that helps to prove if an abstract group in $\Aut(T)$ is regular branch (\cref{theorem: regular branch when closure is finite type}).

\subsection{The maximal branching subgroup of groups of finite type} 

Recall the map $\delta_v$ defined in \cref{equation: map deltav}. Let $T$ be a $d$-regular tree, $D$ a natural number and $\mathcal{P} \leq \Aut(T^D)$ a minimal pattern subgroup. Define $$K_\mathcal{P} = \set{g \in G_\mathcal{P}: \delta_v(g) \in G_\mathcal{P} \text{ for all $v$ vertex in $T$}}$$ and $\widetilde{K_\mathcal{P}} = \pi_D(K_\mathcal{P})$.

\begin{theorem}
Let $T$ be a $d$-regular tree, $D$ a natural number and $\mathcal{P} \leq \Aut(T^D)$ a minimal pattern subgroup. Then

\begin{enumerate}
    \item $K_\mathcal{P}$ is a subgroup of $G_\mathcal{P}$,
    \item $\St_{G_\mathcal{P}}(D-1) \leq K_\mathcal{P}$,
    \item $K_\mathcal{P} = \pi_D^{-1}(\widetilde{K_\mathcal{P}})$,
    \item $[G_\mathcal{P}: K_\mathcal{P}] = [\mathcal{P}: \widetilde{K_\mathcal{P}}]$,
    \item $K_\mathcal{P}$ is the maximal branching subgroup of $G_\mathcal{P}$,
    \item $K_\mathcal{P}$ is clopen in the congruence topology,
    \item $\mathcal{K}_\mathcal{P} \lhd G_\mathcal{P}$ if and only if $\widetilde{K_\mathcal{P}} \lhd \mathcal{P}$.
    \item If $G_\mathcal{P}$ is fractal, then $K_\mathcal{P} \lhd G_\mathcal{P}$.
\end{enumerate}
\label{theorem: properties KP}
\end{theorem}

\begin{proof}
For (1), for any $v$ vertex in $T$, we have $\delta_v(1) = 1 \in G_\mathcal{P}$ and if $g,h \in K_\mathcal{P}$, then $\delta_v(gh^{-1}) = \delta_v(g) \delta_v(h)^{-1} \in G_\mathcal{P}$ since $\delta_v$ is a morphism of groups and $G_\mathcal{P}$ is a group. 

For (2), since $G_\mathcal{P}$ is regular branch over $\St_{G_\mathcal{P}}(D-1)$, we have that if $s$ is an element in $\St_{G_\mathcal{P}}(D-1)$, then $\delta_v(s) \in G_\mathcal{P}$ for all $v$ vertex in $T$. 

For (3), we have $\pi_D^{-1}(\widetilde{K_\mathcal{P}}) = \<K_\mathcal{P}, \St_{G_\mathcal{P}}(D)> = K_\mathcal{P}$ by point 2. 

For (4), we have $[G_\mathcal{P} : K_\mathcal{P}] = [\pi_D^{-1}(\mathcal{P}): \pi_D^{-1}(\widetilde{K_\mathcal{P}})] = [\mathcal{P}: \widetilde{K_\mathcal{P}}]$ since $\pi_D$ is surjective. 

For (5), by point 4 we have $K_\mathcal{P} \leq_f G_\mathcal{P}$. Let $g = (g_1,\dots,g_d) \in (K_\mathcal{P})_1$, the geometric product of $K_\mathcal{P}$. For any $v$ vertex in $T$, we have $$\delta_v(g) = \prod_{i = 1}^d \delta_{vi}(g_i) \in G_\mathcal{P}$$ since each $g_i \in K_\mathcal{P}$. Hence $(K_\mathcal{P})_1 \leq K_\mathcal{P}$. Then $(K_\mathcal{P})_1 \leq_f K_\mathcal{P}$, by \cref{lemma: regular branch for self-similar groups}, as $G_\mathcal{P}$ is self-similar.

To prove that $K_\mathcal{P}$ is the maximal branching subgroup, let $H$ be a subgroup such that $H_1 \leq_f H \leq_f G_\mathcal{P}$. Let $v \in \mathcal{L}_n$ and $h \in H$. Then, we have $\delta_v(h) \in H_n$ and since $H_1 \leq H$, by induction $H_n \leq H \leq G_\mathcal{P}$, so $h \in K_\mathcal{P}$.

For (6), it follows from the fact that $K_\mathcal{P} = \pi_D^{-1}(\widetilde{K_\mathcal{P}})$, the map $\pi_D$ is continuous and the topology in $\mathcal{P}$ is discrete. 

For (7), we just use that $\pi_D$ is surjective for the direct. 

For (8), it follows from \cite[Lemma 1.3]{Bartholdi2012}.
\end{proof}

\subsection{maximal branching subgroup in groups whose closure is a group of finite type}

Suppose that $G$ is an abstract group whose closure is a group of finite type. Can we use properties of its closure to know whether the group $G$ is regular branch? The answer is yes and we develop some tools in this subsection for this purpose. We start with the following general lemmas of abstract groups and their closures.

\begin{Lemma}
Let $\mathcal{G}$ be a topological group, a subgroup $G \leq \mathcal{G}$, the closure $\overline{G}$ of $G$ in $\mathcal{G}$ and $H$ a closed subgroup with finite index in $\overline{G}$. Then $\overline{G \cap H} = H$. 
\label{lemma: closure(GcapH) = H}
\end{Lemma}

\begin{proof}
First of all, $G \cap H \subseteq H$, so taking closure and using that $H$ is closed, we get the inclusion $\overline{G \cap H} \subseteq H$.

Now, take $\set{1 = g_1,\dots,g_n}$ a set of representatives for $\overline{G}/H$ and write $\overline{G}$ as the disjoint union $\overline{G} = \bigcup_{i = 1}^n g_i H$. Intersecting with $G$ and taking the closure we get $$\overline{G} = \overline{\bigcup_{i = 1}^n g_i (G \cap H)} = \bigcup_{i = 1}^n g_i (\overline{G \cap H})$$
since the closure of the union of finitely many sets is the union of the closures.

But now, we have two disjoint descriptions of $\overline{G}$, where $g_i (\overline{G \cap H}) \subseteq g_i H$ for all $i = 1,\dots,n$, so necessarily $g_i (\overline{G \cap H}) = g_i H$. In particular, if we take $i = 1$, then $\overline{G \cap H} = H$. 
\end{proof}

\cref{lemma: closure(GcapH) = H} is not true if $H$ does not have finite index. For an example of uncountable index, consider $\mathcal{G} = \RR$, $G = \QQ$ and $H = \<\pi>$. Then $\overline{G} = \RR$ and  $\overline{G \cap H} = \set{0}$. For a countable case, consider $\mathcal{G} = \QQ$, $G = \<\ZZ,\frac{1}{2^k}: k \geq 1>$ and $H = \<\frac{1}{3}>$. Then $\overline{G} = \QQ$ and since it is countable, any index will be countable. Then $\overline{G \cap H} = \ZZ$.

\begin{Lemma}
Let $\mathcal{G}$ be a topological group and $K,H$ two subgroups of $\mathcal{G}$ such that $K \leq_f H$. Then $[\overline{H}: \overline{K}] \leq [H:K]$.
\label{lemma: index closure subgroups and subgroups}
\end{Lemma}

\begin{proof}
Let $h_1,\dots,h_n$ be a set of representatives of $H/K$, so $H = \bigcup_{i = 1}^n h_i K$. Taking closure, $$\overline{H} = \overline{ \bigcup_{i = 1}^n h_i K} = \bigcup_{i = 1}^n \overline{h_i K} = \bigcup_{i = 1}^n h_i \overline{K},$$ where the last equality was because the translation is a homeomorphism. The intersection may not be pairwise disjoint, so we obtain the inequality.
\end{proof}

\begin{theorem}[{\cite[Theorem 2]{GarridoSunic2023}}]
Let $G$ be a self-similar, level-transitive and regular branch group with maximal branching subgroup $K$. Then $G$ has trivial rigid kernel if and only if there exists $n \geq 1$ such that $K \geq \St_G(n)$.
\label{theorem: trivial rigid kernel equivalent conditions maximal regular branch}
\end{theorem}

We are now ready to prove the main theorem of this subsection:

\begin{theorem}
Let $G \leq \Aut(T)$ be a self-similar group such that $\overline{G} = G_\mathcal{P}$ is a group of finite type. Let $K$ be a subgroup of $G$ such that the geometric product $K_1 \leq G$, the subgroup $K \leq K_\mathcal{P}$ and $[G:K] = [G_\mathcal{P}: K_\mathcal{P}]$. Then:
\begin{enumerate}
\item $K = G \cap K_\mathcal{P}$,
\item $G$ is regular branch over $K$,
\item $K$ is the maximal branching subgroup of $G$.
\item If $G$ is level-transitive, then $G$ has trivial rigid kernel.
\end{enumerate}
\label{theorem: regular branch when closure is finite type}
\end{theorem}

\begin{proof}
For (1), by hypothesis, $K \leq K_\mathcal{P}$ and $K \leq G$, so $K \leq G \cap K_\mathcal{P}$. Then, 
\begin{align*}
[G:K] = [G: G \cap K_\mathcal{P}][G \cap K_\mathcal{P}:K] \geq [\overline{G}: \overline{G \cap K_\mathcal{P}}][G \cap K_\mathcal{P}:K] = \\
[\overline{G}: K_\mathcal{P}][G \cap K_\mathcal{P}:K] = [G: K][G \cap K_\mathcal{P}:K]
\end{align*}
where in the first inequality \cref{lemma: index closure subgroups and subgroups} was used, in the next equality we used \cref{lemma: closure(GcapH) = H} and in the last equality the hypothesis over the indices. Cancelling out $[G:K]$ on both sides, we get that $1 \geq [G \cap K_\mathcal{P}:K]$ and consequently $G \cap K_\mathcal{P} = K$.

For (2), we start proving that the geometric product $K_1 = G \cap (K_\mathcal{P})_1$. Indeed, since $K \leq K_\mathcal{P}$ and $K_1 \leq G$ by hypothesis, we have that $K_1 \leq G \cap (K_\mathcal{P})_1$. Take now $g = (k_1,\dots,k_d) \in G$ such that $k_i \in K_\mathcal{P}$ for all $i = 1,\dots,d$. Since $G$ is self-similar and $g \in G$, then the sections $k_i \in G$, so $k_i \in G \cap K_\mathcal{P} = K$ for the point 1. Then $g \in K_1$, concluding that $K_1 = G \cap (K_\mathcal{P})_1$.

Now, by hypothesis $[G:K] = [G_\mathcal{P}: K_\mathcal{P}]$, so $[G:K] < \infty$ by \cref{theorem: properties KP}. Furthermore, 
$(K_\mathcal{P})_1 \leq K_\mathcal{P}$, so $K_1 = G \cap (K_\mathcal{P})_1 \leq G \cap K_\mathcal{P} = K$. Then, $G$ is regular branch over $K$ by \cref{lemma: regular branch for self-similar groups}.

For (3), suppose that $H$ is another subgroup regular branching $G$, namely, we have the inclusions $H_1 \leq_f H \leq_f G$. Taking closure, we keep the inclusions and the finiteness of the indices by \cref{lemma: index closure subgroups and subgroups}, so $\overline{H_1} = (\overline{H})_1 \leq_f \overline{H} \leq_f \overline{G}$. Consequently, $\overline{G}$ is regular branch over $\overline{H}$ and by \cref{theorem: properties KP}, we have $\overline{H} \leq K_\mathcal{P}$. Finally, intersecting with $G$, we get that $G \cap \overline{H} \leq G \cap K_\mathcal{P} = K$. In particular $H \leq G \cap \overline{H} \leq K$, proving the result.

For (4), by \cref{theorem: properties KP} we have $K_\mathcal{P} \geq \St_{G_\mathcal{P}}(D-1)$, so $K = G \cap K_\mathcal{P} \geq \St_G(D-1)$ and by \cref{theorem: trivial rigid kernel equivalent conditions maximal regular branch}, then $G$ has trivial rigid kernel.
\end{proof}

If $G$ is a finitely generated group in the conditions of \cref{theorem: regular branch when closure is finite type}. We can also give a pseudo-algorithm that helps to find a candidate $K$ for $G$. Let $S$ be a finite generating set of $G$ and assume $\widetilde{K_\mathcal{P}} \lhd \mathcal{P}$. Then $\mathcal{P}/\widetilde{K_\mathcal{P}}$ has a presentation with a generating set $\pi_D(S) \widetilde{K_\mathcal{P}}$ and finitely many relations $r_1,\dots,r_m$ since the group is finite. This means that $\widetilde{K_\mathcal{P}} = \<r_1,\dots,r_n>^\mathcal{P}$ where $r_1,\dots,r_n$ are words in terms of the elements in $S$. Our candidate for \cref{theorem: regular branch when closure is finite type} is $$K = \<r_1,\dots,r_m>^G.$$ 

First, of course $K \lhd G$ and since we added all the relations of $\mathcal{P}/\widetilde{K_\mathcal{P}}$, we have the isomorphism of groups $G/K \rightarrow\mathcal{P}/\widetilde{K_\mathcal{P}} \simeq G_\mathcal{P}/K_\mathcal{P}$ and consequently we have $K \leq K_\mathcal{P}$ and the equalities $[G:K] = [\mathcal{P}: \widetilde{K_\mathcal{P}}] = [G_\mathcal{P}: K_\mathcal{P}]$. It remains to check that $K_1 \leq G$. The following lemma can help in this task:

\begin{Lemma}
Let $G$ be a fractal group such that $(x,1,\dots,1) \in G$.

\begin{enumerate}
\item If $g \in G$ then $(gxg^{-1},1,\dots,1) \in G$.
\item If $y \in G$ then $([x,y],1,\dots,1) \in G$.
\item Let $x,y,z \in G$. Then $(yxz,1,\dots,1) \in G$ if and only if $(yz,1,\dots,1) \in G$. 
\end{enumerate}
\label{lemma: tricks for Kx...xK}
\end{Lemma}

\begin{proof}
For (1), since $G$ is fractal, there exists an element $h$ fixing the first vertex $1$ such that $h|_1 = g$. Then $$h(x,1,\dots,1)h^{-1} = (gxg^{-1},1\dots,1).$$ 

For (2), we start using the point 1 to justify that $(yx^{-1}y^{-1},1,\dots,1) \in G$. Then multiplying by $(x,1,\dots,1)$, we get that $(xyx^{-1}y^{-1},1,\dots,1) = ([x,y],1,\dots,1) \in G$. 
 
For (3), 
\begin{multline*}
(yxz,1,\dots,1) = (yzx[x^{-1},z^{-1}],1,\dots,1) = \\(yz,1,\dots,1)(x,1,\dots,1)([x^{-1},z^{-1}],1,\dots,1).
\end{multline*} 
Since $(x,1,\dots,1) \in G$, the last two terms are in $G$, giving the result. 
\end{proof}

The most useful property of \cref{lemma: tricks for Kx...xK} will be the last one, as it allows us to annihilate terms that we already know that are in $G$.  

\begin{Example}
\begin{figure}[t]
\begin{tikzpicture}[shorten >=1pt,node distance=3cm,on grid,>={Stealth[round]},
    every state/.style={draw=blue!50,very thick,fill=blue!20}, bend angle = 15]

  \node[state] (1)  {$1$};
  \node[state] (a) [right=of 1] {$a$};
  \node[state] (b) [below=of a] {$b$};
  \node[state] (c) [below=of 1] {$c$};

  \path[->] 
  (1) edge [loop above] node  {x/x} ()       
  (a)  edge node [above] {x/x+1} (1)
  (b)  edge node [right] {1/2} (a)
  (b)  edge [bend right] node [above] {2/1} (c)
  (c)  edge [loop above] node {1/1} ()
  (c)  edge [bend right] node [below] {2/2} (b);
\end{tikzpicture}
\caption{Another automaton group whose closure is the same as the closure of the Grigorchuk group. The label $x/x$ means that the output is the same as the input and it passes to the same state. The label $x/x+1$ indicates that the output is one more than the input considering $X = \set{1,2}$ as numbers modulo $2$.}
\label{figure: automaton max reg brach}
\end{figure}
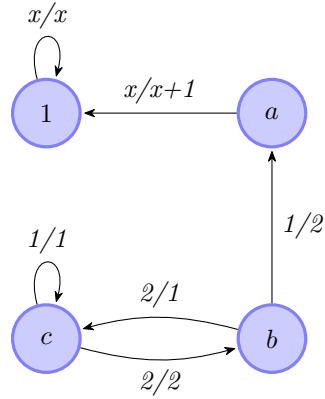

Consider the group $G$ generated by the automaton in the \cref{figure: automaton max reg brach}. Then, $G$ is generated by the automorphisms $a$, $b$ and $c$ where $a = (1,1)\sigma$, $b = (a,c)\sigma$, $c = (c,b)$ and $\sigma$ is the non-trivial permutation of $\Sym(2)$. The group $G$ is given by an automaton, so it is in particular self-similar. The group is also strongly fractal as $c^{-1} b^2 = (a, b^{-1}ac)$, $aca = (b,c)$ and $c = (c,b)$, and conjugating these three elements with $a$, we get $a,b$ and $c$ on the second coordinate. By \cref{theorem: closure automaton and finite type}, we find that $\overline{G}$ is the closure of the first Grigorchuk group, that is a group of finite type of depth $4$ as we already observed in \cref{subsection: the case (dD) = (24)}. Since $G$ is strongly fractal, then $\overline{G} = G_\mathcal{P}$ is fractal and by \cref{theorem: properties KP}, then $\widetilde{K_\mathcal{P}} \lhd \mathcal{P}$. Moreover, the quotient $\mathcal{P}/\widetilde{K_\mathcal{P}}$ has the presentation  $$\mathcal{P}/\widetilde{K_\mathcal{P}} = \<a,b,c \mid a^2, b^2, c^2, (cba)^2, (cacb)^2, (bc)^4>,$$ so we need to prove that the last five relations $r_i$ (because $a^2 = 1$) satisfy that $(r_i,1) \in G$, since by conjugation by $a$, we obtain $(1,r_i) = a(r_i,1)a \in G$.

We have $ba = (c,a)$, so $(ba)^2 = \boxed{(c^2,1) \in G}$. Also $c^2 = (c^2,b^2)$, so we deduce $c^2(ba)^{-2} = (1,b^2)$ and therefore $\boxed{(b^2,1) \in G}$. 

Continuing, we have $acb = (ba,c^2)$ and since $(1,c^2) \in G$, then $\boxed{(ba,1) \in G}$ and also $\boxed{(ab^{-1},1) \in G}$. 

We now use \cref{lemma: tricks for Kx...xK}.(3) to prove the remaining relations. For example 
\begin{align*}
((cba)^2,1) = (c(ba)c(ba),1) \in G \Leftrightarrow (c^2,1) \in G.
\end{align*}

For ``$(cacb)^2$" we will need one relation of $G$. It can be verified that $$bacb^{-2}acbc^{-1}ac^{-1}a = 1,$$ so $cbc^{-1}ac^{-1}a = ab^2c^{-1}ab^{-1}$. Then 
\begin{align*}
(cacbcacb,1) = (cacbc^{-1}(c^2)ac^{-1}(c^2)a(ab),1) \in G \Leftrightarrow \\ (ca(cbc^{-1}ac^{-1}a),1) = (caab^2c^{-1}ab^{-1},1) \in G \Leftrightarrow \\ (ab^{-1},1) \in G.
\end{align*}

Finally, for ``$(cb)^4$", 
\begin{align*}
(c(ab)c(aca))^2 = ((c,b)(a,c)(c,b)(b,c))^2 = ((cacb)^2, (bc)^4).
\end{align*}
Since we already proved that $((cacb)^2,1) \in G$, then $(1,(bc)^4) \in G$ and consequently $((bc)^4,1) \in G$. 

By \cref{theorem: regular branch when closure is finite type}, we conclude that $G$ is regular branch over $$K = \<b^2,c^2,(cba)^2,(cacb)^2,(bc)^4>^G$$ and $K$ is the maximal branching subgroup. As $G$ is level-transitive, by \cref{theorem: regular branch when closure is finite type} we obtain that the rigid kernel is trivial. 
\end{Example}

\section{List of automata groups whose closure is a group of finite type}
\label{section: List of automata groups whose closure is a group of finite type}

To culminate this article, we present a list of automata such that the closure of the automata groups coincide with the infinite topologically finitely generated groups of finite type studied in \cref{section: analysis of different cases}. The search was made in GAP \cite{GAP} by constructing all the automata groups with $r$ states and applying \cref{theorem: closure automaton and finite type}. It is worth mentioning that even though we are just putting one automata per group of finite type, many other automatas whose generated groups have the same closure but are not isomorphic were found.  

In the following figures, when in an automaton appears one arrow labeled as $x/x$, this means that the output is the same as the input and it passes to the same state. The label $x/x+1$ indicates that the output is one more than the input considering $X$ as numbers modulo $d$. Similarly, the label $x/x-1$ indicates that the output is one less than the input considering $X$ as numbers modulo $d$. 

By the analysis done in \cref{section: analysis of different cases}, we know that in the binary tree, there are no infinite topologically finitely generated groups of finite type for depths $2$ and $3$. 

\subsection{List of automata groups for the case $(d,D) = (2,4)$}

In this case, we have 32 infinite topologically finitely generated groups of finite type (see \cref{subsection: the case (dD) = (24)}). The search showed that no $3$-state automata group has as closure these 32 groups. There are 12 of them that coincide with the closure of $4$-state automata groups and 18 of them that coincide with the closure of $5$-state automata groups. There are 2 groups of finite type that do not coincide with the closure of an automata group with at most 5 states.

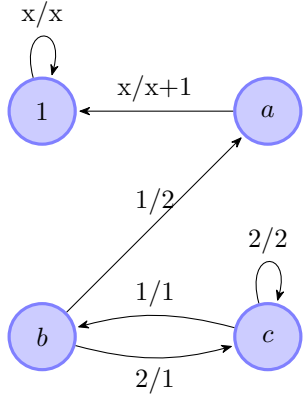
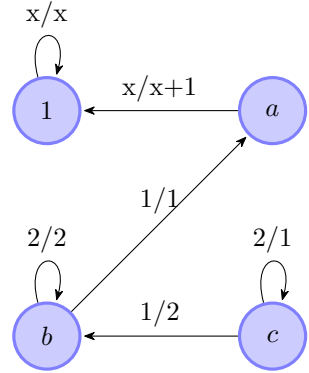
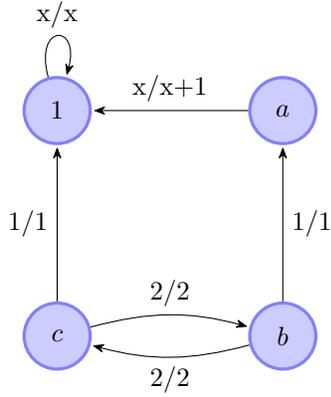
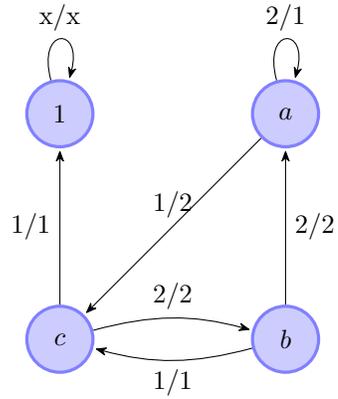
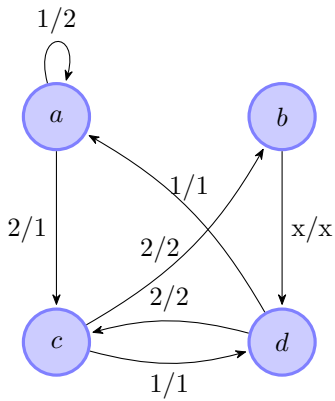
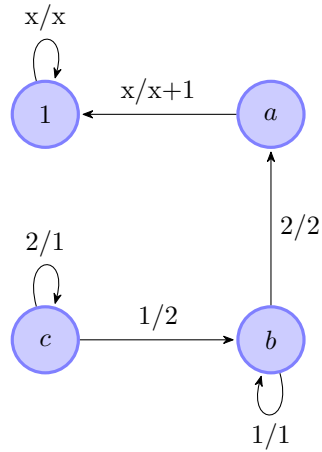
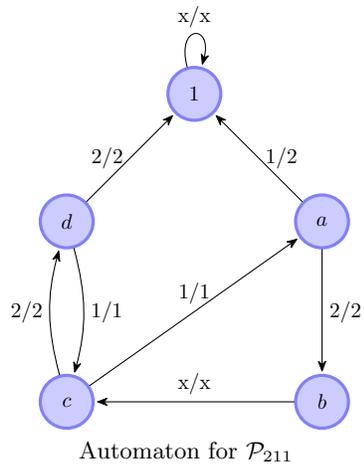
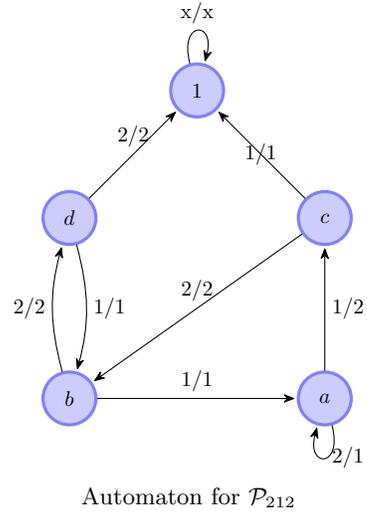
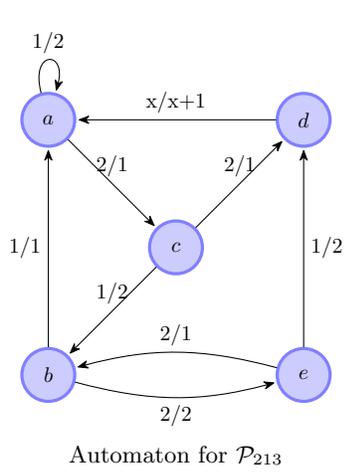
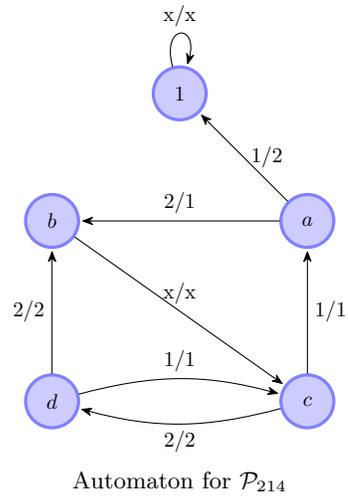
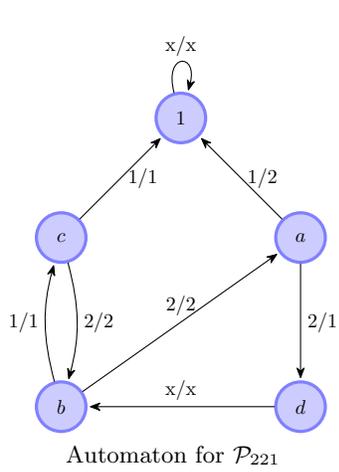
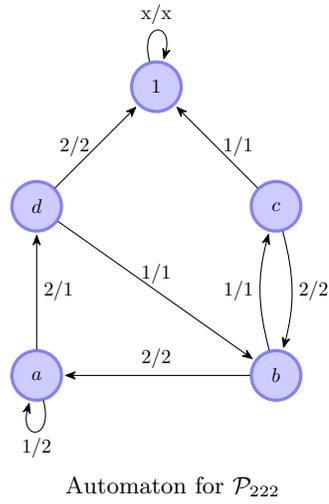
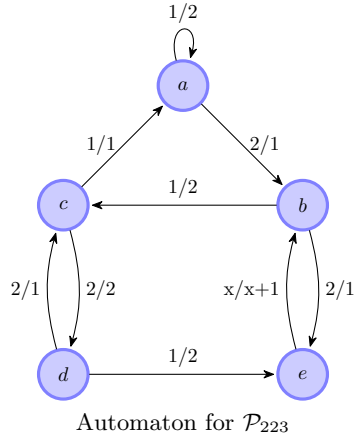
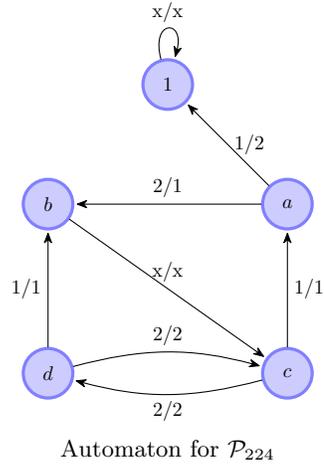
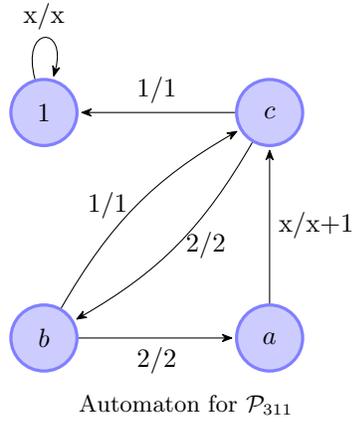
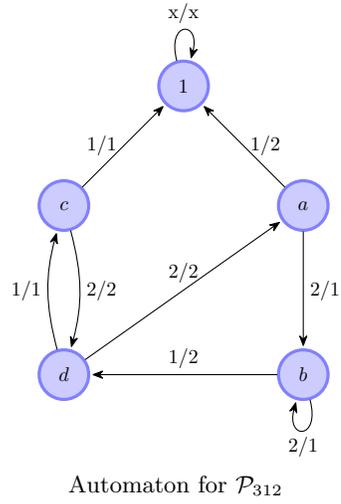
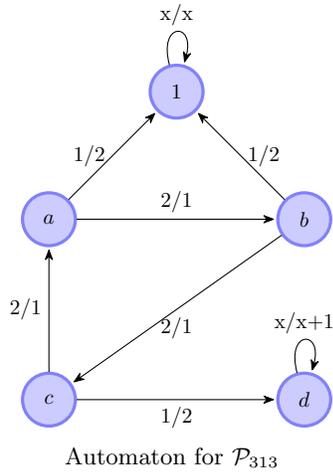
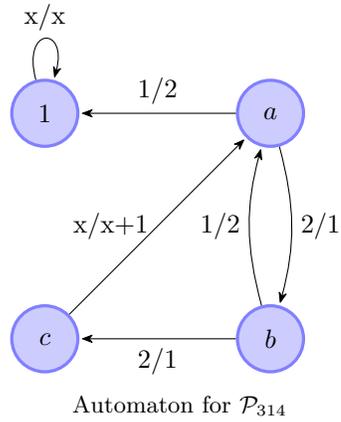
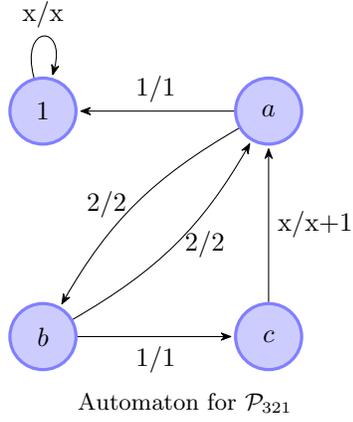
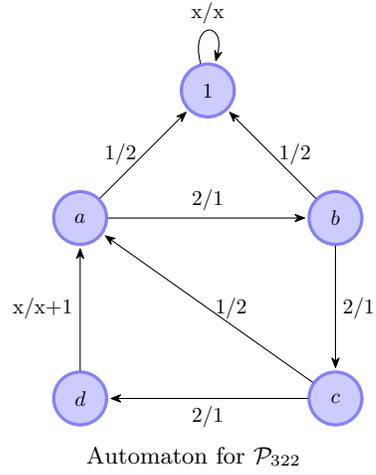
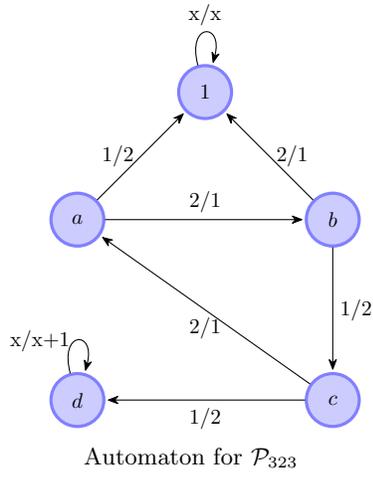
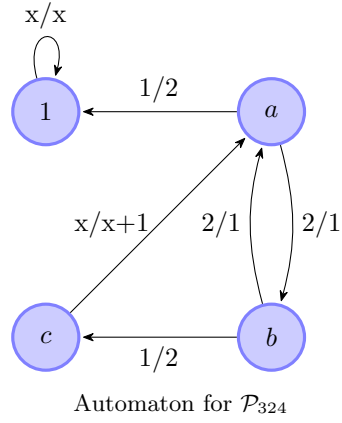
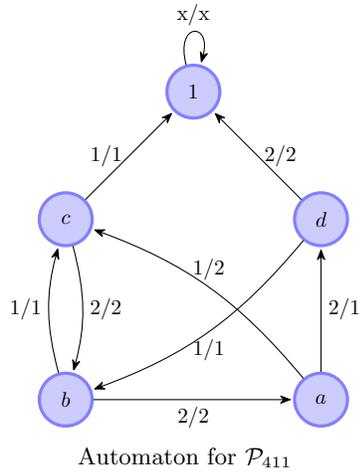
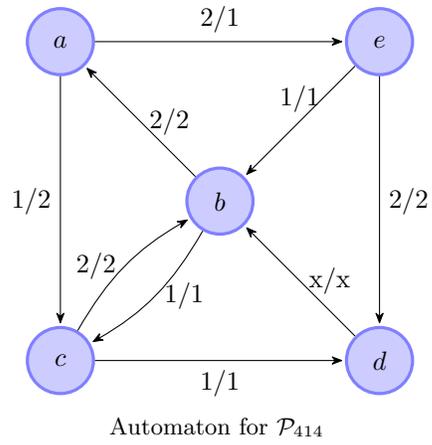
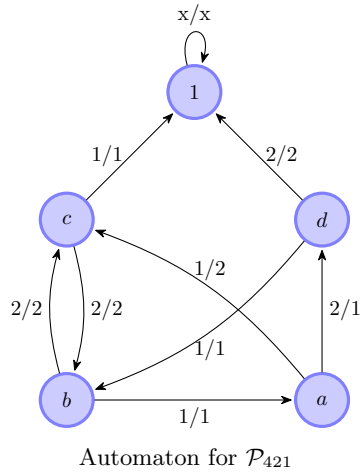
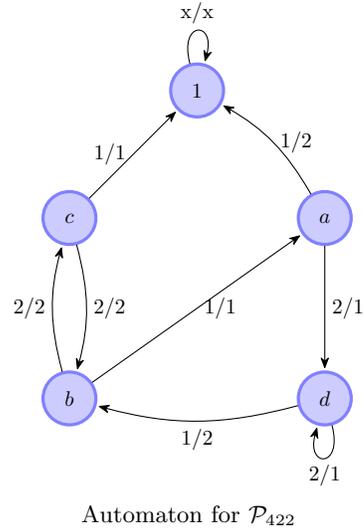
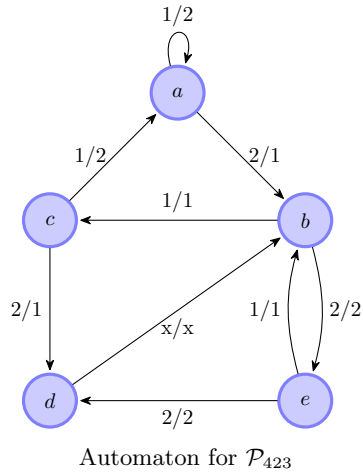
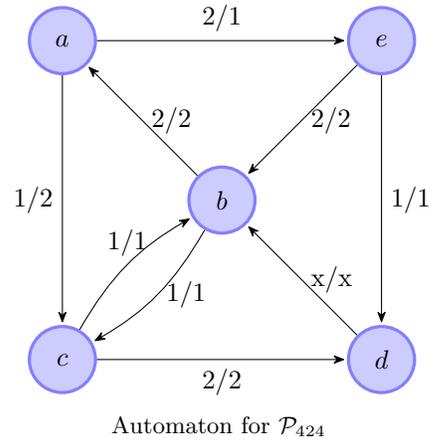
\begin{minipage}{0.45\textwidth}
    \centering
    \begin{tikzpicture}[shorten >=1pt,node distance=3cm,on grid,>={Stealth[round]},
        every state/.style={draw=blue!50,very thick,fill=blue!20}, bend angle = 15]
    
      \node[state] (1)  {$1$};
      \node[state] (a) [right=of 1] {$a$};
      \node[state] (b) [below =of a] {$b$};
      \node[state] (c) [below =of 1] {$c$};
    
      \path[->] 
      (1) edge [loop above] node  {x/x} ()       
      (a)  edge node [above] {x/x+1} (1)
      (b)  edge node [right] {1/1} (a)
      (b)  edge [bend left] node [below] {2/2} (c)
      (c)  edge [bend left] node [above] {1/1} (b)
      (c)  edge node [left] {2/2} (1);
    \end{tikzpicture}
    
  \small Automaton for $\mathcal{P}_{111}$ $G(\mathcal{A}) = \IMG(z^2+i)$
\end{minipage}
\hfill
\begin{minipage}{0.45\textwidth}
    \centering
    \begin{tikzpicture}[shorten >=1pt,node distance=3cm,on grid,>={Stealth[round]},
        every state/.style={draw=blue!50,very thick,fill=blue!20}, bend angle = 15]
    
      \node[state] (1)  {$1$};
      \node[state] (a) [right=of 1] {$a$};
      \node[state] (b) [below =of 1] {$b$};
      \node[state] (c) [below =of a] {$c$};
    
      \path[->] 
      (1) edge [loop above] node  {x/x} ()       
      (a)  edge [loop above] node  {1/2} ()
      (a)  edge node [right] {2/1} (b)
      (b)  edge [bend right] node [below] {1/1} (c)
      (b)  edge node [above right] {2/2} (1)
      (c)  edge node [right] {1/1} (a)
      (c)  edge [bend right] node [above] {2/2} (b);
    \end{tikzpicture}

  \small Automaton for $\mathcal{P}_{112}$
\end{minipage}

\bigskip

\begin{minipage}{0.45\textwidth}
    \centering
    \begin{tikzpicture}[shorten >=1pt,node distance=3cm,on grid,>={Stealth[round]},
        every state/.style={draw=blue!50,very thick,fill=blue!20}, bend angle = 15]
    
      \node[state] (1)  {$1$};
      \node[state] (a) [right=of 1] {$a$};
      \node[state] (b) [below =of 1] {$b$};
      \node[state] (c) [below =of a] {$c$};
    
      \path[->] 
      (1) edge [loop above] node  {x/x} ()       
      (a)  edge node [above] {x/x+1} (1)
      (b)  edge node [above] {1/2} (a)
      (b)  edge [bend right] node [below] {2/1} (c)
      (c)  edge [loop above] node [above] {2/2} ()
      (c)  edge [bend right] node [above] {1/1} (b);
    \end{tikzpicture}
    
  \small Automaton for $\mathcal{P}_{113}$
\end{minipage}
\hfill
\begin{minipage}{0.45\textwidth}
    \centering
    \begin{tikzpicture}[shorten >=1pt,node distance=3cm,on grid,>={Stealth[round]},
        every state/.style={draw=blue!50,very thick,fill=blue!20}, bend angle = 15]
    
      \node[state] (1)  {$1$};
      \node[state] (a) [right=of 1] {$a$};
      \node[state] (b) [below =of 1] {$b$};
      \node[state] (c) [below =of a] {$c$};
    
      \path[->] 
      (1) edge [loop above] node  {x/x} ()       
      (a)  edge node [above] {x/x+1} (1)
      (b)  edge node [above] {1/1} (a)
      (b)  edge [loop above] node [above] {2/2} ()
      (c)  edge [loop above] node [above] {2/1} ()
      (c)  edge node [above] {1/2} (b);
    \end{tikzpicture}

  \small Automaton for $\mathcal{P}_{114}$
\end{minipage}

\bigskip

\begin{minipage}{0.45\textwidth}
    \centering
    \begin{tikzpicture}[shorten >=1pt,node distance=3cm,on grid,>={Stealth[round]},
        every state/.style={draw=blue!50,very thick,fill=blue!20}, bend angle = 15]
    
      \node[state] (1)  {$1$};
      \node[state] (a) [right=of 1] {$a$};
      \node[state] (b) [below =of a] {$b$};
      \node[state] (c) [below =of 1] {$c$};
    
      \path[->] 
      (1) edge [loop above] node  {x/x} ()       
      (a)  edge node [above] {x/x+1} (1)
      (b)  edge node [right] {1/1} (a)
      (b)  edge [bend left] node [below] {2/2} (c)
      (c)  edge node [left] {1/1} (1)
      (c)  edge [bend left] node [above] {2/2} (b);
    \end{tikzpicture}
    
  \small Automaton for $\mathcal{P}_{121}$ ($G(\mathcal{A})$ is the third Grigorchuk group)
\end{minipage}
\hfill
\begin{minipage}{0.45\textwidth}
    \centering
    \begin{tikzpicture}[shorten >=1pt,node distance=3cm,on grid,>={Stealth[round]},
        every state/.style={draw=blue!50,very thick,fill=blue!20}, bend angle = 15]
    
      \node[state] (1)  {$1$};
      \node[state] (a) [right=of 1] {$a$};
      \node[state] (b) [below =of a] {$b$};
      \node[state] (c) [below =of 1] {$c$};
    
      \path[->] 
      (1) edge [loop above] node  {x/x} ()       
      (a) edge [loop above] node  {2/1} ()       
      (a)  edge node [above] {1/2} (c)
      (b)  edge [bend left] node [below] {1/1} (c)
      (b)  edge node [right] {2/2} (a)
      (c)  edge node [left] {1/1} (1)
      (c)  edge [bend left] node [above] {2/2} (b);
    \end{tikzpicture}

  \small Automaton for $\mathcal{P}_{122}$
\end{minipage}

\bigskip

\begin{minipage}{0.45\textwidth}
    \centering
\begin{tikzpicture}[shorten >=1pt,node distance=3cm,on grid,>={Stealth[round]},
    every state/.style={draw=blue!50,very thick,fill=blue!20}, bend angle = 15]

  \node[state] (a)  {$a$};
  \node[state] (b) [right=of a] {$b$};
  \node[state] (c) [below =of a] {$c$};
  \node[state] (d) [below =of b] {$d$};

  \path[->] 
  (a) edge [loop above] node  {1/2} ()       
  (a)  edge node [left] {2/1} (c)
  (b)  edge node [right] {x/x} (d)
  (c)  edge [bend right] node [below] {1/1} (d)
  (c)  edge [bend right] node [left] {2/2} (b)
  (d)  edge [bend right] node [above] {1/1} (a)
  (d)  edge [bend right] node [above] {2/2} (c);
\end{tikzpicture}
    
  \small Automaton for $\mathcal{P}_{123}$ ($G_{\mathcal{P}_{123}}$ is the closure of the first Grigorchuk group)
\end{minipage}
\hfill
\begin{minipage}{0.45\textwidth}
    \centering
    \begin{tikzpicture}[shorten >=1pt,node distance=3cm,on grid,>={Stealth[round]},
        every state/.style={draw=blue!50,very thick,fill=blue!20}, bend angle = 15]
    
      \node[state] (1)  {$1$};
      \node[state] (a) [right=of 1] {$a$};
      \node[state] (b) [below =of a] {$b$};
      \node[state] (c) [below =of 1] {$c$};
    
      \path[->] 
      (1) edge [loop above] node  {x/x} ()       
      (a) edge node [above]  {x/x+1} (1)
      (b) edge [loop below] node [below] {1/1} ()
      (b) edge node [right] {2/2} (a)
      (c) edge [loop above] node  {2/1} ()
      (c) edge node [above] {1/2} (b);
    \end{tikzpicture}

  \small Automaton for $\mathcal{P}_{124}$
\end{minipage}

\bigskip

\begin{minipage}{0.45\textwidth}
    \centering
    \begin{tikzpicture}[scale=0.8, transform shape, shorten >=1pt,node distance=3cm,on grid,>={Stealth[round]},
        every state/.style={draw=blue!50,very thick,fill=blue!20}, bend angle = 15]
    
      \node[state] (1)  {$1$};
      \node[state] (a) [below right=of 1] {$a$};
      \node[state] (b) [below = of a] {$b$};
      \node[state] (d) [below left =of 1] {$d$};  
      \node[state] (c) [below =of d] {$c$};
      
      \path[->] 
      (1) edge [loop above] node  {x/x} ()       
      (a) edge node [right]  {1/2} (1)
      (a) edge node [right]  {2/2} (b)
      (b) edge  node [above]  {x/x} (c)
      (c) edge  node [above] {1/1} (a)
      (c) edge [bend left] node [left] {2/2} (d)
      (d) edge [bend left] node [right] {1/1} (c)
      (d)  edge node [left] {2/2} (1);
    \end{tikzpicture}
    
  \small Automaton for $\mathcal{P}_{211}$
\end{minipage}
\hfill
\begin{minipage}{0.45\textwidth}
    \centering
    \begin{tikzpicture}[scale=0.8, transform shape, shorten >=1pt,node distance=3cm,on grid,>={Stealth[round]},
    every state/.style={draw=blue!50,very thick,fill=blue!20}, bend angle = 15]

  \node[state] (1)  {$1$};
  \node[state] (c) [below right=of 1] {$c$};
  \node[state] (a) [below = of a] {$a$};
  \node[state] (d) [below left =of 1] {$d$};  
  \node[state] (b) [below =of d] {$b$};
  
  \path[->] 
  (1) edge [loop above] node  {x/x} ()       
  (a) edge node [right]  {1/2} (c)
  (a) edge [loop below] node [right] {2/1} () 
  (b) edge  node [above]  {1/1} (a)
  (b) edge [bend left] node [left] {2/2} (d)
  (c) edge  node  {1/1} (1)
  (c)  edge node [above] {2/2} (b)
  (d) edge [bend left] node [right] {1/1} (b)
  (d)  edge node [above] {2/2} (1);
\end{tikzpicture}

  \small Automaton for $\mathcal{P}_{212}$
\end{minipage}

\bigskip

\begin{minipage}{0.45\textwidth}
    \centering
    \begin{tikzpicture}[scale=0.8, transform shape, shorten >=1pt,node distance=3cm,on grid,>={Stealth[round]},
        every state/.style={draw=blue!50,very thick,fill=blue!20}, bend angle = 15]
    
      \node[state] (c)  {$c$};
      \node[state] (a) [above left=of c] {$a$};
      \node[state] (b) [below left= of c] {$b$};
      \node[state] (d) [above right =of c] {$d$};  
      \node[state] (e) [below right=of c] {$e$};
      
      \path[->] 
      (a) edge [loop above] node  {1/2} ()    
      (a) edge node [above]  {2/1} (c)
      (b) edge node [left] {1/1} (a) 
      (b) edge [bend right] node [below] {2/2} (e)
      (c) edge  node [above]  {1/2} (b)
      (c) edge node [above] {2/1} (d)
      (d) edge node [above] {x/x+1} (a)
      (e) edge [bend right] node [above] {2/1} (b)
      (e) edge node [right] {1/2} (d);
    \end{tikzpicture}
    
  \small Automaton for $\mathcal{P}_{213}$
\end{minipage}
\hfill
\begin{minipage}{0.45\textwidth}
    \centering
    \begin{tikzpicture}[scale=0.8, transform shape, shorten >=1pt,node distance=3cm,on grid,>={Stealth[round]},
        every state/.style={draw=blue!50,very thick,fill=blue!20}, bend angle = 15]
    
      \node[state] (1)  {$1$};
      \node[state] (a) [below right=of 1] {$a$};
      \node[state] (c) [below = of a] {$c$};
      \node[state] (b) [below left =of 1] {$b$};  
      \node[state] (d) [below =of b] {$d$};
      
      \path[->] 
      (1) edge [loop above] node  {x/x} ()       
      (a) edge node [right]  {1/2} (1)
      (a) edge node [above] {2/1} (b) 
      (b) edge node [above]  {x/x} (c)
      (c) edge [right] node  {1/1} (a)
      (c) edge [bend left] node [below] {2/2} (d)
      (d) edge [bend left] node [above] {1/1} (c)
      (d)  edge node [left] {2/2} (b);
    \end{tikzpicture}

  \small Automaton for $\mathcal{P}_{214}$
\end{minipage}

\bigskip

\begin{minipage}{0.45\textwidth}
    \centering
    \begin{tikzpicture}[scale=0.75, transform shape, shorten >=1pt,node distance=3cm,on grid,>={Stealth[round]},
        every state/.style={draw=blue!50,very thick,fill=blue!20}, bend angle = 15]
    
      \node[state] (1)  {$1$};
      \node[state] (a) [below right=of 1] {$a$};
      \node[state] (d) [below = of a] {$d$};
      \node[state] (c) [below left =of 1] {$c$};  
      \node[state] (b) [below =of c] {$b$};
      
      \path[->] 
      (1) edge [loop above] node  {x/x} ()       
      (a) edge node [right]  {1/2} (1)
      (a) edge node [right] {2/1} (d) 
      (b) edge [bend left] node [left]  {1/1} (c)
      (b) edge node [above]  {2/2} (a)
      (c) edge [right] node  {1/1} (1)
      (c) edge [bend left] node [right] {2/2} (b)
      (d) edge node [above] {x/x} (b);
    \end{tikzpicture}

  \small Automaton for $\mathcal{P}_{221}$
\end{minipage}
\hfill
\begin{minipage}{0.45\textwidth}
    \centering
    \begin{tikzpicture}[scale=0.75, transform shape, shorten >=1pt,node distance=3cm,on grid,>={Stealth[round]},
        every state/.style={draw=blue!50,very thick,fill=blue!20}, bend angle = 15]
    
      \node[state] (1)  {$1$};
      \node[state] (c) [below right=of 1] {$c$};
      \node[state] (b) [below = of c] {$b$};
      \node[state] (d) [below left =of 1] {$d$};  
      \node[state] (a) [below =of d] {$a$};
      
      \path[->] 
      (1) edge [loop above] node  {x/x} ()       
      (a) edge [loop below] node [below]  {1/2} ()
      (a) edge node [right] {2/1} (d) 
      (b) edge [bend left] node [left]  {1/1} (c)
      (b) edge node [above]  {2/2} (a)
      (c) edge [right] node  {1/1} (1)
      (c) edge [bend left] node [right] {2/2} (b)
      (d) edge node [above] {1/1} (b)
      (d) edge node [left] {2/2} (1);
    \end{tikzpicture}

  \small Automaton for $\mathcal{P}_{222}$
\end{minipage}

\bigskip

\begin{minipage}{0.45\textwidth}
    \centering
    \begin{tikzpicture}[scale=0.75, transform shape, shorten >=1pt,node distance=3cm,on grid,>={Stealth[round]},
        every state/.style={draw=blue!50,very thick,fill=blue!20}, bend angle = 15]
    
      \node[state] (a)  {$a$};
      \node[state] (b) [below right=of 1] {$b$};
      \node[state] (e) [below = of c] {$e$};
      \node[state] (c) [below left =of 1] {$c$};  
      \node[state] (d) [below =of c] {$d$};
      
      \path[->] 
      (a) edge [loop above] node  {1/2} ()       
      (a) edge node [right]  {2/1} (b)
      (b) edge node [above] {1/2} (c) 
      (b) edge [bend left] node [right]  {2/1} (e)
      (c) edge node [left]  {1/1} (a)
      (c) edge [bend left] node [right] {2/2} (d)
      (d) edge node [above] {1/2} (e)
      (d) edge [bend left] node [left] {2/1} (c)
      (e) edge [bend left] node [left] {x/x+1} (b);
    \end{tikzpicture}

  \small Automaton for $\mathcal{P}_{223}$
\end{minipage}
\hfill
\begin{minipage}{0.45\textwidth}
    \centering
    \begin{tikzpicture}[scale=0.75, transform shape, shorten >=1pt,node distance=3cm,on grid,>={Stealth[round]},
        every state/.style={draw=blue!50,very thick,fill=blue!20}, bend angle = 15]
    
      \node[state] (1)  {$1$};
      \node[state] (a) [below right=of 1] {$a$};
      \node[state] (b) [below left =of 1] {$b$};  
      \node[state] (d) [below = of b] {$d$};
      \node[state] (c) [below =of a] {$c$};
      
      \path[->] 
      (1) edge [loop above] node  {x/x} ()       
      (a) edge node [right]  {1/2} (1)
      (a) edge node [above]  {2/1} (b)
      (b) edge node [above] {x/x} (c) 
      (c) edge node [right]  {1/1} (a)
      (c) edge [bend left] node [below] {2/2} (d)
      (d) edge node [left] {1/1} (b)
      (d) edge [bend left] node [above] {2/2} (c);
    \end{tikzpicture}
    
  \small Automaton for $\mathcal{P}_{224}$
\end{minipage}

\bigskip

\begin{minipage}{0.45\textwidth}
    \centering
    \begin{tikzpicture}[shorten >=1pt,node distance=3cm,on grid,>={Stealth[round]},
        every state/.style={draw=blue!50,very thick,fill=blue!20}, bend angle = 15]
    
      \node[state] (1)  {$1$};
      \node[state] (c) [right=of 1] {$c$};
      \node[state] (a) [below =of c] {$a$};
      \node[state] (b) [below =of 1] {$b$};
    
      \path[->] 
      (1) edge [loop above] node  {x/x} ()       
      (a) edge node [right]  {x/x+1} (c)
      (b) edge [bend left] node [left] {1/1} (c)
      (b) edge node [below] {2/2} (a)
      (c) edge node [above] {1/1} (1)
      (c) edge [bend left] node [right] {2/2} (b);
    \end{tikzpicture}

  \small Automaton for $\mathcal{P}_{311}$
\end{minipage}
\hfill
\begin{minipage}{0.45\textwidth}
    \centering
    \begin{tikzpicture}[scale=0.75, transform shape, shorten >=1pt,node distance=3cm,on grid,>={Stealth[round]},
        every state/.style={draw=blue!50,very thick,fill=blue!20}, bend angle = 15]
    
      \node[state] (1)  {$1$};
      \node[state] (a) [below right=of 1] {$a$};
      \node[state] (c) [below left =of 1] {$c$};  
      \node[state] (d) [below = of c] {$d$};
      \node[state] (b) [below =of a] {$b$};
      
      \path[->] 
      (1) edge [loop above] node  {x/x} ()       
      (a) edge node [right]  {1/2} (1)
      (a) edge node [right]  {2/1} (b)
      (b) edge node [above] {1/2} (d)
      (b) edge [loop below] node [below] {2/1} ()
      (c) edge node [left]  {1/1} (1)
      (c) edge [bend left] node [right] {2/2} (d)
      (d) edge [bend left] node [left] {1/1} (c)
      (d) edge  node [above] {2/2} (a);
    \end{tikzpicture}
  \small Automaton for $\mathcal{P}_{312}$
\end{minipage}

\bigskip

\begin{minipage}{0.45\textwidth}
    \centering
    \begin{tikzpicture}[scale=0.8, transform shape, shorten >=1pt,node distance=3cm,on grid,>={Stealth[round]},
        every state/.style={draw=blue!50,very thick,fill=blue!20}, bend angle = 15]
    
      \node[state] (1)  {$1$};
      \node[state] (b) [below right=of 1] {$b$};
      \node[state] (a) [below left =of 1] {$a$};  
      \node[state] (d) [below = of b] {$d$};
      \node[state] (c) [below =of a] {$c$};
      
      \path[->] 
      (1) edge [loop above] node  {x/x} ()       
      (a) edge node [left]  {1/2} (1)
      (a) edge node [above]  {2/1} (b)
      (b) edge node [right] {1/2} (1)
      (b) edge node [below] {2/1} (c)
      (c) edge node [below]  {1/2} (d)
      (c) edge node [left] {2/1} (a)
      (d) edge [loop above] node [above] {x/x+1} ();
    \end{tikzpicture}

  \small Automaton for $\mathcal{P}_{313}$
\end{minipage}
\hfill
\begin{minipage}{0.45\textwidth}
    \centering
    \begin{tikzpicture}[shorten >=1pt,node distance=3cm,on grid,>={Stealth[round]},
        every state/.style={draw=blue!50,very thick,fill=blue!20}, bend angle = 15]
    
      \node[state] (1)  {$1$};
      \node[state] (a) [right=of 1] {$a$};
      \node[state] (b) [below =of a] {$b$};
      \node[state] (c) [below =of 1] {$c$};
    
      \path[->] 
      (1) edge [loop above] node  {x/x} ()       
      (a) edge node [above]  {1/2} (1)
      (a) edge [bend left] node [right]  {2/1} (b)
      (b) edge [bend left] node [left]  {1/2} (a)
      (b) edge node [below] {2/1} (c)
      (c) edge node [left] {x/x+1} (a);
    \end{tikzpicture}
  \small Automaton for $\mathcal{P}_{314}$
\end{minipage}

\bigskip

\begin{minipage}{0.45\textwidth}
    \centering
    \begin{tikzpicture}[shorten >=1pt,node distance=3cm,on grid,>={Stealth[round]},
        every state/.style={draw=blue!50,very thick,fill=blue!20}, bend angle = 15]
    
      \node[state] (1)  {$1$};
      \node[state] (a) [right=of 1] {$a$};
      \node[state] (b) [below =of 1] {$b$};
      \node[state] (c) [below =of a] {$c$};
    
      \path[->] 
      (1) edge [loop above] node  {x/x} ()       
      (a) edge node [above]  {1/1} (1)
      (a) edge [bend right] node [left] {2/2} (b)
      (b) edge node [below]  {1/1} (c)
      (b) edge [bend right] node [right] {2/2} (a)
      (c) edge  node [right] {x/x+1} (a);
    \end{tikzpicture}

  \small Automaton for $\mathcal{P}_{321}$
\end{minipage}
\hfill
\begin{minipage}{0.45\textwidth}
    \centering
    \begin{tikzpicture}[scale=0.8, transform shape, shorten >=1pt,node distance=3cm,on grid,>={Stealth[round]},
        every state/.style={draw=blue!50,very thick,fill=blue!20}, bend angle = 15]
    
      \node[state] (1)  {$1$};
      \node[state] (b) [below right=of 1] {$b$};
      \node[state] (a) [below left =of 1] {$a$};  
      \node[state] (c) [below = of b] {$c$};
      \node[state] (d) [below =of a] {$d$};
      
      \path[->] 
      (1) edge [loop above] node  {x/x} ()       
      (a) edge node [left]  {1/2} (1)
      (a) edge node [above]  {2/1} (b)
      (b) edge node [right] {1/2} (1)
      (b) edge node [right] {2/1} (c)
      (c) edge node [right] {1/2} (a)
      (c) edge node [below] {2/1} (d)
      (d) edge node [left] {x/x+1} (a);
    \end{tikzpicture}
  \small Automaton for $\mathcal{P}_{322}$
\end{minipage}

\bigskip

\begin{minipage}{0.45\textwidth}
    \centering
    \begin{tikzpicture}[scale=0.8, transform shape, shorten >=1pt,node distance=3cm,on grid,>={Stealth[round]},
        every state/.style={draw=blue!50,very thick,fill=blue!20}, bend angle = 15]
    
      \node[state] (1)  {$1$};
      \node[state] (b) [below right=of 1] {$b$};
      \node[state] (a) [below left =of 1] {$a$};  
      \node[state] (c) [below = of b] {$c$};
      \node[state] (d) [below =of a] {$d$};
      
      \path[->] 
      (1) edge [loop above] node  {x/x} ()       
      (a) edge node [left]  {1/2} (1)
      (a) edge node [above]  {2/1} (b)
      (b) edge node [right] {1/2} (c)
      (b) edge node [right] {2/1} (1)
      (c) edge node [below]  {1/2} (d)
      (c) edge node [below] {2/1} (a)
      (d) edge [loop above] node [left] {x/x+1} ();
    \end{tikzpicture}

  \small Automaton for $\mathcal{P}_{323}$
\end{minipage}
\hfill
\begin{minipage}{0.45\textwidth}
    \centering
    \begin{tikzpicture}[shorten >=1pt,node distance=3cm,on grid,>={Stealth[round]},
        every state/.style={draw=blue!50,very thick,fill=blue!20}, bend angle = 15]
    
      \node[state] (1)  {$1$};
      \node[state] (a) [right=of 1] {$a$};
      \node[state] (b) [below =of a] {$b$};
      \node[state] (c) [below =of 1] {$c$};
    
      \path[->] 
      (1) edge [loop above] node  {x/x} ()       
      (a) edge  node [above] {1/2} (1)
      (a) edge [bend left] node [right] {2/1} (b)
      (b) edge node [below]  {1/2} (c)
      (b) edge [bend left] node [left] {2/1} (a)
      (c) edge node [left] {x/x+1} (a);
    \end{tikzpicture}
  \small Automaton for $\mathcal{P}_{324}$
\end{minipage}

\bigskip

\begin{minipage}{0.45\textwidth}
    \centering
    \begin{tikzpicture}[scale=0.8, transform shape, shorten >=1pt,node distance=3cm,on grid,>={Stealth[round]},
        every state/.style={draw=blue!50,very thick,fill=blue!20}, bend angle = 15]
    
      \node[state] (1)  {$1$};
      \node[state] (d) [below right=of 1] {$d$};
      \node[state] (c) [below left =of 1] {$c$};  
      \node[state] (a) [below = of d] {$a$};
      \node[state] (b) [below =of c] {$b$};
      
      \path[->] 
      (1) edge [loop above] node  {x/x} ()       
      (a) edge [bend right] node [above] {1/2} (c)
      (a) edge node [right]  {2/1} (d)
      (b) edge [bend left] node [left] {1/1} (c)
      (b) edge node [below] {2/2} (a)
      (c) edge node [left]  {1/1} (1)
      (c) edge [bend left] node [right] {2/2} (b)
      (d) edge [bend left] node [below] {1/1} (b)
      (d) edge node [right] {2/2} (1);
    \end{tikzpicture}

  \small Automaton for $\mathcal{P}_{411}$
\end{minipage}
\hfill
\begin{minipage}{0.45\textwidth}
    \centering
    \begin{tikzpicture}[shorten >=1pt,node distance=3cm,on grid,>={Stealth[round]},
        every state/.style={draw=blue!50,very thick,fill=blue!20}, bend angle = 15]
    
      \node[state] (b)  {$b$};
      \node[state] (a) [above left=of b] {$a$};
      \node[state] (c) [below left =of b] {$c$};  
      \node[state] (e) [above right= of b] {$e$};
      \node[state] (d) [below right=of b] {$d$};
      
      \path[->] 
      (a) edge node [left] {1/2} (c)       
      (a) edge node [above] {2/1} (e)
      (b) edge [bend left] node [right] {1/1} (c)
      (b) edge node [right] {2/2} (a)
      (c) edge node [below]  {1/1} (d)
      (c) edge [bend left] node [left] {2/2} (b)
      (d) edge node [right] {x/x} (b)
      (e) edge node [above] {1/1} (b)
      (e) edge node [right] {2/2} (d);
    \end{tikzpicture}
  \small Automaton for $\mathcal{P}_{414}$
\end{minipage}

\bigskip

\begin{minipage}{0.45\textwidth}
    \centering
    \begin{tikzpicture}[scale=0.8, transform shape, shorten >=1pt,node distance=3cm,on grid,>={Stealth[round]},
        every state/.style={draw=blue!50,very thick,fill=blue!20}, bend angle = 15]
    
      \node[state] (1)  {$1$};
      \node[state] (d) [below right=of 1] {$d$};
      \node[state] (c) [below left =of 1] {$c$};  
      \node[state] (a) [below = of d] {$a$};
      \node[state] (b) [below =of c] {$b$};
      
      \path[->] 
      (1) edge [loop above] node  {x/x} ()       
      (a) edge [bend right] node [above] {1/2} (c)
      (a) edge node [right]  {2/1} (d)
      (b) edge node [below] {1/1} (a)
      (b) edge [bend left] node [left] {2/2} (c)
      (c) edge node [left]  {1/1} (1)
      (c) edge [bend left] node [right] {2/2} (b)
      (d) edge [bend left] node [below] {1/1} (b)
      (d) edge node [right] {2/2} (1);
    \end{tikzpicture}

  \small Automaton for $\mathcal{P}_{421}$
\end{minipage}
\hfill
\begin{minipage}{0.45\textwidth}
    \centering
    \begin{tikzpicture}[scale=0.8, transform shape, shorten >=1pt,node distance=3cm,on grid,>={Stealth[round]},
        every state/.style={draw=blue!50,very thick,fill=blue!20}, bend angle = 15]
    
      \node[state] (1)  {$1$};
      \node[state] (a) [below right=of 1] {$a$};
      \node[state] (c) [below left =of 1] {$c$};  
      \node[state] (d) [below = of a] {$d$};
      \node[state] (b) [below =of c] {$b$};
      
      \path[->] 
      (1) edge [loop above] node  {x/x} ()       
      (a) edge [bend right] node [right] {1/2} (1)
      (a) edge node [right]  {2/1} (d)
      (b) edge node [right] {1/1} (a)
      (b) edge [bend left] node [left] {2/2} (c)
      (c) edge node [left]  {1/1} (1)
      (c) edge [bend left] node [right] {2/2} (b)
      (d) edge [bend left] node [below] {1/2} (b)
      (d) edge [loop below] node [below] {2/1} ();
    \end{tikzpicture}
  \small Automaton for $\mathcal{P}_{422}$
\end{minipage}

\bigskip

\begin{minipage}{0.45\textwidth}
    \centering
    \begin{tikzpicture}[scale=0.8, transform shape, shorten >=1pt,node distance=3cm,on grid,>={Stealth[round]},
        every state/.style={draw=blue!50,very thick,fill=blue!20}, bend angle = 15]
    
      \node[state] (a)  {$a$};
      \node[state] (b) [below right=of 1] {$b$};
      \node[state] (c) [below left =of 1] {$c$};  
      \node[state] (e) [below = of b] {$e$};
      \node[state] (d) [below =of c] {$d$};
      
      \path[->] 
      (a) edge [loop above] node  {1/2} ()       
      (a) edge node [right] {2/1} (b)
      (b) edge node [above] {1/1} (c)
      (b) edge [bend left] node [right] {2/2} (e)
      (c) edge node [left]  {1/2} (a)
      (c) edge node [left] {2/1} (d)
      (d) edge node [below] {x/x} (b)
      (e) edge [bend left] node [left] {1/1} (b)
      (e) edge node [below]  {2/2} (d);
    \end{tikzpicture}

  \small Automaton for $\mathcal{P}_{423}$
\end{minipage}
\hfill
\begin{minipage}{0.45\textwidth}
    \centering
\begin{tikzpicture}[shorten >=1pt,node distance=3cm,on grid,>={Stealth[round]},
    every state/.style={draw=blue!50,very thick,fill=blue!20}, bend angle = 15]

  \node[state] (b)  {$b$};
  \node[state] (a) [above left=of b] {$a$};
  \node[state] (c) [below left =of b] {$c$};  
  \node[state] (e) [above right= of b] {$e$};
  \node[state] (d) [below right=of b] {$d$};
  
  \path[->] 
  (a) edge node [left] {1/2} (c)       
  (a) edge node [above] {2/1} (e)
  (b) edge [bend left] node [right] {1/1} (c)
  (b) edge node [right] {2/2} (a)
  (c) edge [bend left] node [above]  {1/1} (b)
  (c) edge node [below] {2/2} (d)
  (d) edge node [right] {x/x} (b)
  (e) edge node [right] {1/1} (d)
  (e) edge node [right]  {2/2} (b);
\end{tikzpicture}
  \small Automaton for $\mathcal{P}_{424}$
\end{minipage}

\subsection{List of automata groups for the case $(d,D) = (3,2)$}

In this case, we have 2 infinite topologically finitely generated groups of finite type (see \cref{subsection: the case (dD) = (32)}) but they are not $p$-groups. One of them was recognized in \cref{proposition: depth Hanoi towers closure} to be the Hanoi towers group, whose automaton can be seen in the \cref{figure: Hanoi towers automaton}. 

\bigskip

\begin{figure}[h!]
\centering
\begin{tikzpicture}[shorten >=1pt,node distance=3cm,on grid,>={Stealth[round]},
    every state/.style={draw=blue!50,very thick,fill=blue!20}, bend angle = 15]

  \node[state] (1)  {$1$};
  \node[state] (a) [above left=of 1] {$a$};
  \node[state] (b) [above right =of 1] {$b$};  
  \node[state] (c) [below left= of 1] {$c$};
  
  \path[->] 
  (1) edge [loop right] node  {x/x} ()
  (a) edge [loop left] node  {1/1} ()
  (b) edge [loop right] node  {2/2} ()
  (c) edge [loop left] node  {3/3} ()  
  (a) edge [bend left] node [above] {2/3} (1)
  (a) edge [bend right] node [below] {3/2} (1)
  (b) edge [bend right] node [above] {1/3} (1)
  (b) edge [bend left] node [below] {3/1} (1)
  (c) edge [bend left] node [above] {1/2} (1)
  (c) edge [bend right] node [below] {2/1} (1);
\end{tikzpicture}
  \caption{Automaton of the Hanoi towers group}
  \label{figure: Hanoi towers automaton}
\end{figure}
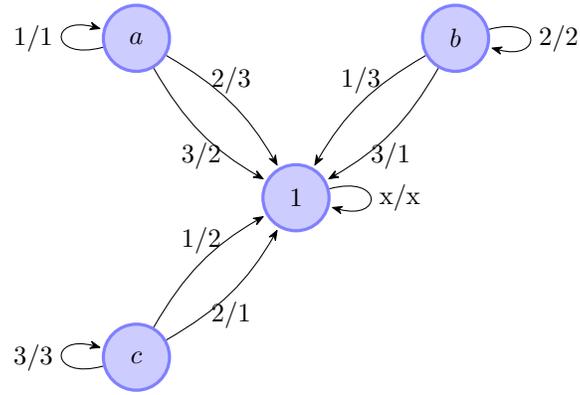

\subsection{List of automata groups for the case $(d,D) = (3,3)$}

In this case, we have 216 infinite topologically finitely generated groups of finite type (see \cref{subsection: the case (dD) = (33)}) but we will only work with the 12 groups of finite type given in this article explicitly. 9 of them are obtained as the closure of a 3-state automaton whereas the remaining 3 are given by a 4-state automaton. 

\bigskip

\begin{minipage}{0.45\textwidth}
    \centering
    \begin{tikzpicture}[scale=0.9, transform shape, shorten >=1pt,node distance=3cm,on grid,>={Stealth[round]},
        every state/.style={draw=blue!50,very thick,fill=blue!20}, bend angle = 15]
    
      \node[state] (1)  {$1$};
      \node[state] (a) [right=of 1] {$a$};
      \node[state] (b) [below =of 1] {$b$};  
      \node[state] (c) [below = of a] {$c$};
      
      \path[->] 
      (1) edge [loop left] node  {x/x} ()       
      (a) edge node [above] {x/x+1} (1)
      (b) edge node [left] {1/1} (1)
      (b) edge [bend right] node [below] {2/2} (c)
      (b) edge [loop left] node  {3/3} ()       
      (c) edge node [right]  {1/2} (a)
      (c) edge [bend right] node [above] {2/3} (b)
      (c) edge [loop right] node  {3/1} ();
    \end{tikzpicture}

  \small Automaton for $\mathcal{P}_{1,1}$
\end{minipage}
\hfill
\begin{minipage}{0.45\textwidth}
    \centering
    \begin{tikzpicture}[scale=0.9, transform shape, shorten >=1pt,node distance=3cm,on grid,>={Stealth[round]},
        every state/.style={draw=blue!50,very thick,fill=blue!20}, bend angle = 15]
    
      \node[state] (1)  {$1$};
      \node[state] (a) [right=of 1] {$a$};
      \node[state] (c) [below =of 1] {$c$};  
      \node[state] (b) [below = of a] {$b$};
      
      \path[->] 
      (1) edge [loop left] node  {x/x} ()       
      (a) edge node [above] {x/x+1} (1)
      (b) edge node [right]  {1/1} (a)
      (b) edge [bend right] node [above] {2/2} (c)
      (b) edge [loop right] node  {3/3} ()
      (c) edge node [left] {1/2} (1)
      (c) edge [bend right] node [below] {2/3} (b)
      (c) edge [loop left] node  {3/1} ();      
    \end{tikzpicture}
  \small Automaton for $\mathcal{P}_{1,2}$
\end{minipage}

\bigskip

\begin{minipage}{0.45\textwidth}
    \centering
    \begin{tikzpicture}[scale=0.9, transform shape, shorten >=1pt,node distance=3cm,on grid,>={Stealth[round]},
        every state/.style={draw=blue!50,very thick,fill=blue!20}, bend angle = 15]
    
      \node[state] (1)  {$1$};
      \node[state] (a) [right=of 1] {$a$};
      \node[state] (b) [below =of 1] {$b$};  
      \node[state] (c) [below = of a] {$c$};
      
      \path[->] 
      (1) edge [loop left] node {x/x} ()       
      (a) edge node [above] {x/x+1} (1)
      (b) edge node [left] {1/3} (1)
      (b) edge [bend right] node [below] {2/1} (c)
      (b) edge [loop left] node  {3/2} ()       
      (c) edge node [right]  {1/3} (a)
      (c) edge [bend right] node [above] {2/1} (b)
      (c) edge [loop right] node  {3/2} ();
    \end{tikzpicture}

  \small Automaton for $\mathcal{P}_{2,1}$
\end{minipage}
\hfill
\begin{minipage}{0.45\textwidth}
    \centering
    \begin{tikzpicture}[scale=0.9, transform shape, shorten >=1pt,node distance=3cm,on grid,>={Stealth[round]},
        every state/.style={draw=blue!50,very thick,fill=blue!20}, bend angle = 15]
    
      \node[state] (c) {$c$};
      \node[state] (a) [above left = of c] {$a$};
      \node[state] (b) [above right =of c] {$b$};  
      
      \path[->] 
      (a) edge [loop left] node {1/1} ()       
      (a) edge node [above] {2/2} (b)
      (a) edge [bend right] node [below] {3/3} (c)
      (b) edge [bend left] node [right] {x/x+1} (c)
      (c) edge [bend right] node [above] {1/2} (a)
      (c) edge [bend left] node [left] {2/3} (b)
      (c) edge [loop below] node  {3/1} ();      
    \end{tikzpicture}
  \small Automaton for $\mathcal{P}_{3,1}$
\end{minipage}

\bigskip

\begin{minipage}{0.45\textwidth}
    \centering
    \begin{tikzpicture}[scale=0.9, transform shape, shorten >=1pt,node distance=3cm,on grid,>={Stealth[round]},
        every state/.style={draw=blue!50,very thick,fill=blue!20}, bend angle = 15]
    
      \node[state] (b) {$b$};
      \node[state] (1) [above left = of b] {$1$};
      \node[state] (a) [above right =of b] {$a$};  
      
      \path[->] 
      (1) edge [loop above] node {x/x} ()       
      (a) edge node [above] {x/x+1} (1)
      (b) edge node [left] {1/1} (a)
      (b) edge node [right] {2/2} (1)
      (b) edge [loop below] node  {3/3} ();      
    \end{tikzpicture}

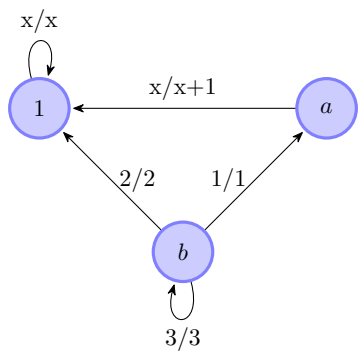
  \small Automaton for $\mathcal{P}_{3,2}$, $G(\mathcal{A}) = G_{(1,0)}$
\end{minipage}
\hfill
\begin{minipage}{0.45\textwidth}
    \centering
    \begin{tikzpicture}[scale=0.9, transform shape, shorten >=1pt,node distance=3cm,on grid,>={Stealth[round]},
        every state/.style={draw=blue!50,very thick,fill=blue!20}, bend angle = 15]
    
      \node[state] (b) {$b$};
      \node[state] (1) [above left = of b] {$1$};
      \node[state] (a) [above right =of b] {$a$};  
      
      \path[->] 
      (1) edge [loop above] node {x/x} ()     
      (a) edge node [above] {1/2} (1)
      (a) edge [loop above] node {3/1} ()
      (a) edge [bend right] node [left] {2/3} (b)
      (b) edge node [left] {1/1} (1)
      (b) edge [loop below] node  {2/2} ()
      (b) edge [bend right] node [right] {3/3} (a);
    \end{tikzpicture}
  \small Automaton for $\mathcal{P}_{4,1}$
\end{minipage}

\bigskip

\begin{minipage}{0.45\textwidth}
    \centering
    \vspace{3em}
    \begin{tikzpicture}[scale=0.9, transform shape, shorten >=1pt,node distance=3cm,on grid,>={Stealth[round]},
        every state/.style={draw=blue!50,very thick,fill=blue!20}, bend angle = 15]
    
      \node[state] (a) {$a$};
      \node[state] (b) [above left = of a] {$b$};
      \node[state] (c) [above right =of a] {$c$};  
      
      \path[->] 
      (a) edge [loop below] node {1/2} ()
      (a) edge [bend right] node [right] {2/3} (b)
      (a) edge [bend left] node [left] {3/1} (c)
      (b) edge [bend right] node [left] {x/x+1} (a)
      (c) edge [bend left] node [right] {x/x-1} (a);
    \end{tikzpicture}

  \small Automaton for $\mathcal{P}_{4,2}$
\end{minipage}
\hfill
\begin{minipage}{0.45\textwidth}
    \centering
    \begin{tikzpicture}[scale=0.9, transform shape, shorten >=1pt,node distance=3cm,on grid,>={Stealth[round]},
        every state/.style={draw=blue!50,very thick,fill=blue!20}, bend angle = 15]
    
      \node[state] (c) {$c$};
      \node[state] (a) [above left = of c] {$a$};
      \node[state] (b) [above right =of c] {$b$};  
      
      \path[->] 
      (a) edge [loop above] node {1/2} ()       
      (a) edge node [above] {3/1} (b)
      (a) edge [bend right] node [left] {2/3} (c)
      (b) edge [bend left] node [right] {x/x+1} (c)
      (c) edge [bend right] node [right] {1/3} (a)
      (c) edge [bend left] node [left] {3/2} (b)
      (c) edge [loop below] node {2/1} ();
    \end{tikzpicture}
  \small Automaton for $\mathcal{P}_{5,1}$
\end{minipage}

\bigskip

\begin{minipage}{0.45\textwidth}
    \centering
    \begin{tikzpicture}[scale=0.9, transform shape, shorten >=1pt,node distance=3cm,on grid,>={Stealth[round]},
        every state/.style={draw=blue!50,very thick,fill=blue!20}, bend angle = 15]
    
      \node[state] (b) {$b$};
      \node[state] (1) [above left = of b] {$1$};
      \node[state] (a) [above right =of b] {$a$};  
      
      \path[->] 
      (1) edge [loop above] node {x/x} ()       
      (a) edge node [above] {x/x+1} (1)
      (b) edge node [below] {1/2} (1)
      (b) edge node [below] {3/1} (a)
      (b) edge [loop below] node  {2/3} ();      
    \end{tikzpicture}

  \small Automaton for $\mathcal{P}_{5,2}$
\end{minipage}
\hfill
\begin{minipage}{0.45\textwidth}
    \centering
    \begin{tikzpicture}[scale=0.9, transform shape, shorten >=1pt,node distance=3cm,on grid,>={Stealth[round]},
        every state/.style={draw=blue!50,very thick,fill=blue!20}, bend angle = 15]
    
      \node[state] (a) {$a$};
      \node[state] (b) [above left = of a] {$b$};
      \node[state] (c) [above right =of a] {$c$};  
      
      \path[->] 
      (a) edge [loop below] node {1/1} ()       
      (a) edge [bend left] node [left] {2/2} (b)
      (a) edge [bend left] node [left] {3/3} (c)
      (b) edge [loop above] node {2/3} ()       
      (b) edge [bend left] node [right] {1/2} (a)
      (b) edge node [above] {3/1} (c)
      (c) edge [bend left] node [right] {x/x-1} (a);
    \end{tikzpicture}
  \small Automaton for $\mathcal{P}_{6,1}$
\end{minipage}

\bigskip

\begin{minipage}{0.45\textwidth}
    \centering
    \vspace{4em}
    \begin{tikzpicture}[scale=0.9, transform shape, shorten >=1pt,node distance=3cm,on grid,>={Stealth[round]},
        every state/.style={draw=blue!50,very thick,fill=blue!20}, bend angle = 15]
    
      \node[state] (a) {$a$};
      \node[state] (b) [above left = of a] {$b$};
      \node[state] (c) [above right =of a] {$c$};  
      
      \path[->] 
      (a) edge [loop below] node {1/1} ()
      (a) edge [bend right] node [right] {2/2} (b)
      (a) edge [bend left] node [left] {3/3} (c)
      (b) edge [bend right] node [left] {x/x+1} (a)
      (c) edge [bend left] node [right] {x/x-1} (a);
    \end{tikzpicture}

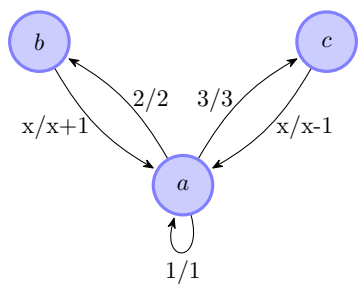
  \small Automaton for $\mathcal{P}_{6,2}$ ($G_{\mathcal{P}_{6,2}}$ is the closure of $G_{(1,2)}$)
\end{minipage}
\hfill
\begin{minipage}{0.45\textwidth}
    \centering
    \begin{tikzpicture}[scale=0.9, transform shape, shorten >=1pt,node distance=3cm,on grid,>={Stealth[round]},
        every state/.style={draw=blue!50,very thick,fill=blue!20}, bend angle = 15]
    
      \node[state] (b) {$b$};
      \node[state] (1) [above left = of b] {$1$};
      \node[state] (a) [above right =of b] {$a$};  
      
      \path[->] 
      (1) edge [loop above] node {x/x} ()       
      (a) edge node [above] {x/x+1} (1)
      (b) edge node [left] {1/3} (1)
      (b) edge node [right] {2/1} (a)
      (b) edge [loop below] node  {3/2} ();      
    \end{tikzpicture}
  \small Automaton for $\mathcal{P}_{6,3}$
\end{minipage}

\bibliographystyle{unsrt}

\end{document}